\documentclass[10pt,reqno]{amsart}
\usepackage{amssymb,mathrsfs,color}

\usepackage{enumerate}
\usepackage{graphicx}
\usepackage{xcolor} 
\usepackage{geometry}
\usepackage{amsfonts,amssymb,amsmath}
\usepackage{slashed}
\usepackage{tikz}
\usetikzlibrary{patterns}

\usepackage{wrapfig}
\usetikzlibrary{arrows,calc,decorations.pathreplacing}
\usepackage{cite}

\definecolor{green}{rgb}{0,0.8,0} 

\renewcommand{\Re}{\mathrm{Re}}
\renewcommand{\Im}{\mathrm{Im}}

\newcommand{\wgamma}{w_{\gamma}}
\newcommand{\wgammap}{w_{\gamma}^\prime}

\newcommand{\ubar}{\bar u}
\newcommand{\alphabar}{\bar\alpha}
\newcommand{\Kbar}{\con K_0}
\newcommand{\Cbar}{\bar C}
\newcommand{\Vo}{V_{\mathrm{O}}}

\definecolor{deepgreen}{cmyk}{1,0,1,0.5}

\newcommand{\chibar}{\underline{\chi}}

\newcommand{\LL}{\mathcal{L}}

\newcommand{\EE}{\mathscr{E}}

\newcommand{\R}{\mathbb{R}}

\newcommand{\Lbar}{\bar L}

\makeatletter

\newcommand{\Rmnum}[1]{\expandafter\@slowromancap\romannumeral #1@}
\makeatother

\newcommand{\taup}{\tau_{+}}
\newcommand{\taum}{\tau_{-}}

\newcommand{\con}{\overline}

\renewcommand{\bar}{\underline}
\newcommand{\phibar}{\overline{\phi}}

\newcommand{\Align}[1]{\begin{align}\begin{split} #1 \end{split}\end{align}}
\newcommand{\Aligns}[1]{\begin{align*}\begin{split} #1 \end{split}\end{align*}}
\newcommand{\EQ}[1]{\begin{equation}\begin{split} #1 \end{split}\end{equation}}

\newcommand{\Del}[1]{}

\newcommand{\sD}{\slashed{D}}
\newcommand{\snab}{\slashed{\nabla}}
\newcommand{\SGamma}{\slashed{\Gamma}}
\renewcommand{\star}{{}^{*}}

\numberwithin{equation}{section}

\newtheorem{theorem}{Theorem}[section]
\newtheorem{corollary}{Corollary}[section]
\newtheorem{lemma}{Lemma}[section]
\newtheorem{proposition}{Proposition}[section]

\theoremstyle{remark}
\newtheorem{remark}{Remark}[section]

\renewcommand\Re{\mathrm{Re}\,}
\renewcommand\Im{\mathrm{Im}\,}

\newcommand{\phithree}{\LL^k\phi}
\newcommand{\rhotwo}{\rho}
\newcommand{\alphatwo}{\alpha}
\newcommand{\alphabartwo}{\alphabar}
\newcommand{\sigmatwo}{\sigma}
\newcommand{\phitwo}{\phi}
\newcommand{\rhobar}{\overline{\rho}}

\begin{document}
\title[Asymptotic properties of small data MKG]{Asymptotic properties of solutions of the Maxwell Klein Gordon equation with small data}
\author{Lydia Bieri}
\author{Shuang Miao}
\author{Sohrab Shahshahani}
\begin{abstract}
We prove peeling estimates for the small data solutions of the Maxwell Klein Gordon equations with non-zero charge and with a non-compactly supported scalar field, in $(3+1)$ dimensions. We obtain the same decay rates as in an earlier work by Lindblad and Sterbenz, but giving a simpler proof. In particular we dispense with the fractional Morawetz estimates for the electromagnetic field, as well as certain space-time estimates. In the case that the scalar field is compactly supported we can avoid fractional Morawetz estimates for the scalar field as well. All of our estimates are carried out using the double null foliation and in a gauge invariant manner.
\end{abstract}
\maketitle
\section{Introduction}
In this work we study the asymptotic behavior of solutions to the Maxwell Klein Gordon (MKG) equations on $\R^{3+1}$ with small initial data. The exact sense in which we take the data to be small will be made precise later. The MKG equations describe an electromagnetic field  represented by a two form $F_{\mu\nu}$ and a {charged scalar field represented by a complex valued function $\phi,$} and are given by
\EQ{\label{MKG}
&\nabla^\mu F_{\mu\nu}=\Im(\con\phi D_\nu\phi)=:J_\nu(\phi)\\
&\nabla^\mu \star F_{\mu\nu}=0\\
&D^\mu D_\mu\phi=0,
}
where $D$ and $F$ are related by
\Aligns{
D_\mu=\nabla_\mu+iA_\mu,\quad\quad\quad F_{\mu\nu}=\nabla_\mu A_\nu-\nabla_\nu A_\mu.
}
Here $\nabla$ is the Levi-Civita covariant differentiation operator and $*$ the Hodge star operator. {In fact the second equation of (\ref{MKG}) is a consequence of the other two and the relation between $F$ and $A.$} From a variational point of view these equations arise as the Euler-Lagrange equations associated to the {Lagrangian
\[\LL_{MKG}:=\frac{1}{4}F_{\mu\nu}F^{\mu\nu}+\frac{1}{2}D_\mu\phi \overline{D^\mu\phi}.\]
The} natural mathematical setting for studying these equations is a complex line bundle over $\R^{3+1}.$ In this setting $\phi$ represents a section of the bundle and $F$ is the curvature form associated with the bundle covariant differentiation $D.$ \\\\
Decay estimates for the null decomposition of the electromagnetic field $F$ (peeling estimates) were first obtained for the linear Maxwell equation in \cite{CK1}. The chargeless MKG was treated by Choquet-Bruhat and Christodoulou in \cite{ChCh1} using the conformal method. Global regularity for MKG was also studied by Eardley and Moncrief in \cite{EardleyMoncrief1,EardleyMoncrief2}. See also \cite{GinVel1,Seg1}. While to the best of our knowledge \cite{LS1} and \cite{Shu2} are the only works where peeling estimates are considered for the MKG equations, such estimates for the massive MKG and the Yang-Mills equation were proved in \cite{Psarelli1} and \cite{Shu1} respectively. The low-regularity MKG problem has also been an active area of recent research, but as we are concerned only with the high regularity scenario we only list the following few works to which we refer the reader for more details and references: \cite{KlaMach1,KeelRoyTao,MachSter1,KriegerSterbenzTataru1,SigmundTesfahun1,R-T}.\\\\
The study of the asymptotic behavior of small data solutions of (\ref{MKG}) with charge was initiated  by Shu in \cite{Shu2}. However, as noted by Lindblad and Sterbenz in \cite{LS1}, Shu provides only a rough outline of a method in the article \cite{Shu2} {and many important details are missing from the argument.}  In \cite{LS1} Lindblad and Sterbenz derived the (slightly modified statements of the) results claimed in \cite{Shu2} using a different argument. Lindblad and Sterbenz introduced fractional Morawetz estimates to handle the contributions coming from the charge, and they considered the case where the scalar field has non-compact support. In addition to the fractional Morawetz estimates, these authors use weighted $L^2L^\infty$ ``Strichartz-type" estimates to control some of the error terms arising in their analysis. The difficulty of handing the charge is reflected in the fact that the ``charged null component" $\rho$ of $F$ enjoys worse decay outside a fixed null cone $C_{-1}$ compared to the decay inside the cone. Our goal in this paper is to obtain the decay estimates derived in \cite{LS1} using the method outlined by Shu. Our main achievement in the non-compactly support case, which allows for a shorter and simpler proof, is to avoid use of fractional Morawetz estimates for the electromagnetic field and the Strichartz-type estimates. In the case where the scalar field has compact support we are also able to dispense with the fractional Morawetz estimates for the scalar field. {Moreover, our method seems more readily adaptable to curved backgrounds, a fact which can be significant for instance when considering the Einstein equations coupled to the MKG equations.}
\subsection{Statement of the result and outline of the strategy} The following theorem is the main result of this work (see Section \ref{Notations} for the definition of the notation).
{\begin{theorem}\label{main theorem}
Suppose that one of the following holds: 
\begin{enumerate}[(i)]
\item $\phi(0,\cdot)$ and $D_t\phi(0,\cdot)$ are supported on $\{r\leq3/4\}$, and that the initial data for the Cauchy problem (\ref{MKG}) satisfy the following smallness conditions
\Align{\label{initial estimates}
&\sum_{\stackrel{\Gamma\in\mathbb{L}}{k\leq7}}\int_{\Sigma_0}(|D_{t,x}\LL^k_\Gamma\phi|^2+|\LL^k_\Gamma\phi|^2)r^2dx\leq\epsilon^2,\\
&\sum_{\stackrel{\Gamma\in\mathbb{L}}{k\leq7}}\int_{\Sigma_0}\left(|\alpha(\LL^k_\Gamma F)|^2+|\sigma(\LL^k_\Gamma F)|^2+|\alphabar(\LL^k_\Gamma F)|^2+r^{-2}|\rho(\LL^k_\Gamma F)|^2\right)r^2dx\leq\epsilon^2,\\
&\sum_{\stackrel{\Gamma\in\mathbb{L}}{k\leq6}}\int_{\Sigma_0}|\rho(\LL_{\Omega_{ij}}\LL^k_\Gamma F)|^2r^2dx\leq\epsilon^2.
}
\item $\phi(0,\cdot)$ and $D_t\phi(0,\cdot)$ are not compactly supported but (\ref{initial estimates}) holds with the following modification for the scalar field energies:
\Align{\label{modified initial estimates}
&\sum_{\stackrel{\Gamma\in\mathbb{L}}{k\leq7}}\int_{\Sigma_0}(|D_{t,x}\LL^k_\Gamma\phi|^2+|\LL^k_\Gamma\phi|^2)r^{2+2\gamma}dx\leq\epsilon^2,\\
}
for some $\gamma\in(0,1/2).$
\end{enumerate}
Then if $\epsilon$ is sufficiently small, the locally defined solution can be extended globally in time and the following decay estimates hold
\begin{alignat}{2}\label{peeling estimates}
&|\phi|\lesssim\epsilon\taup^{-1}\taum^{-1/2},\quad\quad\quad\quad&&|\sD\phi|\lesssim\epsilon\taup^{-2}\taum^{-1/2},\nonumber\\
&|D_L\phi|\lesssim\epsilon\taup^{-2}\taum^{-1/2},\quad\quad\quad\quad&&|D_{\Lbar}\phi|\lesssim\epsilon\taup^{-1}\taum^{-3/2},\nonumber\\
&|\alpha|\lesssim\epsilon\taup^{-5/2},\quad\quad\quad\quad&&|\alphabar|\lesssim\epsilon\taup^{-1}\taum^{-3/2},\\
&|\sigma|\lesssim\epsilon\taup^{-2}\taum^{-1/2},\quad\quad\quad\quad&&|\rho|\lesssim\begin{cases} \epsilon\taup^{-2}\taum^{-1/2} &\mbox{if } u \geq -1 \\
\epsilon r^{-2} & \mbox{if } u\leq-1. \end{cases} \nonumber
\end{alignat}
Moreover in case (ii) above the scalar field enjoys the stronger decay estimates
\begin{alignat}{2}\label{peeling estimates}
&|\phi|\lesssim\epsilon\taup^{-1}\taum^{-1/2-\gamma},\quad\quad\quad\quad&&|\sD\phi|\lesssim\epsilon\taup^{-2}\taum^{-1/2-\gamma},\nonumber\\
&|D_L\phi|\lesssim\epsilon\taup^{-2}\taum^{-1/2-\gamma},\quad\quad\quad\quad&&|D_{\Lbar}\phi|\lesssim\epsilon\taup^{-1}\taum^{-3/2-\gamma}.\nonumber
\end{alignat}
\end{theorem}}
The number of derivatives is not optimal, but regularity is not our concern here. By carefully studying our proof the minimum possible order can be investigated. Note also that in the proof we commute $k$ derivatives with $k\geq7$ and the last derivative that we commute will always be a rotational vector-field. This has to do with using the Poincar\'e inequality to deal with the contribution of the charge, and is explained during the proof. The proof follows the classical energy method used in \cite{CK1}. We use the Morawetz vector field $K_0$ and the time translation vector field $\partial_t$ as multipliers and $\partial_\mu,~{S,}$ and $\Omega_{\mu\nu}$ as commutators (see Subsection \ref{vector fields} for the definition of these vector fields). In implementing the energy method we face two {major} challenges. First, since the equation is not conformally Killing invariant, we will need to bound error terms when using the energy method. It turns out that to obtain optimal decay of the error terms, we need to carefully take into account the structure of the equations to observe certain cancelations. While some of these cancelations were already observed in \cite{Shu2}, there are many more which arise in the terms not considered there. {These cancelations are especially important in view of our use of the {double} null foliation in the error analysis.}\\
The second difficulty {is} due to the non-vanishing of the electric charge (defined below). Specifically, the multiplier $K_0$ contributes boundary terms which, in the presence of a non-zero charge, are not finite. In particular, with {$E_i:=F_{0i},$} it can be seen  from the divergence theorem that if the electric charge
\[e(t)=\frac{1}{4\pi}\lim_{r\rightarrow\infty}\int_{S_{t,r}}E_r\]
is non-zero (say for $t=0$), then the following integral cannot be convergent
\[\int_{\Sigma_0}(1+r^2)|E|^2 dx.\]
In the null decomposition of $F$ this is reflected in the fact that the energy for {$\rho:=\frac{1}{2}F_{L\Lbar}$} on $\Sigma_t$ is infinite. To salvage the situation, we divide $\R^{3+1}$ into the two regions $\{u\geq -1\}$ and $\{u\leq-1\}.$ Inside the null-cone $u=-1$ the space-like hypersurface $t=0$ has finite radius and the integral above is convergent. To bound the energies outside the $u=-1$ cone and the flux along $C_{-1}$ we make crucial use of the Poincar\'e inequality
\[\int_{S_{t,r}}|\rho-\con{\rho}|^2\leq C\sum_{1\leq i<j\leq3}\int_{S_{t,r}}|\LL_{\Omega_{ij}}\rho|^2.\]
The point here is that $\LL_{\Omega_{ij}}F$ is automatically chargeless and therefore by assuming sufficient decay on the initial data for $\LL_{\Omega_{ij}}F$ we are able to obtain the required estimates for the null decomposition of $\LL_{\Omega_{ij}}F.$ On the other hand, due to the finiteness of the charge, the average of $\rho$ on $\Sigma_0$ decays like $r^{-2},$ and by considering the propagation equation for $\rho$ we are able to show that this decay in fact holds uniformly in time. Combining this with the estimates for $\LL_{\Omega_{ij}}\rho$ we are able to prove the decay rate of $r^{-2}$ for $\rho$ itself outside the cone $u=-1$ (compare with $\taup^{-2}\taum^{-1/2}$ inside the cone), which is consistent with the result obtained in \cite{LS1}. {With these observations we are able to complete the proof of our estimates in $\{u\leq-1\}$ when the scalar field is compactly supported, because the equations degenerate to the free Maxwell equations there. However, when the scalar field is not compactly supported we will still have error terms involving $\rho$ (and not $\rhobar$), coming from commuting Lie derivatives with the equation for the scalar field. Here we need to use fractional Morawetz estimates for the scalar field to compensate for the loss of decay in $\rho.$ These estimates are proven under stronger initial decay assumptions on the scalar field. We should note that for the proof of our Morawetz estimates for $\phi$ (both inside and outside the null cone $u=-1$) we use the elegant argument presented by Lindblad and Sterbenz in \cite{LS1}, and that as in \cite{LS1} the double {integral on} the left hand side of the statement of Lemma \ref{fractional Morawetz lemma} is a key ingredient in our analysis in $\{u\leq-1\}.$ Besides the modification of $F$ described above, and careful exploitation of the structure of the equation suitable for analysis using the double null foliation, the other key ingredient which allows us to avoid fractional Morawetz estimates for the electromagnetic field is use of the $L^4$ Sobolev estimates from Lemma \ref{L4 Sobolev lemma}, which are taken from \cite{Shu1}.}
\subsection{Structure of the paper}
In section \ref{Notations} we introduce the notation used in the paper and record some preliminary structural results as well as standard estimates. In particular, in Subsection \ref{Diff Ops} we establish a few commutation relations which will be needed once we commute vector fields with the equations. In Subsection \ref{Energies} we introduce the energies as well as the energy momentum tensor. We then prove (following \cite{LS1}) a Morawetz estimate and derive the structure of the error terms arising in our use of the divergence theorem. In Subsection \ref{Basic Estimates} we state and prove some, mostly standard, decay and Sobolev estimates which are the technical tools for the proof of Theorem \ref{main theorem}.\\
We begin the proof of Theorem \ref{main theorem} in Section \ref{Outside V_T}, where we prove the estimates outside of the null cone $C_{-1}.$ {This section, which is where the contribution of the charge comes into play, is divided into two subsections: In {the} first subsection we consider the case of compactly supported scalar fields. Even though this is a special case of general scalar fields, we have chosen to treat it separately both to emphasize the simplifications resulting from the assumption of compact support, and to better illustrate the procedure for treating the charged {component of} the electromagnetic field. The second subsection, which is independent of the first, treats general scalar fields. Here we present the fractional Morawetz estimate for {scalar fields} which is used only in this subsection.} Once the estimates are established in this outer region, we are able to bound the flux on $C_{-1}.$ Section \ref{Inside V_T}, which is the longest section, is then devoted to the proof of the error estimates inside $C_{-1}.$ 
\section{Notation and preliminaries}\label{Notations}
{In this subsection we define the notation that is used in the remainder of this work. We have tried to stay as close as possible to the commonly used notation in the related literature. We also establish some of the basic relations that hold among the various quantities we introduce.
\subsection{Coordinates and weights}\label{Coordinates}
We use two different sets of coordinates. The rectangular coordinates are $x^0,\dots,x^3,$ but we also use $t$ for $x^0.$ When using rectangular coordinates Greek indices ($\alpha$, $\beta$, ...) can take any value between $0$ and $4$ and lower case Roman indices ($i,~j,$ ...) correspond only to spatial variables and thus take values between $1$ and $3.$ We will also use null coordinates defined as
\Aligns{
u=t-r,\quad\quad\quad\ubar=t+r.
}
In these coordinates capital Roman indices ($A,~B,$ ...) correspond to spherical variables. We also introduce the weights $$\taup^2=1+\ubar^2,\quad\quad\quad\taum^2=1+u^2.$$ The null derivatives $L$ and $\Lbar$ are defined as $$L:=\partial_t+\partial_r,\quad\quad\quad\Lbar=\partial_t-\partial_r.$$ Given a 2-form $G_{\mu\nu}$ we define its null-decomposition as
\begin{align*}
\alphabar_{A}(G):=G_{A\Lbar},\quad \alpha_{A}(G):=G_{AL},\quad \rho(G):=\frac{1}{2}G_{L\Lbar},\quad \sigma(G)\epsilon_{AB}:=G_{AB},
\end{align*}
where $\epsilon_{AB}$ is the the volume form on $S_{t,r}$ (the sphere of radius $r$ on the time-slice $x^0\equiv t,$ see below). If there is no risk of confusion we simply write $\alpha$ instead of $\alpha(G)$ and similarly for the other components.
\subsection{Regions}\label{Regions}
We denote by $\Sigma_t$ the spatial hypersurface $x^0\equiv t.$ $C_u$ and ${\Cbar}_{\ubar}$ denote the outgoing and incoming null cones $t-r\equiv u$ and $t+r\equiv\ubar$ respectively. $S_{t,r},~\tilde{S}_{u,r}$  and $\tilde{S}_{\ubar,r}$ are the spheres of constant radius $r$ on $\Sigma_t,~C_u$ and ${\Cbar}_{\ubar}$ {respectively. Due} to the presence of charge we need separate estimates in the interior and exterior of a fixed outgoing null cone. For this we introduce the following regions
\Aligns{
V_T:=\{t\leq T,~u\leq-1\},\quad\quad V^O_T:={\{t\leq T,~u\geq -1\}}.
}
We also define
\Aligns{
&\Sigma_t(T)=\Sigma_t\cap V_T,\quad\quad C_u(T)=C_u\cap V_T,\quad\quad {\Cbar}_{\ubar}(T)={\Cbar}_{\ubar}\cap V_T,\\
&{\Sigma^O_t(T)=\Sigma_t}\cap V^O_T,\quad\quad C^O_u(T)=C_u\cap V^O_T,\quad\quad {\Cbar}_{\ubar}^O(T)={\Cbar}_{\ubar}\cap V^O_T.
}
\subsection{Vector fields}\label{vector fields}
We use the following vector fields in this work
\begin{align}
&T_\mu=\partial_\mu,\label{T}\\
&S=x^\mu\partial_\mu,\label{S}\\
&\Omega_{\mu\nu}=x_\mu\partial_\nu-x_\nu\partial_\mu,\quad \mu<\nu,\label{Omega}\\
&K_0=\frac{1}{2}(u^2\bar L+\ubar^2 L),\label{K_0}\\
&\Kbar=T_0+K_0=\left(\frac{1+u^2}{2}\right)\Lbar+\left(\frac{1+\bar u^2}{2}\right)L.\label{Kbar}
\end{align}
$T_\mu,~S,$ and $\Omega_{\mu\nu}$ will be used as commutators, while $T_0$ and $K_0$ will be our multipliers. We denote the Lie algebra generated by the commutator vector-fields by $\mathbb{L}.$} We define $\omega_i$ and $\omega_i^A$ via the relation
\Align{
\partial_i=\omega_i\partial_r+\omega_i^Ae_A.
}
The following identities will be used repeatedly.
\Align{\label{vf decompositions}
&T_0=\frac{1}{2}(L+\Lbar),\\
&T_i=\frac{\omega_i}{2}(L-\Lbar)+\omega_i^Ae_A,\\
&\Omega_{ij}=\Omega^A_{ij}e_A=(x_i\omega_j^A-x_j\omega_i^A)e_A,\\
&{\Omega_{i0}}=\frac{\omega_i}{2}(\ubar L-u\Lbar)+t\omega_i^Ae_A,\\
&S=\frac{1}{2}(\ubar L+u\Lbar).
}
For the covariant derivatives we have
\begin{alignat}{3}
&\nabla_L\Lbar=0,\quad\quad\quad\quad\quad\quad&&\nabla_{\Lbar}L=0,\nonumber\\
&\nabla_{\Lbar}\Lbar=0,&&\nabla_L\Lbar=0,\nonumber\\
&\nabla_Le_A=0,&&\nabla_{\Lbar}e_A=0,\label{Ricci coefficients}\\
&\nabla_{e_A}L=\frac{1}{r}e_A,&&\nabla_{e_A}\Lbar=-\frac{1}{r}e_A,\nonumber\\
&\nabla_{e_A}e_B=\SGamma^D_{AB}e_D+\frac{1}{2r}\delta_{AB}(\Lbar-L),\nonumber
\end{alignat}
where in the last term $\SGamma^D_{AB}$ denote the Christoffel symbols of the spheres. Similarly
 \Align{\label{vf convariant der decompositions}
 &\nabla_LT_\alpha=\nabla_{\Lbar}T_\alpha=\nabla_{e_A}T_\alpha=0,\\
 &\nabla_L\Omega_{ij}=-\nabla_{\Lbar}\Omega_{ij}=\frac{1}{r}\Omega_{ij},\\
 &\nabla_{e_B}\Omega_{ij}=e_B(\Omega^A_{ij})e_A+\Omega^A_{ij}\SGamma^D_{AB}e_D+\frac{1}{2r}\Omega^B_{ij}(\Lbar-L),\\
 &\nabla_L\Omega_{i0}=\omega_iL+\omega_i^Ae_A,\\
 &\nabla_{\Lbar}\Omega_{i0}=-\omega_i\Lbar+\omega_i^Ae_A,\\
 &\nabla_{e_B}\Omega_{i0}=\frac{\omega_i^B}{2}(L+\Lbar),\\
 &\nabla_LS=L,\\
 &\nabla_{\Lbar}S=\Lbar,\\
 &\nabla_{e_A}S=e_A.
 }
Finally we have the following commutation relations
\Align{\label{vectorfield commutators}
&[L,T_0]=[\Lbar,T_0]=[e_A,T_0]=0,\\
&[L,T_i]=\omega_i^A[L,e_A]=-\frac{\omega_i^A}{r}e_A,\\
&[\Lbar,T_i]=\omega_i^A[\Lbar,e_A]=\frac{\omega_i^A}{r}e_A,\\
&[L,\Omega_{ij}]=[\Lbar,\Omega_{ij}]=0,\\
&[e_B,\Omega_{ij}]=e_B(\Omega_{ij}^A)e_A+\Omega_{ij}^A[e_B,e_A]^De_D,\\
&[L,\Omega_{i0}]=\omega_i L-\frac{u}{r}\omega_i^Ae_A,\\
&[\Lbar,\Omega_{i0}]=-\omega_i\Lbar+\frac{\ubar}{r}\omega_i^Ae_A,\\
&[e_B,\Omega_{i0}]=\frac{\omega_i^B}{2r}(\ubar L-u\Lbar)-t\omega_i^A\SGamma^D_{AB}e_D,\\
&[L,S]=L,\\
&[\Lbar,S]=\Lbar,\\
&[e_A,S]=e_A,\\
&[\Omega_{i0},\Omega_{jk}]=\delta_{i}^{j}\Omega_{0k}-\delta_{i}^{k}\Omega_{0j},\\
&[\Omega_{0i},\Omega_{0j}]=\Omega_{ij},\\
&[\Omega_{ij},\Omega_{kl}]=\delta_{j}^{k}\Omega_{il}+\delta_{j}^{l}\Omega_{ki}+\delta_{i}^{k}\Omega_{lj}+\delta_{il}\Omega_{jk},\\
&[\Omega_{0i},S]=[\Omega_{ij},S]{=0.}
}
{\subsection{Differential operators and their commutation relations}\label{Diff Ops} For a vector field $X$ we denote by $L_X$ the usual Lie derivative with respect to $X.$ Moreover, as stated earlier, $D_X$ is defined as $\nabla_X+iX^\alpha A_\alpha,$ where $\nabla_X$ is the Levi-Cevita covariant differentiation operator and $A_\alpha$ is defined by {$F_{\alpha\beta}=\nabla_\alpha A_\beta-\nabla_\beta A_\alpha.$ {Operators} on the spheres $S_{t,r}$ are written as $\snab,~\sD,$ ... .}
If $X$ is conformal Killing we denote by $\Omega_X$ its conformal factor, $L_Xg=\Omega_Xg.$ We have
\Aligns{
\Omega_{T_\mu}=\Omega_{\Omega_{\mu\nu}}=0,\quad \Omega_S=2,\quad\Omega_{\Kbar}=4t.
}
If $\phi,~J$ and $G$ are a scaler field, one-form, and two-form respectively, we define the modified Lie derivatives
\Aligns{
&\LL_X\phi=D_X\phi+\Omega_X\phi,\\
&\LL_X J=L_X J+\Omega_X J,\\
&\LL_XG=L_XG.
}
We also define $\LL$ for bundle-valued tensors by requiring it to be a derivation and that it agree with $D$ on sections. In particular we have
\Aligns{
&\LL_X D_\mu\phi:=X^\nu D_\nu D_\mu\phi+\Omega_X D_\mu\phi+\nabla_\mu X^\nu D_\nu\phi,\\
&\LL_X D_\mu D_\nu\phi:=X^\alpha D_\alpha D_\mu D_\nu\phi+\Omega_X D_\mu D_\nu\phi+\nabla_\mu X^\alpha D_\alpha D_\nu\phi+\nabla_\nu X^\alpha D_\mu D_\alpha\phi.
}
If $X_1,\dots,X_k$ are vector fields in $\{S,\Omega_{\mu\nu}\},$ and $I=(i_1,\dots,i_m)$ a multi-index we define $|I|:=m$ and
\Aligns{
\LL^{|I|}_{X_I}:=\LL_{X_{i_1}}\dots \LL_{X_{i_m}}.
}
$L^{|I|}_{X_I}$ is defined similarly. When only the number of differentiations is important (rather than the vector fields $X_i$ themselves) we will sometimes abbreviate this notation to $\LL^{|I|}$ and $L^{|I|}.$ If $I=(i_1,\dots,i_m)$ we will write $i_1,\dots,i_m\in I.$
}
The next lemmas summarize some commutation properties of the modified Lie derivative.
\begin{lemma}\label{1st commutator} Let $X$ be a conformal Killing vector field with constant conformal factor $\Omega_X$. Then
\Align{\label{[L,equations]}
&D^\mu D_\mu\LL_X\phi-\LL_XD^\mu D_\mu \phi=i\left(2X^\alpha F_{\mu\alpha}D^\mu\phi+\nabla^\mu(X^\alpha F_{\mu\alpha})\phi\right),\\
&\nabla^\mu \LL_XG_{\mu\nu}-\LL_X(\nabla^\mu G_{\mu\nu})=0,\\
&\star \LL_X G_{\mu\nu}-\LL_X\star G_{\mu\nu}=0.
}
\end{lemma}
\begin{proof}
We only prove the first two statements. First note that
\Align{\label{D-LL commutator}
D_\mu\LL_X\phi&=\Omega_XD_\mu\phi+D_\mu D_X\phi=X^\alpha D_\alpha D_\mu\phi+\Omega_X D_\mu\phi+\nabla_\mu X^\alpha D_\alpha\phi+iX^\alpha F_{\mu\alpha}\phi\\
&=\LL_X D_\mu\phi+iX^\alpha F_{\mu\alpha}\phi.
}
It follows that
\Aligns{
{D_\nu D_\mu}\LL_X\phi=&\nabla_\nu X^\alpha D_\alpha D_\mu\phi+X^\alpha D_\alpha D_\nu D_\mu\phi+iX^\alpha F_{\nu\alpha}D_\mu\phi+\Omega_XD_\nu D_\mu\phi\\
&+\nabla_\nu\nabla_\mu X^\alpha D_\alpha\phi+\nabla_\mu X^\alpha D_\nu D_\alpha\phi+i\nabla_\nu(X^\alpha F_{\mu\alpha})\phi+iX^\alpha F_{\mu\alpha}D_\nu\phi.
}
Since $\nabla^2 X=0$ for the vector-fields under consideration, the first statement of the lemma follows by contracting in the $\mu$ and $\nu$ coordinates. For the second statement note that
\Aligns{
\nabla_\rho\LL_XG_{\mu\nu}&=\nabla_\rho X^\alpha\nabla_\alpha G_{\mu\nu}+X^\alpha\nabla_\alpha\nabla_\rho G_{\mu\nu}+\nabla_\rho\nabla_\mu X^\alpha G_{\alpha\nu}+\nabla_\mu X^\alpha\nabla_\rho G_{\alpha\nu}\\
&\quad+\nabla_\rho\nabla_\nu X^\alpha G_{\mu\alpha}+\nabla_\nu X^\alpha\nabla_\rho G_{\mu\alpha}\\
&=X^\alpha\nabla_\alpha\nabla_\rho G_{\mu\nu}+\nabla_\rho X^\alpha \nabla_\alpha G_{\mu\nu}+\nabla_\mu X^\alpha \nabla_\rho G_{\alpha\nu}+\nabla_\nu X^\alpha\nabla_\rho G_{\mu\alpha}\\
&=L_X\nabla_\rho G_{\mu\nu}.
}
Noting that $L_X g^{\mu\rho}=-\Omega_X g^{\mu\rho}$ we get
\Aligns{
\nabla^{\mu}\LL_X G_{\mu\nu}=L_X(\nabla^\mu G_{\mu\nu})+\Omega_X\nabla^\mu G_{\mu\nu}=\LL_X(\nabla^\mu G_{\mu\nu}).
}
\end{proof}
\begin{lemma}\label{second derivative commutation} Let $X$ and $Y$ be conformal Killing vectorfields with constant conformal factors. Then
\Aligns{
D^\mu D_\mu\LL_Y\LL_X\phi-\LL_Y\LL_XD^\mu D_\mu\phi=&i\left(2Y^\beta F_{\mu\beta}D^\mu\LL_X\phi+\nabla^\mu(Y^\beta F_{\mu\beta})\LL_X\phi\right)\\
&+i\left(2X^\alpha F_{\mu\alpha}D^\mu\LL_Y\phi+\nabla^\mu(X^\alpha F_{\mu\alpha})\LL_Y\phi\right)\\
&+i\left(2[Y,X]^\alpha F_{\mu\alpha}D^\mu\phi+\nabla^\mu([Y,X]^\alpha F_{\mu\alpha})\phi\right)\\
&+i\left(2X^\alpha\LL_YF_{\mu\alpha}D^\mu\phi+\nabla^\mu (X^\alpha\LL_Y F_{\mu\alpha})\phi\right)\\
&+2X^\alpha Y^\beta F^\mu_{~\alpha}F_{\mu\beta}\phi.
}
\end{lemma}
\begin{proof}
From Lemma \ref{1st commutator} we have
\Align{\label{2nd der com 1}
D^\mu D_\mu\LL_Y\LL_X\phi-\LL_Y D^\mu D_\mu\LL_X\phi=i\left(2Y^\nu F_{\mu\nu}D^\mu\LL_X\phi+\nabla^\mu(Y^\nu F_{\mu\nu})\LL_X\phi\right).
}
Again applying Lemma \ref{1st commutator} to $\LL_Y D^\mu D_\mu\LL_X\phi$ we get
\Aligns{
\LL_Y D^\mu D_\mu\LL_X\phi-\LL_Y\LL_X D^\mu D_\mu\phi&=ig^{\mu\nu}\left[2Y^\beta\nabla_\beta X^\alpha F_{\mu\alpha}D_\nu\phi+2X^\alpha Y^\beta\nabla_\beta F_{\mu\alpha}D_\nu\phi\right.\\
&\quad\quad\quad+2X^\alpha\nabla_\mu Y^\beta F_{\beta\alpha}D_\nu\phi+2X^\alpha Y^\beta F_{\mu\alpha}D_\nu D_\beta\phi\\
&\quad\quad\quad+2iX^\alpha Y^\beta F_{\mu\alpha}F_{\beta\nu}\phi+2X^\alpha F_{\mu\alpha}D_\nu\Omega_Y\phi\\
&\quad\quad\quad+2X^\alpha F_{\mu\alpha}\nabla_\nu Y^\beta D_\beta \phi\\
&\quad\quad\quad+Y^\beta\nabla_\beta\nabla_\nu(X^\alpha F_{\mu\alpha})\phi+\nabla_\nu Y^\beta\nabla_\beta(X^\alpha F_{\mu\alpha})\phi\\
&\quad\quad\quad\left. +\nabla_\nu(X^\alpha F_{\mu\alpha})Y^\beta D_\beta\phi+\nabla_\nu(X^\alpha F_{\mu\alpha})\Omega_Y\phi\right]\\
&=i\left(2X^\alpha F_{\mu\alpha}D^\mu\LL_Y\phi+\nabla^\mu(X^\alpha F_{\mu\alpha})\LL_Y\phi\right)\\
&\quad+i\left(2[Y,X]^\alpha F_{\mu\alpha}D^\mu\phi+\nabla^\mu([Y,X]^\alpha F_{\mu\alpha})\phi\right)\\
&\quad+i\left(2X^\alpha\LL_YF_{\mu\alpha}D^\mu\phi+\nabla^\mu (X^\alpha\LL_Y F_{\mu\alpha})\phi\right)\\
&\quad+2X^\alpha Y^\beta F^\mu_{~\alpha}F_{\mu\beta}\phi.
}
\end{proof}
\begin{lemma}\label{kth commutator}
For $X_1,\dots,X_k,~k\geq2,$ conformal Killing vector-fields in {$\mathbb{L},$}
\Aligns{
D^\mu D_\mu \LL^k_{X_1,\dots,X_k}\phi-\LL^k_{X_1,\dots,X_k}D^\mu D_\mu\phi=A+B,
}
where $A$ is a linear combination (with coefficients zero or one) of terms of the form
\Aligns{
i\left(2L_{X_{I_1}}^{|I_1|}X_{i_1}^\alpha \LL_{X_{I_2}}^{|I_2|}F_{\mu\alpha} D^\mu\LL_{X_{I_3}}^{|I_3|}\phi+\nabla^\mu\left(L_{X_{I_1}}^{|I_1|}X_{i_1}^\alpha\LL_{X_{I_2}}^{|I_2|}F_{\mu\alpha}\right)\LL_{X_{I_3}}^{|I_3|}\phi\right),\quad &|I_1|+|I_2|+|I_3|=k-1,\\&i_1\notin I_1,
}
and $B$ is a linear combination (with coefficients zero or one) of terms of the form
\Aligns{
2L_{X_{I_1}}^{|I_1|}X_{i_1}^\alpha L_{X_{I_2}}^{|I_2|}X_{i_2}^\beta\LL_{X_{I_3}}^{|I_3|}F_{\mu\alpha}\LL_{X_{I_4}}^{|I_4|}F^\mu_{~\beta}\LL_{X_{I_5}}^{|I_5|}\phi,\quad |I_1|+\dots+|I_5|=k-2,~i_j\notin I_j.
}
\end{lemma}
{\begin{proof}
The proof is the same as in the previous lemma.
\end{proof}
The following technical lemma will be used in showing that $\taup D_L\phi$ and $\taum D_{\Lbar} \phi$ have the same decay as $\phi.$}
\begin{lemma}\label{KS commutator estimates 1}{The following commutation estimates hold for any scalar field $\phi.$
\begin{align}
&\left|\LL_{S}(\taum D_{\Lbar}\phi)-\taum D_{\Lbar}\LL_{S}\phi\right| \lesssim \taum |D_{\Lbar}\phi|+\taum\taup|\rho||\phi|\label{Sob-Klain com1}\\
&\left|\LL_{S}(\taup D_{L}\phi)-\taup D_{L}\LL_{S}\phi\right| \lesssim \taup|D_{L}\phi|+\taum\taup|\rho||\phi|\label{Sob-Klain com2}\\
&\left|\LL_{\Omega_{ij}}(\taum D_{\Lbar}\phi)-\taum D_{\Lbar}\LL_{\Omega_{ij}}\phi\right|\lesssim\taum\taup|\alphabar||\phi|+\taum|\slashed{D}\phi|+|D_{\Lbar}\phi|\label{Sob-Klain com3}\\
&\left|\LL_{\Omega_{ij}}(\taup D_{L}\phi)-\taup D_{L}\LL_{\Omega_{ij}}\phi\right|\lesssim\taup^2|\alpha||\phi|+\taup|\slashed{D}\phi|+|D_{L}\phi|\label{Sob-Klain com4}\\
&\left|\LL_{\Omega_{i0}}(\taum D_{\Lbar}\phi)-\taum D_{\Lbar}\LL_{\Omega_{i0}}\phi\right|\lesssim\taum |D_{\Lbar}\phi|+\taum\taup(|\alphabar|+|\rho|)|\phi|+\taum|\slashed{D}\phi|\label{Sob-Klain com5}\\
&\left|\LL_{\Omega_{i0}}(\taup D_{L}\phi)-\taup D_{L}\LL_{\Omega_{i0}}\phi\right|\lesssim\taup |D_L\phi|+(\taup^2|\alpha|+\taup\taum|\rho|)|\phi|+\taup|\slashed{D}\phi|\label{Sob-Klain com6}\\
&|\LL_{\partial_{0}}(\taum D_{\Lbar}\phi)-\taum D_{\Lbar}\LL_{\partial_{0}}\phi|\lesssim \taum|\rho||\phi|+|D_{\Lbar}\phi|\label{Sob-Klain com8}\\
&{|\LL_{\partial_{0}}}(\taup D_{L}\phi)-\taup D_{L}\LL_{\partial_{0}}\phi|\lesssim \taup|\rho||\phi|+|D_{L}\phi|\label{Sob-Klain com9}\\
&|\LL_{\partial_{i}}(\taum D_{\Lbar}\phi)-\taum D_{\Lbar}\LL_{\partial_{i}}\phi|\lesssim \taum|\rho||\phi|+|D_{\Lbar}\phi|+\taum|\slashed{D}\phi|+\frac{1}{r}|D_{\Lbar}\phi|\label{Sob-Klain com10}\\
&|\LL_{\partial_{i}}(\taup D_{L}\phi)-\taup D_{L}\LL_{\partial_{i}}\phi|\lesssim \taup|\rho||\phi|+|D_{L
}\phi|+\taup|\slashed{D}\phi|+\frac{1}{r}{|D_{L}\phi|.}\label{Sob-Klain com11}
\end{align}}
Moreover if ${\Gamma_1,\Gamma_2\in\mathbb{L}}$ then
\Align{\label{Sob-Klain com7}
&\left|\LL_{\Gamma_1}\LL_{\Gamma_2}(\taum D_{\Lbar}\phi)-\taum D_{\Lbar}\LL_{\Gamma_1}\LL_{\Gamma_2}\phi\right|+\left|\LL_{\Gamma_1}\LL_{\Gamma_2}(\taup D_{L}\phi)-\taup D_{L}\LL_{\Gamma_1}\LL_{\Gamma_2}\phi\right|\\
&\quad\lesssim\taum\left(|D_{\Lbar}\phi|+\sum_{\Gamma\in\{S,\Omega_{\mu\nu}\}}|D_{\Lbar}\LL_{\Gamma}\phi|\right)+\taup\left(|D_{L}\phi|+\sum_{\Gamma\in\{S,\Omega_{\mu\nu}\}}|D_{L}\LL_{\Gamma}\phi|\right)\\
&\quad\quad+\left(\sum_{\stackrel{|I|\leq1}{\Gamma\in\{S,\Omega_{\mu\nu}\}}}\frac{\taup}{r}|\LL^I_\Gamma\phi|\right)\left(\taup\taum|\alphabar(F)|+\taup^2\left(|\alpha(F)|+|\rho(F)|+|\sigma(F)|\right)\right)\\
&\quad\quad+|\phi| \left(\sum_{\Gamma\in\{S,\Omega_{\mu\nu}\}}\left(\taup\taum|\alphabar(\LL_\Gamma F)|+\taup^2\left(|\alpha(\LL_\Gamma F)|+|\rho(\LL_\Gamma F)|+|\sigma(\LL_\Gamma F)|\right)\right)\right).
} 
\end{lemma}
{\begin{proof}
First we note that if $\Gamma_i,~i=1,2,$ are two conformal Killing vector fields whose conformal factors $\Omega_i$ are constants, then
\Aligns{
&\LL_{\Gamma_1}(\taum D_{\Lbar}\phi)=\taum D_{\Lbar}\LL_{\Gamma_1}\phi+\Gamma_1(\taum)D_{\Lbar}\phi
+i\taum \Gamma_1^\nu F_{\nu\Lbar}\phi
+\taum\nabla_{\Gamma_{1}}\Lbar^{\nu}D_{\nu}\phi,\\
&\LL_{\Gamma_1}(\taup D_{L}\phi)=\taup D_{L}\LL_{\Gamma_1}\phi+\Gamma_1(\taup) D_{L}\phi+i\taup \Gamma_1^\nu F_{\nu L}\phi+\taup \nabla_{\Gamma_{1}}L^{\nu}D_{\nu}\phi,
}
and
\begin{align*}
&\LL_{\Gamma_{2}}\LL_{\Gamma_{1}}(\taum D_{\Lbar}\phi)\\
&=\Gamma_{2}\Gamma_{1}(\taum)D_{\Lbar}\phi+\Gamma_{1}(\taum)\LL_{\Gamma_{2}}D_{\Lbar}\phi+\Gamma_{2}(\taum)\LL_{\Gamma_{1}}D_{\Lbar}\phi\\
&+\taum\LL_{\Gamma_{2}}\left(D_{\Lbar}\LL_{\Gamma_{1}}\phi+i\Gamma_{1}^{\nu}F_{\nu\Lbar}\phi+\nabla_{\Gamma_{1}}\Lbar^{\nu}D_{\nu}\phi\right)=I+II
\end{align*}
$I$ can be treated as the first order commutators. Since we have either $\nabla_{\Gamma_{1}}\Lbar=0$ or $\nabla_{\Gamma_{1}}\Lbar\sim\Lbar$, the last term in $II$ is either $0$ or $\Lbar^{\nu}\LL_{\Gamma_{2}}D_{\nu}\phi$, and can be seen to have the right behavior. The second term in $II$ is
\begin{align*}
i\taum\LL_{\Gamma_{2}}\left(\Gamma_{1}^{\nu}F_{\nu\Lbar}\phi\right)+i\LL_{\Gamma_{2}}(\taum)\Gamma_{1}^{\nu}F_{\nu\Lbar}\phi.
\end{align*} 
The second term of the above has the right behavior. For the first term, we note that since the usual and modified Lie derivatives agree for two-forms
\Align{\label{G2ofG1F-1}
\Gamma_2^\mu\nabla_\mu(\Gamma_1^\nu F_{\nu\Lbar})=\Gamma_1^\nu\LL_{\Gamma_2}F_{\nu\Lbar}+[\Gamma_2,\Gamma_1]^\nu F_{\nu\Lbar}+\Gamma_1^\nu F_{\nu\mu}[\Gamma_2,\Lbar]^\mu
}
and
\Align{\label{G2ofG1F-2}
\Gamma_2^\mu\nabla_\mu(\Gamma_1^\nu F_{\nu L})=\Gamma_1^\nu\LL_{\Gamma_2}F_{\nu L}+[\Gamma_2,\Gamma_1]^\nu F_{\nu L}+\Gamma_1^\nu F_{\nu\mu}[\Gamma_2,L]^\mu.
}
Since $[\Gamma_{1},\Gamma_{2}]\sim\Omega_{ij}, \taup L, \taum\Lbar$ and $[\Gamma,\Lbar]\sim \Lbar,0$ and $ [\Gamma,L]\sim L,0$, the contribution from these terms are also of the right order. 
\end{proof}}
{\subsection{Energies}\label{Energies} To a two-form $G$ we associate the following energy momentum tensor}
\EQ{\label{QG}
T(G)_{\mu\nu}=G_{\mu\alpha}G_\nu^{~\alpha}+\star G_{\mu\alpha}\star G_\nu^{~\alpha}.
}
Similarly for a scalar field $\phi$ we define
\EQ{\label{QPHI}
T(\phi)_{\mu\nu}=\Re(\con{D_\mu\phi}D_\nu\phi)-\frac{1}{2}g_{\mu\nu}\con{D^\alpha\phi}{D_\alpha\phi,}
}
and let
\EQ{\label{QGP}
T(G,\phi)_{\mu\nu}=T(G)_{\mu\nu}+{T(\phi)_{\mu\nu}.}
}
When there is no risk of confusion {we write} $T$ instead of $T(G,\phi)$ . The following lemma is standard.
\begin{lemma}
$T(G,\phi),~T(\phi),$ and $T(G)$ are symmetric and $T(G)$ is traceless. Moreover if $(F,\phi)$ is a solution of (\ref{MKG}) then $\nabla^\mu T(F,\phi)_{\mu\nu}=0.$
\end{lemma}

The energy norms {inside of $V_T$} are defined as
\begin{align}\label{energies}
&Q_0(\phi)(t,T)^2=\int_{\Sigma_t(T)}\left(\tau_{+}^2|(D_L+\frac{1}{r})\phi|^2+\tau_{-}^2|(D_{\Lbar}-\frac{1}{r})\phi|^2+(\tau_{+}^2+\tau_{-}^2)(|\sD\phi|^2+\frac{|\phi|^2}{r^2})\right)\\\notag
&Q_0(G)(t,T)^2=\int_{\Sigma_t(T)}\left(\taup^2|\alpha(G)|^2+\taum^2|\alphabar(G)|^2+(\taup^2+\taum^2)(|\rho(G)|^2+|\sigma(G)|^2)\right)\\\notag
&Q_{out}(\phi)(u,T)^2=\int_{C_u(T)}\left(\taup^2|(D_L+\frac{1}{r})\phi|^2+\taum^2(|\sD\phi|^2+\frac{|\phi|^2}{r^2})\right)\\\notag
&Q_{out}(G)(u,T)^2=\int_{C_u(T)}\left(\taup^2|\alpha(G)|^2+\taum^2(|\rho(G)|^2+|\sigma(G)|^2)\right)\\\notag
&Q_{in}(\phi)(\ubar,T)^2=\int_{{\Cbar}_{\ubar}(T)}\left(\taum^2|(D_{\Lbar}-\frac{1}{r})\phi|^2+\taup^2(|\sD\phi|^2+\frac{|\phi|^2}{r^2})\right)\\\notag
&Q_{in}(G)(\ubar,T)^2=\int_{{\Cbar}_{\ubar}(T)}\left(\taum^2|\alphabar(G)|^2+\taup^2(|\rho(G)|^2+|\sigma(G)|^2)\right).
\end{align}
We also let $Q_0(G,\phi)=Q_0(G)+Q_0(\phi)$ and define $Q_{in}(G,\phi)$ and $Q_{out}(G,\phi)$ analogously. The supremum of these energies in a region of spacetime is encoded in the following energies
\EQ{\label{Q*}
&Q^{*}(\phi)(T)=\sup_{0\leq t\leq T}Q_0(\phi)(t,T)+\sup_{-1\leq u\leq\infty}Q(\phi)(u,T)+\sup_{0<\ubar<\infty}Q(\phi)(\ubar,T)\\
&Q^{*}(\phi)(T)=\sup_{0\leq t\leq T}Q_0(G)(t,T)+\sup_{-1\leq u\leq\infty}Q(G)(u,T)+\sup_{0<\ubar<\infty}Q(G)(\ubar,T)\\
&Q^{*}(G,\phi)(T)=Q^{*}(G)(T)+Q^{*}(\phi)(T).
}
Finally the higher energy norms are defined as
\EQ{\label{Qi}
{Q_{0}}_{|I|}(G)(t,T)={\sum_{\Gamma_{I}\in\mathbb{L}^{|I|}}Q_{0}(\LL^{|I|}_{\Gamma_I}G)(t,T),}
}
with similar formulas for the other energies. We will eventually be interested in bounding {$Q(\LL^I_\Gamma F,\LL^I_\Gamma\phi).$ We} will next prove some energy estimates for the linear inhomogeneous equations
{
\begin{align}\label{inhomogeneous linear equations}
\nabla_{[\alpha}G_{\beta\gamma]}=0,\quad\nabla^{\mu}G_{\mu\nu}=J_{\nu},\quad D^{\alpha}D_{\alpha}\phi=f.
\end{align}}
\begin{proposition}
For the first two equations in \eqref{inhomogeneous linear equations}, we have the following estimate
{\begin{align}\label{Morawetz-F1}
&\int_{\Sigma_{t}(T)}\big(\taup^{2}|\alpha|^{2}+\taum^{2}|\alphabar|^{2}+(\taup^{2}+\taum^{2})(|\rho|^{2}+|\sigma|^{2})\big)\\\notag
&+\int_{\Cbar_{\ubar}(T)\bigcap\{0\leq t'\leq t\}}\big(\taum^{2}|\alphabar|^{2}+\taup^{2}(|\rho|^{2}+|\sigma|^{2})\big)\\\notag
&+\int_{C_{u}(T)\bigcap\{0\leq t'\leq t\}}\big(\taup^{2}|\alpha|^{2}+\taum^{2}(|\rho|^{2}+|\sigma|^{2})\big)\\\notag
&\leq C\int_{C_{-1}\bigcap\{0\leq t'\leq t\}}\big(\taup^{2}|\alpha|^{2}+\taum^{2}(|\rho|^{2}+|\sigma|^{2})\big)+\int_{0}^{t}\int_{\Sigma_{t'}(T)}\big|\overline{K}^{\nu}_{0}{G}_{\mu\nu}J^{\mu}\big|dxdt'\\\notag
&\quad+C\int_{\Sigma_{0}(T)}r^{2}\big(|\alpha|^{2}+|\alphabar|^{2}+|\rho|^{2}+|\sigma|^{2}\big)dx.
\end{align}
Here $\alpha,\alphabar, \rho,\sigma$ are associated to {$G$}.}
\begin{proof}
The proof is {a standard application of the divergence theorem to $T^{\mu\nu}{\Kbar}_\nu,$}and we just need to observe that{
\begin{align*}
Q(G)(\Lbar,\Lbar)=2|\alphabar|^{2},\quad Q(G)(L,L)=2|\alpha|^{2},\quad Q(G)(\Lbar,L)=2(\rho^{2}+\sigma^{2}), \quad \nabla^\mu(Q_{\mu\nu}\overline{K}_{0}^{\nu})=\overline{K}^{\nu}_{0}\nabla^{\mu}Q_{\mu\nu}.
\end{align*}}
\end{proof}
\end{proposition}

{\begin{proposition}\label{LS Morawetz}
{For the third equation of \eqref{inhomogeneous linear equations}, and $\Omega=r$ or $u\ubar$
\begin{align}
&\int_{\Sigma_{t}(T)}\left(\ubar^{2}|\frac{1}{\Omega}D_{L}(\Omega\phi)|^{2}+u^{2}|\frac{1}{\Omega}
D_{\Lbar}(\Omega\phi)|^{2}+(\ubar^{2}+u^{2})(|\slashed{D}\phi|^{2}+\frac{|\phi|^{2}}{r^{2}})\right)\nonumber\\
&+\int_{\Cbar_{\ubar}(T)\cap\{0\leq t'\leq t\}}\left(u^{2}|(\frac{1}{\Omega}D_{\Lbar}(\Omega\phi)|^{2}+\ubar^{2}|\slashed{D}\phi|^{2}+\frac{\ubar^{2}}{r^{2}}|\phi|^{2}\right)\nonumber\\
&+\int_{C_{u}(T)\cap\{0\leq t'\leq t\}}\left(\ubar^{2}|\frac{1}{\Omega}D_{L}(\Omega\phi)|^{2}+u^{2}|\slashed{D}\phi|^{2}+\frac{\ubar^{2}}{r^{2}}|\phi|^{2}\right)\nonumber\\
&\leq C\int_{0}^{t}\int_{\Sigma_{t}(T)}\left(|f\cdot\frac{1}{r}D_{K_{0}}(r\phi)|+|f\cdot\frac{1}{u\ubar}D_{K_{0}}(u\ubar\phi)|+|K^{\beta}_{0}F_{\beta\gamma}\Im
\left(\phi\overline{D^{\gamma}\phi}\right)|\right)dxdt'\nonumber\\
&\quad+CQ_{0}(\phi)(0,T)^{2}+\int_{C_{-1}(T)\cap\{0\leq t'\leq t\}}\left(\ubar^{2}|\frac{1}{\Omega}D_{L}(\Omega\phi)|^{2}+u^{2}|\slashed{D}\phi|^{2}+\frac{\ubar^{2}}{r^{2}}|\phi|^{2}\right).\label{Morawetz-phi}
\end{align}}
\end{proposition}}
\begin{proof}
We follow the argument in {\cite{LS1}} to prove \eqref{Morawetz-phi}. We denote the Minkowski metric by $g_{\alpha\beta}$ and consider its conformal metric
\begin{align*}
\widetilde{g}_{\alpha\beta}=\frac{1}{\Omega^{2}}g_{\alpha\beta}
\end{align*}
for some weight function $\Omega$ on $\mathbb{R}\times\mathbb{R}^{3}$. Then $\Omega\phi$ satisfies
\begin{align*}
\widetilde{D}^{\alpha}\widetilde{D}_{\alpha}(\Omega\phi)-\Omega f\Omega^{3}\nabla^{\alpha}\nabla_{\alpha}
\left(\frac{1}{\Omega}\right)=\Omega^{3}f.
\end{align*}
Now we fix $\Omega=r$ or $u\ubar,$ and note that for these choices $\nabla^{\alpha}\nabla_{\alpha}\left(\Omega^{-1}\right)=0$ and $\LL_{K_{0}}\widetilde{g}=0.$ We define the corresponding energy-momentum tensor associated to these two conformal factors as
\begin{align*}
\widetilde{Q}_{\alpha\beta}[\phi]=\Re\big(D_{\alpha}(\Omega\phi)\overline{D_{\beta}(\Omega\phi)}\big)-\frac{1}{2}\widetilde{g}_{\alpha\beta}\widetilde{D}^{\gamma}(\Omega\phi)\overline{D_{\gamma}(\Omega\phi)}.
\end{align*}
Here we have used the notation $\widetilde{D}^{\gamma}=\widetilde{g}^{\alpha\gamma}
D_{\alpha}$. The actual definition for $\widetilde{D}^{\gamma}$ is $\widetilde{D}^{\gamma}=\widetilde{g}^{\alpha\gamma}
\widetilde{D}_{\alpha}$, but here we only apply $\widetilde{D}_{\alpha}$ to a scalar function, we can replace it by $D_{\alpha}$. By direct calculation
\begin{align*}
\widetilde{\nabla}^{\alpha}\widetilde{Q}_{\alpha\beta}[\phi]=\Omega^{4}\Big(\Re\big(f\cdot\overline{\frac{1}{\Omega}D_{\beta}(\Omega\phi)}\big)+F_{\beta\gamma}\Im\big(\phi\overline{\frac{1}{\Omega}D^{\gamma}(\Omega\phi)}\big)\Big).
\end{align*}
Here $F_{\beta\gamma}$ comes from the commutator $[D_{\beta},D_{\gamma}]$. Applying the above identity to the multiplier $\overline{K}_{0}$ we have:
\begin{align*}
\widetilde{\nabla}^{\alpha}\big(\widetilde{Q}_{\alpha\beta}[\phi]\big(K_{0}\big)^{\beta}\big)=
\Omega^{4}\Big(\Re\big(f\cdot\overline{\frac{1}{\Omega}D_{K_{0}}(\Omega\phi)}\big)+K_{0}^{\beta}F_{\beta\gamma}\Im\big(\phi\overline{\frac{1}{\Omega}D^{\gamma}(\Omega\phi)}\big)\Big).
\end{align*}
Integrating this identity over various space-time domains we obtain energy estimates involving space-time integrals of $f$. Since now we work with the conformal metric $\widetilde{g}$, the spacetime measure $d\widetilde{V}$ is given in terms of the measure in Minkowski spacetime $dV$ by
\begin{align*}
d\widetilde{V}=\frac{1}{\Omega^{4}}dV.
\end{align*}
Similarly, the measures associated to the conformal metric on the hypersurfsces $\Sigma_{t}$ and $C_{u}$ are
\begin{align*}
d\widetilde{\Sigma}_{t}=\frac{1}{\Omega^{3}}d\Sigma_{t},\quad d\widetilde{C}_{u}=\frac{1}{\Omega^{3}}dC_{u}
\end{align*}
On the other hand, associated to the conformal metric $\widetilde{g}$, the unit normal vector-fields to {$\Sigma_{t},~C_{u},$ and ${\Cbar}_{\ubar}$ are, up to multiplication by a constant, $\Omega(L+\Lbar),~\Omega L,$ and $\Omega\Lbar$ respectively.} Noting that $\phi$ vanishes on $C_{-1},$ for $\Omega=r,~u\ubar,$ we get{
\begin{align*}
&\int_{\Sigma_{t}(T)}\left(\ubar^{2}|\frac{1}{\Omega}(D_{L}(\Omega\phi)|^{2}+u^{2}|\frac{1}{\Omega}D_{\Lbar}(\Omega\phi)|^{2}+(\ubar^{2}+u^{2})|\slashed{D}\phi|^{2}\right)\\
&+\int_{{\Cbar}_{\ubar}(T)\cap\{0\leq t'\leq t\}}\left(u^{2}|\frac{1}{\Omega}D_{\Lbar}(\Omega\phi)|^{2}+\ubar^{2}|\slashed{D}\phi|^{2}\right)\\
&+\int_{C_{u}(T)\cap\{0\leq t'\leq t\}}\left(\ubar^{2}|\frac{1}{\Omega}D_{L}(\Omega\phi)|^{2}+u^{2}|\slashed{D}\phi|^{2}\right)\\
&\leq ``\textrm{initial data}" +C\int_{0}^{t}\int_{\Sigma_t(T)}\left(|f\cdot\frac{1}{\Omega}D_{K_{0}}(\Omega\phi)|+|K^{\beta}_{0}F_{\beta\gamma}\Im
\left(\phi\overline{D^{\gamma}\phi}\right)|\right)dxdt'.
\end{align*}}
{Schematically, here we have first used $\int_{\Sigma_t}+\int_{{\Cbar}_{\ubar}}\leq\int\int+\int_{\Sigma_0}$ and then $\int_{C_u}\leq\int_{\Sigma_t}+\int_{\Sigma_0}+\int\int.$} {To complete the proof of the proposition we only need to control $\left(\frac{\ubar^{2}+u^{2}}{r^{2}}\right)|\phi|^{2},$ for which it suffices to note that
\begin{align*}
\frac{-u\phi}{r}=(\ubar D_{L}\phi+2\phi)-\big(\ubar D_{L}\phi+\frac{\ubar}{r}\phi\big),\quad \frac{\ubar\phi}{r}=(uD_{\Lbar}\phi+2\phi)-\big(uD_{\Lbar}\phi-\frac{u}{r}\phi\big),
\end{align*}
as well as
\begin{align*}
&\ubar^{2}\big|\frac{1}{r}D_{L}(r\phi)\big|^{2}=\big|\ubar D_{L}\phi+\frac{\ubar\phi}{r}\big|^{2},\quad u^{2}\big|\frac{1}{r}D_{\Lbar}(r\phi)\big|^{2}=\big|uD_{\Lbar}\phi-\frac{u\phi}{r}\big|^{2}\\
&\ubar^{2}\big|\frac{1}{u\ubar}D_{L}(u\ubar\phi)\big|^{2}=|\ubar D_{L}\phi+2\phi|^{2},\quad u^{2}\big|\frac{1}{u\ubar}D_{\Lbar}(u\ubar\phi)\big|^{2}=|uD_{\Lbar}\phi+2\phi|^{2}.
\end{align*}}
\end{proof}
{Noting that $\phi$ is supported in $V_T$ and combining Proposition \ref{LS Morawetz} with the usual energy estimates associated to the multiplier $D_{t}\phi$ and Hardy's inequality{
\begin{align*}
\int_{\Sigma_{t}}\Big|\frac{\phi}{r}\Big|^{2}dx\lesssim\int_{\Sigma_{t}}|D\phi|^{2}dx,
\end{align*}}}
we get the following energy estimates for $\phi.$
{\begin{corollary}
Under the same assumption of Proposition \ref{LS Morawetz}
\begin{align}\label{Morawetz-Kbar}
&\int_{\Sigma_{t}(T)}\Big(\taup^{2}|(D_{L}+\frac{1}{r})\phi|^{2}+\taum^{2}|(D_{\Lbar}-\frac{1}{r})\phi|^{2}+(\taup^{2}+\taum^{2})(|\sD\phi|^{2}+\frac{|\phi|^{2}}{r^{2}})\Big)\\\notag
&+\int_{\Cbar_{\ubar}(T)\cap\{0\leq t'\leq t\}}\Big(\taum^{2}|(D_{\Lbar}-\frac{1}{r})\phi|^{2}+\taup^{2}(|\sD\phi|^{2}+\frac{|\phi|^{2}}{r^{2}})\Big)\\\notag
&+\int_{C_{u}(T)\cap\{0\leq t'\leq t\}}\Big(\taup^{2}|(D_{L}+\frac{1}{r})\phi|^{2}+\taum^{2}(|\sD\phi|^{2}+\frac{|\phi|^{2}}{r^{2}})\Big)\\\notag
&\leq CQ_{0}(\phi)(0,T)^{2}+C\int_{0}^{t}\int_{\Sigma_t(T)}|f\cdot\frac{1}{r}D_{\overline{K}_{0}}(r\phi)|+|\overline{K}^{\beta}_{0}F_{\beta\gamma}\Im
\big(\phi\overline{D^{\gamma}\phi}\big)|dxdt'\\\notag
&+\int_{C_{-1}(T)\cap\{0\leq t'\leq t\}}\Big(\taup^{2}|(D_{L}+\frac{1}{r})\phi|^{2}+\taum^{2}(|\sD\phi|^{2}+\frac{|\phi|^{2}}{r^{2}})\Big).
\end{align}
\end{corollary} 
Note that we did not include the contribution from $|f\cdot\dfrac{1}{u\ubar}D_{K_{0}}(u\ubar\phi)|$ on the right hand side because
\begin{align*}
|\frac{1}{u\ubar}D_{K_{0}}(u\ubar\phi)|\lesssim \ubar^{2}|(D_{L}+\frac{1}{r})\phi|+u^{2}|(D_{\Lbar}-\frac{1}{r})\phi|
\end{align*}
which is already included in $|\dfrac{1}{r}D_{\overline{K}_{0}}(r\phi)|$.}
\begin{corollary}\label{div theorem} Under the assumptions of Theorem \ref{main theorem}, and with $\Omega=r$ (or $u\ubar$)
\Aligns{
Q^{*}(G,\phi)(T)^2\lesssim I.D. &+ Q_{out}(G,\phi)(-1,T)^2+\int_{V_T}\left|\frac{1}{\Omega}D_{\Kbar}(\Omega\phi)\right||D^\alpha D_\alpha\phi|\\
&+\int_{V_T}|\Kbar^\nu(F_{\mu\nu}J^\mu(\phi)+G_{\nu\mu}\nabla^\alpha G_\alpha^{~\mu}+\star G_{\nu\mu}\nabla^\alpha\star {G_\alpha^{~\mu})|.}
}
\end{corollary}
{
\subsection{Isoperimetric, Poincar\'e, and Sobolev estimates}\label{Basic Estimates}
In this subsection we record some standard estimates.
\begin{lemma}[Isoperimetric Inequality]\label{isoperimetric lemma}
Let $S$ be a sphere and $\con{f}$ be the average of $f$ over $S.$ Then
\[\int_S|f-\con{f}|^2\leq C\left(\int_S|\snab f|\right)^2,\]
where $\snab$ is the covariant differentiation on $S,$ and $C$ is independent of the radius of the sphere.
\end{lemma}
}
{See for instance \cite{Oss1} for a proof.} {The proof of the following Poincar\'e estimates can be found in Section 3 of \cite{CK1}.
\begin{lemma}[Poincar\'e Inequalities]\label{Poincare lemma}
Let $F_{\mu\nu}$ be an arbitrary two form with null decomposition $\alpha,~\alphabar,~\sigma, ~\rho.$ With $\overline{\rho}, \overline{\sigma}$ denoting the averages of $\sigma$ and $\rho$ over $S_{t,r}$ respectively, we have
\begin{align*}
&\int_{S_{t,r}}|\alpha|^{2}\leq \sum_{{\Omega_{ij}}}\int_{S_{t,r}}|\mathcal{L}_{{\Omega_{ij}}}\alpha|^{2},\quad \int_{S_{t,r}}|\alphabar|^{2}\leq\sum_{{\Omega_{ij}}}\int_{S_{t,r}}|\mathcal{L}_{{\Omega_{ij}}}\alphabar|^{2},\\
&\int_{S_{t,r}}\big(|\sigma-\overline{\sigma}|^{2}+r^{2}|\slashed{\nabla}\sigma|^{2}\big)\leq\sum_{{\Omega_{ij}}}\int_{S_{t,r}}|\mathcal{L}_{{\Omega_{ij}}}\sigma|^{2},\\
&\int_{S_{t,r}}\big(|\rho-\overline{\rho}|^{2}+r^{2}|\slashed{\nabla}\rho|^{2}\big)\leq\sum_{{\Omega_{ij}}}\int_{S_{t,r}}|\mathcal{L}_{{\Omega_{ij}}}\rho|^{2}.
\end{align*}
\end{lemma}
\begin{remark}\label{Poincare remark}
Since in our case the magnetic charge vanishes, $\overline{\sigma}=0$.
\end{remark}
For completeness we also record the following standard estimate. See \cite{CK1} for references.
\begin{lemma}\label{Klainerman-Sobolev Lemma}
For any $f\in C^\infty(\R^{1+3})$
\[\taup\taum^{\frac{1}{2}}|f(t,x)|\leq {C\sum_{\stackrel{|I|{\leq2}}{\Gamma\in\mathbb{L}}}\|\Gamma^If(t,\cdot)\|_{L^2_x}.}\]
\end{lemma}
In our applications of the this estimate we will need to commute the Lie derivatives $\LL_\Gamma$ with the bundle covariant derivatives $D_L$ and $D_{\Lbar}$ and {the following direct corollary of the definition of the energies and {Lemma \ref{KS commutator estimates 1} will} be useful to this end.
\begin{corollary}\label{KS commutator estimates 2}[of Lemma \ref{KS commutator estimates 1}]
Suppose $F$ satisfies
\Aligns{
\taup^2(|\alpha(F)|+|\alpha(\LL_\Gamma F)|+|\rho(F)|+|\rho(\LL_\Gamma(F))|+|\sigma(F)|+\sigma(\LL_\Gamma F))+\taup\taum(|\alpha(F)|+|\alpha(\LL_\Gamma F)|)\leq C,
}
for all $\Gamma\in\mathbb{L}.$ Then for any scalar field $\phi$
\Aligns{
&\sum_{\stackrel{|I|=1,2}{\Gamma\in\mathbb{L}}}\left(\|\LL^I_{\Gamma}(\taum D_{\Lbar}\phi)-\taum D_{\Lbar}\LL^I_{\Gamma}\phi\|_{L^2_x}+\|\LL^I_{\Gamma}(\taup D_{L}\phi)-\taup D_{L}\LL^I_{\Gamma}\phi\|_{L^2_x}\right)\lesssim CQ^*_2(\phi).
}
\end{corollary}}
We will also need the following $L^p$ Sobolev estimates from \cite{Shu1}.
\begin{lemma}\label{L4 Sobolev lemma}
Let $u\geq-1$ and denote by $\tilde{S}_{u,r}$ and $\tilde{S}_{\ubar,r}$ the spheres of radius $r$ on the null hypersurfaces $C_u$ and ${\Cbar}_{\ubar}$ respectively.
\Aligns{
\left(\int_{C_u(T)}r^6|f|^6\right)^{1/6}+\sup_{\tilde{S}_{u,r}\subseteq C_u(T)}\left(\int_{\tilde{S}_{u,r}}r^4|f|^4\right)^{1/4}\lesssim& \left(\int_{C_{u}}(|f|^2+r^2|{L f}|^2+r^2|\snab f|^2)\right)^{1/2}\\
&+\left(\int_{\Sigma_0}(|f|^2+(1+r^2)|\nabla f|^2)\right)^{1/2}
}
and
\Aligns{
\left(\int_{{\Cbar}_{\ubar}(T)}{r^4\taum^2}|f|^6\right)^{1/6}+\sup_{\tilde{S}_{u,r}\subseteq {\Cbar}_{\ubar}(T)}\left(\int_{\tilde{S}_{u,r}}{r^2\taum^2}|f|^4\right)^{1/4}\lesssim& \left(\int_{{\Cbar}_{\ubar}}(|f|^2+{\taum^2}|{\Lbar f}|^2+r^2|\snab f|^2)\right)^{1/2}\\
&+\left(\int_{\Sigma_0}(|f|^2+(1+r^2)|\nabla f|^2)\right)^{1/2}.
}
\end{lemma}
We next establish an improved decay estimate for $(D_L+\frac{1}{r})\phi,$ Recall that for large $r,$ $C^\infty_0$ solutions of the wave equation $\Box f=0$ in $\R^{1+3}$ satisfy the decay
\[|(L+\frac{1}{r})f|\lesssim\taup^{-3}\taum^{1/2},\]
(of course assuming that appropriate energy norms of $f$ are finite) which is faster than the {$\taum^{-1/2}\taup^{-2}$ decay for $Lf$} guaranteed by Lemma \ref{Klainerman-Sobolev Lemma}. This can be seen for example by writing $\Box f=0$ as
\Aligns{
0&=-\Lbar L f+\snab^A\snab_A f+\frac{1}{r}Lf-\frac{1}{r}\Lbar f\\
&=-\Lbar((L+\frac{1}{r})f)+\snab^A\snab_A f+\frac{1}{r}Lf+\Lbar(\frac{1}{r})f.
}
Now in the expression above all the terms except the first enjoy the decay $\taup^{-3}\taum^{-1/2}$ in the region $r\gtrsim t.$ Integrating in the $u$ direction we get the claimed decay for $(L+\frac{1}{r}) f.$ Our goal in the next lemma is to show that the same conclusion holds for solutions of $D^\mu D_\mu\phi=0.$ {While it is possible to derive this decay rate for solutions of $D^\mu D_\mu\phi=0,$ we will prove only the slower decay of $\taup^{-5/2},$ because this is the best possible rate when considering the inhomogeneous equation satisfied by $\LL_\Omega \phi,$ which is the relevant term when establishing higher regularity. See Remark \ref{L-char decay remark} below.}
{\begin{lemma}\label{L-char decay}
Suppose $\phi$ is a solution of $D^\mu D_\mu\phi=0,$ and that $F$ satisfies
\Aligns{
\taup^{5/2}|\alpha|+\taum^{3/2}\taup|\alphabar|+\taup^2|\rho|\leq C.
}
Then in the region $r\geq t/2,~t\geq1$
\[|(D_L+\frac{1}{r})\phi|\lesssim\taup^{-5/2}\]
where the implicit constants depend only on $Q_3(\phi)$ and $C.$
\end{lemma}}
\begin{proof}
We write the equation for $\phi$ as
\Aligns{
0&=-\frac{1}{2}D_{\Lbar}D_{L}\phi-\frac{1}{2}D_LD_{\Lbar}\phi+D^BD_B\phi\\
&=-D_{\Lbar}D_L\phi-i\rho(F)\phi+D^BD_B\phi.
}
Now note that {with $\chi$ and $\chibar$ denoting the second fundamental forms on the spheres with respect to $L$ and $\Lbar$ respectively}
\Aligns{
(DD\phi)(e_B,e_B)&=D_{e_B}(D_{e_B}\phi)-D_{\nabla_{e_B}e_B}\phi\\
&=\sD_{e_B}(\sD_{e_B}\phi)-\sD_{\snab_{e_B}e_B}\phi-D_{(\chi_{BB} L+\chibar_{BB}\Lbar)}\phi=(\sD\sD\phi)(e_B,e_B)-\frac{1}{r}D_{\Lbar}\phi+\frac{1}{r}D_L\phi,
}
and therefore
\EQ{\label{M eqn}
D_{\Lbar}((D_L+\frac{1}{r})\phi)=\sD^B\sD_B\phi-i\rho(F)\phi+\frac{1}{r}D_L\phi+\Lbar(\frac{1}{r})\phi=:M.
}
Now we claim that for $r\geq t/2,~t\geq1$ we have the bound
\EQ{\label{M-bound}
|M|\lesssim \taup^{-3}\taum^{-1/2}.
}
First we show how the claim proves the lemma. Assuming the {claim, from (\ref{M eqn})
\Aligns{
\nabla_{\Lbar}\left|(D_L+\frac{1}{r})\phi\right|\leq\left|D_{\Lbar}((D_L+\frac{1}{r})\phi)\right|\leq M.
}}
{Now given any point $P$ in our region of consideration, we integrate this equation along a straight line in $\{\ubar=\mathrm{constant}\}$ connecting $P$ to a point on the initial hyper surface $\Sigma_0.$ Since $\taup\sim\taum\sim r$ on $\Sigma_0,$ $(D_L+\frac{1}{r})\phi$ satisfies the desired decay on $\Sigma_0$ and therefore we get the desired bound on $(D_L+\frac{1}{r})\phi$ using the fundamental theorem of calculus.}  It remains to prove (\ref{M-bound}).  The following bounds follow directly from the assumptions of $F:$
\Aligns{
&\left|\rho(F)\phi\right|\lesssim \taup^{-3}\taum^{-1/2},\\
&\left|\Lbar(\frac{1}{r})\phi\right|\lesssim\taup^{-3}\taum^{-1/2}.
}
{Also from the definition of $Q_2(\phi)$ and in view of Lemmas \ref{Klainerman-Sobolev Lemma} and {\ref{KS commutator estimates 1}} $\taup^2\taum^{1/2}|D_L\phi|\lesssim1,$ and therefore}
\Aligns{
\left|\frac{1}{r}D_L\phi\right|\lesssim{\taup^{-3}\taup^{-1/2}.}
}
Finally for $\sD^B\sD_B\phi$ we observe that for appropriate constants $c^{ij}_B$ we can write $e_B=\frac{c^{ij}_B}{r}\Omega_{ij}$ and therefore
\Aligns{
|\sD_A\sD_B\phi|=\frac{c^{ij}_B}{r}|\sD_AD_{\Omega_{ij}}\phi|=\frac{c^{ij}_B}{r}|\sD_A\LL_{\Omega_{ij}}\phi|\lesssim\taup^{-3}\taum^{-1/2}Q_3(\phi).
}
Here to get the last inequality we have again used Lemmas \ref{Klainerman-Sobolev Lemma} and {\ref{KS commutator estimates 1}.}
\end{proof}
{\begin{remark}\label{L-char decay remark}
The proof of this lemma shows that {its} conclusion remains true if $D^\mu D_\mu\phi=G$ with $|G|\lesssim \taup^{-3}\taum^{-1/2}+\taup^{-5/2}\taum^{-3/2}.$ We claim that this estimate holds for $G=D^\mu D_\mu\LL^k\phi$ if we assume the following bootstrap decay rates for the components of $F$ (and its derivatives)
\Aligns{
|\alpha|\lesssim\taup^{-5/2},\quad |\rho|\lesssim \taup^{-2},\quad |\sigma|\lesssim\taup^{-2}\taum^{-1/2},\quad |\alphabar|\lesssim\taup\taum^{-3/2}.
}
To see this we refer the reader to Lemma \ref{kth commutator} above where the term $G$ is computed. The structure of this term for different choices of $\Gamma$ is worked out in equations (\ref{outside SP error S}-\ref{higher derivative extra term}) below. One can now use these equations and the bootstrap assumptions above to see that $G$ has the desired decay. 
\end{remark}}
\section{Estimates in $\{u\leq-1\}$}\label{Outside V_T}
\subsection{{Compactly supported scalar fields}}
In this first subsection we consider the case where $\phi$ is compactly supported. The general case is treated in the next {subsection. Recall} that $V_{T}$ is the spacetime domain enclosed by the outgoing cone $C_{-1}$ and the spacelike hypersurfaces $\Sigma_{0}$ and $\Sigma_T.$ In this section we use the phrase ``outside of $V_T$" to refer to the region $\{u\leq-1,~t\leq T\}$ (see Figure 1). Since $\phi(0,\cdot)$ is supported in {$\{r\leq 3/4\},$} by the finite speed propagation of the equation
\begin{align*}
D^{\alpha}D_{\alpha}\phi=0
\end{align*}
the scalar field $\phi$ vanishes outside $V_{T}.$
\begin{figure}[htb]
\centering
\begin{tikzpicture}[xscale=2, yscale=2]
\path[shade] (0,0) -- (1,0)--(1.65,0.65)--(0,0.65)--cycle;
\path[pattern= north west lines] (1.2,0.2)--(2.98,0.2)--(2.98,0)--(1,0)--cycle;
\path[pattern= north west lines] (1.65,0.65)--(2.98,0.65)--(2.98,0.39)--(1.39,0.39)--cycle;
\path[pattern= north west lines] (1.35,0.35)--(1.75,0)--(1,0)--cycle;
\path[pattern= north west lines] (2.6,0.65)--(2.98,0.65)--(2.98,0)--(2.6,0)--cycle;
\draw [thick] (0,1.25) -- (0,0) -- (3,0);
\node [right] at (3,0) {$\Sigma_0$};
\node [above] at (0,1.25) {$r=0$};
\draw [thick] (1,0) --(2,1);
\node [above right] at (2,1) {$C_{-1}$};
\draw [thick] (0,.65) --(3,.65);
\node [right] at (3,0.65) {$\Sigma_T$};
\draw [thick, decorate,decoration={brace,amplitude=10pt,mirror},xshift=0.4pt,yshift=-0.4pt](0,0) -- (0.98,0) node[black,midway,yshift=-0.6cm] {\footnotesize $\Sigma_0(T)$};
\draw [thick, decorate,decoration={brace,amplitude=10pt,mirror},xshift=0.4pt,yshift=-0.4pt](1.02,0) -- (2.98,0) node[black,midway,yshift=-0.6cm] {\footnotesize $\Sigma_0^O$};
\node at (0.6,0.295) {$V_T$};
\node at (2,0.295) {``Outside of $V_T$"};
\end{tikzpicture}
\caption{}
\label{fig: V_T}
\end{figure}
One of our goals in this section is to prove peeling estimates for $F$ and $\phi$ outside of $V_T.$ Moreover, we will see in the next section that in order to complete the energy estimates inside $V_{T}$, one needs to be given the data on $C_{-1}$. More precisely, the flux
\begin{align}\label{data on Cu}
\int_{C_{-1}}\taup^{2}|\alpha|^{2}+\taum^{2}(\rho^{2}+\sigma^{2})
\end{align}
needs to be bounded in terms of the initial data on $\Sigma_{0}$. Showing why this is true is the {second} goal of this section. Since the electric charge
\begin{align*}
e(t):=\lim_{r\rightarrow\infty}\int_{S_{t,r}}\frac{x^{i}}{r}E_{i}dS_{r},\quad E_{i}:=F_{0i}
\end{align*}
is non-zero, we cannot assume that the initial energy for {electric field $E_{i}$ 
\begin{align*}
\int_{\Sigma_{0}}(1+r^{2})|E|^{2}dx
\end{align*}
is finite.} This means that although outside $V_{T}$ equation \eqref{MKG} becomes the following homogeneous linear Maxwell equations
\begin{align}\label{LinearMaxwell}
\nabla_{\mu}F^{\mu\nu}=0,\quad \nabla_{\mu}({\star F})^{\mu\nu}=0,
\end{align}
one cannot use the multiplier $\overline{K}_{0}$ to get an estimate like \eqref{Morawetz-F1} and use it to derive an estimate for the flux on $C_{-1}$. To recover finiteness, instead of $F$ we will work with $\mathcal{L}_{\Omega}F$ which is charge free, where $\Omega$ is any $\Omega_{ij}.$ To see that $\LL_\Omega F$ is chargeless note that the charge associated to $\mathcal{L}_{\Omega}F$ is
\begin{align*}
\lim_{r\rightarrow\infty}\int_{S_{t,r}}\frac{x^{i}}{r}\big(\mathcal{L}_{\Omega}F\big)_{0i}dS_{r}&=\lim_{r\rightarrow\infty}\int_{S_{t,r}}\big(\mathcal{L}_{\Omega}F\big)_{0r}dS_{r}=\lim_{r\rightarrow\infty}\int_{S_{t,r}}\Omega\big(F_{0r}\big)dS_{r}\\
&=-\lim_{r\rightarrow\infty}\int_{S_{t,r}}\slashed{\textrm{div}}\Omega\cdot F_{0r}dS_{r}.
\end{align*}
Here $\slashed{\textrm{div}}$ represents the divergence operator on $S_{t,r}$. By the definition of $\Omega_{ij}:=x^{i}\partial_{j}-x^{j}\partial_{i}$, we have:
\begin{align*}
\textrm{div}_{\mathbb{R}^{3}}\Omega=0.
\end{align*}
On the other hand, since $\Omega$ is tangential to $S_{t,r}$, we have
\begin{align*}
\slashed{\textrm{div}}\Omega=\textrm{div}_{\mathbb{R}^{3}}\Omega=0
\end{align*}
which implies that $\mathcal{L}_{\Omega}F$ is charge free. Moreover by the conformal invariance of Maxwell's equations (see \cite{CK1}), $\mathcal{L}_{\Omega}F$ also satisfies the linear homogeneous Maxwell equations
\begin{align}\label{LinearMaxwell-chargefree}
\nabla^{\mu}\big(\mathcal{L}_{\Omega}F\big)_{\mu\nu}=0,\quad \nabla^{\mu}{\left(\star\left(\mathcal{L}_{\Omega}F\right)\right)}_{\mu\nu}=0.
\end{align}
It follows that $\nabla^\mu (Q(\LL_\Omega F)_{\mu\nu}\Kbar^\nu)=0,$ and applying the divergence theorem to this identity outside of $V_T$ (see Figure 1) yields the following energy estimate
\begin{align}\label{energyestimateschargefree}
&\int_{\Sigma_{t}^{O}}\big(\taup^{2}|\alpha(\mathcal{L}_{\Omega}F)|^{2}+\taum^{2}|\alphabar(\mathcal{L}_{\Omega}F)|^{2}+(\taup^{2}+\taum^{2})(\rho(\mathcal{L}_{\Omega}F)^{2}+\sigma(\mathcal{L}_{\Omega}F)^{2})\big)\\\notag
&+\int_{C_{-1}\bigcap\{0\leq t'\leq t\}}\taup^{2}|\alpha(\mathcal{L}_{\Omega}F)|^{2}+\taum^{2}\big(\rho(\mathcal{L}_{\Omega}F)^{2}+\sigma(\mathcal{L}_{\Omega}F)^{2}\big) \\\notag
&\leq C\int_{\Sigma_{0}^{O}}(1+r^{2})\big(|\alpha(\mathcal{L}_{\Omega}F)|^{2}+|\alphabar(\mathcal{L}_{\Omega}F)|^{2}+\rho(\mathcal{L}_{\Omega}F)^{2}+\sigma(\mathcal{L}_{\Omega}F)^{2}\big).
\end{align}
Here we are assuming that
\begin{align*}
\int_{\Sigma_{0}^{O}}(1+r^{2})\big(|E(\LL_{\Omega}F)|^{2}+|H(\LL_{\Omega}F)|^{2}\big)dx
\end{align*}
is finite, which is consistent with the chargelessness of $\LL_{\Omega}F.$
{
\begin{lemma}\label{flux estimate}
Under the assumptions of the Theorem \ref{main theorem}, (i)
\Aligns{
\sum_{j\leq k-1}Q_{out}(\LL^jF)(-1,T)+Q_{out}(\LL_\Omega\LL^{k-1}F)(-1,T)\leq C\epsilon.
}
\end{lemma}
}
\begin{proof}
Equation (\ref{energyestimateschargefree}), Lemma \ref{Poincare lemma} and Remark \ref{Poincare remark} give the following estimate on $C_{-1}$
\begin{align*}
&\int_{C_{-1}}\big(\taup^{2}|\alpha|^{2}+\taum^{2}(|\sigma|^{2}+|\rho-\overline{\rho}|^{2})\big)\\\notag
&\quad\quad \leq C\int_{\Sigma_{0}^{O}}(1+r^{2})\big(|\alpha(\mathcal{L}_{\Omega}F)|^{2}+|\alphabar(\mathcal{L}_{\Omega}F)|^{2}+\rho(\mathcal{L}_{\Omega}F)^{2}+\sigma(\mathcal{L}_{\Omega}F)^{2}\big).
\end{align*}
We still need to get an estimate for $\int_{C_{-1}}\taum^{2}\rho^{2}$. {Now}
\begin{align*}
\int_{C_{-1}}\taum^{2}|\rho|^{2}\leq C\int_{C_{-1}}\taum^{2}|\rho-\overline{\rho}|^{2}+C\int_{C_{-1}}\taum^{2}|\overline{\rho}|^{2}.
\end{align*}
{From the previous estimate and (\ref{energyestimateschargefree})} we already have the estimate for the first term on the right hand side, so we only consider (note that $\taum\sim1$ on $C_{-1}$)
\begin{align*}
\int_{C_{-1}}\taum^{2}|\overline{\rho}|^{2}\sim\int_{1}^{\infty}\int_{S_{t,r}}|\overline{\rho}|^{2}dS_{r}d\ubar.
\end{align*}
We first show that the integral $\int_{S_{t,r}}\rho dS_{r}$ is bounded by a constant independent of $t$ and $r$. Writing equations \eqref{LinearMaxwell} in the null frame we see that
\begin{align*}
-\slashed{\textrm{div}}\alpha-L\rho-2r^{-1}\rho=0.
\end{align*}
Integrating this equation on $S_{t,r}$ gives
\begin{align*}
\int_{\mathbb{S}^{2}}r^{2}\big(L\rho+2r^{-1}\rho\big)d\mu_{\mathbb{S}^{2}}=\int_{S_{t,r}}\big(L\rho+2r^{-1}\rho\big)dS_{r}=0,
\end{align*}
which means
\begin{align*}
L\Big(\int_{S_{t,r}}\rho dS_{r}\Big)=L\Big(\int_{\mathbb{S}^{2}}r^{2}\rho(t,\rho,\cdot)d\mu_{\mathbb{S}^{2}}\Big)=\int_{\mathbb{S}^{2}}L\big(r^{2}\rho(t,\rho,\cdot)\big)d\mu_{\mathbb{S}^{2}}=0.
\end{align*}
This implies that the integral $\int_{S_{t,r}}\rho dS_{r}$ is preserved along each outgoing cone $C_{u}$ 
\begin{align*}
\int_{S_{t,r}}\rho dS_{r}=\int_{S_{0,r-t}}\rho dS_{r-t}.
\end{align*}
Therefore we only need to show that on the initial surface $\Sigma_{0}$, the integral $\int_{S_{0,r}}\rho dS_{r}$ is bounded by a constant independent of $r$ which follows from
\begin{align*}
\left|\int_{S_{0,r}}\rho\right|=\left|\int_{\Sigma_0\cap\{|x|\leq r\}}\mathrm{div}E\right|\leq\left(\int_{\Sigma_0}|\phi|^2\right)^{1/2}\left(\int_{\Sigma_0}|D\phi|^2\right)^{1/2}\leq \epsilon.
\end{align*}
We have proved
\begin{align*}
|\overline{\rho}|\leq C\Big|\frac{1}{r^{2}}\int_{S_{t,r}}\rho dS_{r}\Big|\leq \epsilon r^{-2},
\end{align*}
which in turn implies
\begin{align*}
\int_{C_{-1}}\taum^{2}|\overline{\rho}|^{2}=\int_{1}^{\infty}\int_{S_{t,r}}|\overline{\rho}|^{2}dS_{r}d\ubar\leq \epsilon\int_{1}^{\infty}\frac{1}{r^{2}}d\ubar\leq \epsilon\int_{1}^{\infty}\frac{1}{\ubar^{2}}d\ubar< \epsilon.
\end{align*}
This completes the proof of the boundedness of the flux on $C_{-1}$ by a constant which can be made small by taking the initial data to be small. To bound the flux for $\LL^k F$ it suffices to note that by the conformal equivariance of Maxwell's equations $\LL^k F$ also solves
\begin{align*}
\nabla^{\mu}\big(\mathcal{L}^{k}F\big)_{\mu\nu}=0,\quad \nabla^{\mu}\big(\star \mathcal{L}^{k}F\big)_{\mu\nu}=0.
\end{align*}
This means that the analogue of (\ref{energyestimateschargefree}) holds for $\LL^k F$ and therefore we can use the same proof as above to estimate the flux for the derivatives of $F.$ {Since the last vector field we commute is a rotational vector field $\Omega,$ the term with the maximum number of derivatives is automatically chargeless so we do not need an extra derivative to estimate the average.}
\end{proof}
\subsection{{Non-compactly supported scalar fields}} {Here we remove the assumption of compact support from the scalar field $\phi.$ The proofs in this subsection are independent of the results of the previous subsection. The main difference between the compactly supported and general cases is that even if we replace $F$ by $\LL_\Omega F$ using the procedure from the previous subsection, since $F$ no longer satisfies the free Maxwell equations, we will encounter $\rho$ (not $\LL_\Omega\rho$) when commuting derivatives with the equation for $\phi.$ To deal with the slow decay of $\rho$ outside $V_T$ we need to assume more decay on the scalar field on the initial time slice. The following Morawetz estimate is a special case of the estimate proved in \cite{LS1}, and is used to show that the extra decay of $\phi$ is propagated in time. We remark that the double integral on the left hand side of the Morawetz estimate is an important ingredient without which we are not able to estimate the space-time integrals in our error analysis.}
{\begin{lemma}\label{fractional Morawetz lemma}
Suppose $\phi$ satisfies the equation
\begin{align}\label{phi inhomogeneous}
D^{\alpha}D_{\alpha}\phi=G.
\end{align}
Define
\begin{align}\label{bootstrap assumptions}
|F|_{L^{\infty}[0,t]}:=\sup_{\{0\leq t'\leq t\}\times \mathbb{R}^{3}\cap V_{T}^{c}}\left(\tau_{+}^{5/2}\tau_{-}^{-1/2}|\alpha|+{\tau_{+}^{7/4}\tau_{-}^{1/4}|\rho|}+\tau_{+}\tau_{-}|\underline{\alpha}|\right),
\end{align}
and let $w_{\gamma}(t,r):=\taum^{2\gamma}$ and $w'_{\gamma}(t,r):=\taum^{2\gamma-1}.$ Then with
\begin{align*}
\mathcal{E}^{\gamma}_{0}(t):& =\int_{\Sigma_{t}\cap V_{T}^{c}}\left(\taup^{2}|\left(D_{L}+\frac{1}{r}\right)\phi|^{2}+\taum^{2}|\left(D_{\Lbar}-\frac{1}{r}\right)\phi|^{2}+\taup^{2}|\slashed{D}\phi|^{2}+(\taup^{2}+\taum^{2})|\frac{\phi}{r}|^{2}\right)\cdot w_{\gamma}\\
& +\sup_{u\leq-1}\int_{C_{u}\cap \{0\leq t'\leq t\}}\left(\taup^{2}|\left(D_{L}+\frac{1}{r}\right)\phi|^{2}+\taum^{2}(|\slashed{D}\phi|^{2}+|\frac{\phi}{r}|^{2})\right)\cdot w_{\gamma}\\
& +\sup_{\ubar}\int_{\Cbar_{\ubar}\cap\{0\leq t'\leq t\}}\left(\taum^{2}|\left(D_{\Lbar}-\frac{1}{r}\right)\phi|^{2}+\taup^{2}(|\slashed{D}\phi|^{2}+|\frac{\phi}{r}|^{2})\right)\cdot w_{\gamma}\\
& +\iint_{V_{T}^{c}\cap \{0\leq t'\leq t\}}\left(\taup^{2}|\left(D_{L}+\frac{1}{r}\right)\phi|^{2}+\tau_{0}^{3/2}\left(\taum^{2}|\left(D_{\Lbar}-\frac{1}{r}\right)\phi|^{2}+\taup^{2}(|\slashed{D}\phi|^{2}+|\frac{\phi}{r}|^{2})\right)\right)\cdot w'_{\gamma},
\end{align*}
we have
\begin{align}\label{very final energy estimates}
& \mathcal{E}^{\gamma}_{0}(t)\lesssim_{\gamma}\int_{\Sigma_{0}}(1+r^{2})^{1+\gamma}|D\phi|^{2}\\\notag
& +\left(\|\taup\taum^{1/2}Gw_{\gamma}^{1/2}\|^{2}_{L^{2}(V_{T}^{c})}+|F|_{L^{\infty}[0,t]}\mathcal{E}^{\gamma}_{0}(t)\right).
\end{align}
\end{lemma}}
\begin{proof}[Proof. (based on \cite{LS1})]
Let us introduce the following weight function defined in $\{u\leq -3/4\}$:
\begin{align}\label{weight functions}
\widetilde{w}_{\gamma}(t,r):=1+(2-u)^{2\gamma}+(1+\ubar)^{-1/2}(2-u)^{2\gamma+1/2}
\end{align}
where $0<\gamma<\frac{1}{2}$. We also write $\tau_0=\taum/\taup.$ One can check directly that $\widetilde{w}_{\gamma}\sim w_{\gamma}$.  As we did in the domain $V_{T}$, we still work with the conformal metrics $\widetilde{g}_{\alpha\beta}=\dfrac{1}{\Omega^{2}}g_{\alpha\beta}$ in the domain $V_{T}^{c}=\{u\leq -1\}$. However, we shall make a slight modification on the multiplier such that the momentum density is defined as
\begin{align}\label{modified momentum}
\widetilde{P}_{\alpha}^{(\gamma)}[\phi]=\widetilde{Q}_{\alpha\beta}[\phi]\left(K_{0}\right)^{\beta}\cdot \widetilde{w}_{\gamma}.
\end{align} 
The divergence of this is calculated as follows:
\begin{align}\label{divergence of momentum}
\widetilde{\nabla}^{\alpha}\widetilde{P}_{\alpha}^{(\gamma)}[\phi]=& \Omega^{4}\Big(\Re\big({G}\cdot\overline{\frac{1}{\Omega}D_{K_{0}}(\Omega\phi)}\big)+K_{0}^{\beta}F_{\beta\gamma}\Im\big(\phi\overline{\frac{1}{\Omega}D^{\gamma}(\Omega\phi)}\big)\Big)\cdot\widetilde{w}_{\gamma}\\\notag
& -\frac{\Omega^{2}}{2}\widetilde{Q}_{\Lbar\beta}[\phi]K_{0}^{\beta}\cdot L(\widetilde{w}_{\gamma})-\frac{\Omega^{2}}{2}\widetilde{Q}_{L\beta}[\phi]K_{0}^{\beta}\cdot\Lbar(\widetilde{w}_{\gamma}).
\end{align}
Here we have already used the fact that $K_{0}$ is Killing with respect to the conformal metric $\widetilde{g}$. In order to deal with the two terms on the second line, we use the following expressions:
\begin{align}\label{expresson for extra term 1}
-L(\widetilde{w}_{\gamma}) =-\partial_{\ubar}(\widetilde{w}_{\gamma})=\frac{1}{2}(1+\ubar)^{-3/2}\cdot\left(2-u\right)^{2\gamma+1/2} \sim C_{\gamma}\tau_{0}^{3/2}w'_{\gamma},
\end{align}
\begin{align}\label{expression for extra term 2}
-\Lbar(\widetilde{w}_{\gamma}) & =-\partial_{u}(\widetilde{w}_{\gamma})\\\notag & =2\gamma(2-u)^{2\gamma-1}+(2\gamma+1/2)(1+\ubar)^{-1/2}(2-u)^{2\gamma-1/2}\\\notag
& \sim C_{\gamma}w'_{\gamma}.
\end{align}
Here the notation $f\sim g$ is used to mean that $f$ and $g$ have the same sign and have comparable sizes. With this same notation, it follows from the above expressions that
\begin{align}\label{positivity of double integration}
& -\frac{1}{2}\Omega^{-2}\widetilde{Q}_{\Lbar\beta}[\phi]K_{0}^{\beta}\cdot L\left(\widetilde{w}_{\gamma}\right)\sim C_{\gamma}\tau_{0}^{3/2}\left(u^{2}|\frac{1}{\Omega}D_{\Lbar}(\Omega\phi)|^{2}+\ubar^{2}|\slashed{D}\phi|^{2}\right)\cdot w'_{\gamma},\\
& -\frac{1}{2}\Omega^{-2}\widetilde{Q}_{L\beta}[\phi]K_{0}^{\beta}\cdot \Lbar\left(\widetilde{w}_{\gamma}\right)\sim C_{\gamma}\left(\ubar^{2}|\frac{1}{\Omega}D_{L}(\Omega\phi)|^{2}+u^{2}|\slashed{D}\phi|^{2}\right)\cdot w'_{\gamma}.
\end{align}
These two terms give the positive double integration on the left hand side of the energy estimates. More specifically, by the divergence theorem
\begin{align}\label{modified energy estimate-K0}
& \int_{\Sigma_{t}\cap V_{T}^{c}}\left(\ubar^{2}|\frac{1}{\Omega}D_{L}(\Omega\phi)|^{2}+u^{2}|\frac{1}{\Omega}(D_{\Lbar}(\Omega\phi)|^{2}+\ubar^{2}|\slashed{D}\phi|^{2}\right)\cdot w_{\gamma}\\\notag
& +\sup_{u\leq-1}\int_{C_{u}\cap \{0\leq t'\leq t\}}\left(\ubar^{2}|\frac{1}{\Omega}D_{L}(\Omega\phi)|^{2}+u^{2}|\slashed{D}\phi|^{2}\right)\cdot w_{\gamma}\\\notag
& +{\sup_{\ubar\geq0}}\int_{\Cbar_{\ubar}\cap\{0\leq t'\leq t\}}\left(u^{2}|\frac{1}{\Omega}D_{\Lbar}(\Omega\phi)|^{2}+\ubar^{2}|\slashed{D}\phi|^{2}\right)\cdot w_{\gamma}\\\notag
& +\iint_{V_{T}^{c}\cap \{0\leq t'\leq t\}}\left(\ubar^{2}|\frac{1}{\Omega}D_{L}(\Omega\phi)|^{2}+\tau_{0}^{3/2}\left(u^{2}|\frac{1}{\Omega}D_{\Lbar}(\Omega\phi)|^{2}+\ubar^{2}|\slashed{D}\phi|^{2}\right)\right)\cdot w'_{\gamma}\\\notag
& \lesssim \int_{\Sigma_{0}\cap V_{T}^{c}}(1+r^{2})^{1+\gamma}|\frac{1}{\Omega}D(\Omega\phi)|^{2}\\\notag
&\quad+\iint_{V_{T}^{c}\cap\{0\leq t'\leq t\}}|G|\cdot \left(\ubar^{2}|\frac{1}{\Omega}D_{L}(\Omega\phi)|+u^{2}|\frac{1}{\Omega}D_{\Lbar}(\Omega\phi)|\right)\cdot w_{\gamma}\\\notag
& \quad+\iint_{V_{T}^{c}\cap \{0\leq t'\leq t\}}\ubar^{3}\left(|\alpha|\cdot|\frac{\phi}{r}|\cdot|\slashed{D}\phi|+|\rho|\cdot|\frac{\phi}{r}|\cdot|\frac{1}{\Omega}D_{L}(\Omega\phi)|\right)\cdot w_{\gamma}\\\notag
& \quad+\iint_{V_{T}^{c}\cap \{0\leq t'\leq t\}}\ubar u^{2}\left(|\alphabar|\cdot|\frac{\phi}{r}|\cdot|\slashed{D}\phi|+|\rho|\cdot|\frac{\phi}{r}|\cdot|\frac{1}{\Omega}D_{\Lbar}(\Omega\phi)|\right)\cdot w_{\gamma}.
\end{align}
Applying \eqref{modified energy estimate-K0} to $\Omega=r$ and $\Omega=\ubar u,$ and using the relations
\begin{align}\label{relations for phi}
& \frac{-u\phi}{r}=(\ubar D_{L}\phi+2\phi)-\big(\ubar D_{L}\phi+\frac{\ubar}{r}\phi\big),\quad \frac{\ubar\phi}{r}=(uD_{\Lbar}\phi+2\phi)-\big(uD_{\Lbar}\phi-\frac{u}{r}\phi\big),\\\notag
& \ubar^{2}\big|\frac{1}{r}D_{L}(r\phi)\big|^{2}=\big|\ubar D_{L}\phi+\frac{\ubar\phi}{r}\big|^{2},\quad u^{2}\big|\frac{1}{r}D_{\Lbar}(r\phi)\big|^{2}=\big|uD_{\Lbar}\phi-\frac{u\phi}{r}\big|^{2},\\\notag
& \ubar^{2}\big|\frac{1}{u\ubar}D_{L}(u\ubar\phi)\big|^{2}=|\ubar D_{L}\phi+2\phi|^{2},\quad u^{2}\big|\frac{1}{u\ubar}D_{\Lbar}(u\ubar\phi)\big|^{2}=|uD_{\Lbar}\phi+2\phi|^{2},
\end{align}
we get
{\begin{align}\label{pre-final energy estimates}
& \mathcal{E}^{\gamma}_{0}(t)\lesssim_{\gamma}\int_{\Sigma_{0}\cap V_{T}^{c}}(1+r^{2})^{1+\gamma}\left(|D\phi|^{2}+|\frac{\phi}{r}|^{2}\right)\\\notag
& +\left({\|\taup\taum^{1/2}G w_\gamma^{1/2}\|_{L^{2}(V_{T}^{c})}}(\mathcal{E}^{\gamma}_{0}(t))^{1/2}+|F|_{L^{\infty}[0,t]}\mathcal{E}^{\gamma}_{0}(t)\right).
\end{align}}
Dividing \eqref{pre-final energy estimates} by $(\mathcal{E}^{\gamma}_{0}(t))^{1/2}$ and squaring we get
\begin{align}\label{final energy estimates}
& \mathcal{E}^{\gamma}_{0}(t)\lesssim_{\gamma}\int_{\Sigma_{0}\cap V_{T}^{c}}(1+r^{2})^{1+\gamma}\left(|D\phi|^{2}+|\frac{\phi}{r}|^{2}\right)\\\notag
& +\left({\|\taup\taum^{1/2}Gw_{\gamma}^{1/2}\|^{2}_{L^{2}(V_{T}^{c})}}+|F|_{L^{\infty}[0,t]}\mathcal{E}^{\gamma}_{0}(t)\right).
\end{align}
The lemma follows from this and the following Poincar\'e inequality proved in \cite{LS1}
\begin{align}\label{poincare on sigma0}
\int_{\Sigma_{0}}(1+r^{2})^{1+\gamma}\left|\frac{\phi}{r}\right|^{2}\lesssim\int_{\Sigma_{0}}(1+r^{2})^{1+\gamma}|D\phi|^{2}.
\end{align}
Indeed, using the inequality $\nabla|\phi|\lesssim|D\phi|$ and a standard density argument, we only need to show that the following inequality holds for any function $\psi\in C^{\infty}_{c}(\mathbb{R}^{3})$:
\begin{align*}
\int_{\mathbb{R}^{3}}(1+r^{2})^{1+\gamma}\left|\frac{\psi}{r}\right|^{2}\lesssim\int_{\mathbb{R}^{3}}(1+r^{2})^{1+\gamma}|\partial\psi|^{2}.
\end{align*}
This can be shown by using the following identity and the Cauchy-Schwarz inequality:
\begin{align*}
\partial_{r}\left((1+r)^{2\gamma+1}r^{2}\psi^{2}\right)
&=(2\gamma+1)(1+r)^{2\gamma}r^{2}\psi^{2}\\
&+2(1+r)^{2\gamma+1}r\psi^{2}
+2(1+r)^{2\gamma+1}r^{2}\psi\partial_{r}\psi.
\end{align*}
\end{proof}
We will also need a version of Lemma \ref{Klainerman-Sobolev Lemma} adapted to $V_{T}^{c}$. Let us consider the following smooth cut off function:
\begin{align*}
\chi(s):=
\begin{cases}
1 & s\leq -1\\
0 & s\geq -3/4.
\end{cases}
\end{align*}
Then we have the following Lemma.
\begin{lemma}\label{Klainerman-Sobolev}
Let $f(t,x)$ be a smooth function defined in $\{u\leq -3/4\}$, then for any $(t,x)\in\{u\leq-1\}$, we have:
\begin{align*}
|f(t,x)|\lesssim\taup^{-1}\taum^{-1/2}w_{\gamma}^{-1/2}\sum_{|\alpha|\leq 2}\|\left(\mathcal{L}_{\Gamma}^{\alpha}f\right)(t,\cdot)\cdot w^{1/2}_{\gamma}\|_{L^{2}(\Sigma_{t}\cap\{u\leq-3/4\})},
\end{align*}
{where the vector fields $\Gamma$ are chosen from $\mathbb{L}.$}
\end{lemma}
\begin{proof}
Let us introduce the function $\widetilde{f}(t,x):=\chi(u)f(t,x)w^{1/2}_{\gamma}$. By Lemma \ref{Klainerman-Sobolev Lemma}, for any $(t,x)\in\{u\leq-1\}$:
\begin{align}\label{classical Klainerman-Sobolev}
|\widetilde{f}(t,x)|\lesssim\taup^{-1}\taum^{-1/2}\sum_{|\alpha|\lesssim 2}\|\left(\mathcal{L}^{\alpha}_{\Gamma}\widetilde{f}\right)(t,\cdot)\|_{L^{2}(\Sigma_{t}\cap\{u\leq-3/4\})}.
\end{align}
Since $w_{\gamma}$ and $\chi(u)$ depend only on $u$, only the $\Lbar$ component of their derivative is nonzero. But for any $\Gamma$, the coefficient of $\Lbar$ is always controlled by $\taum$ and
\begin{align*}
|\left(\mathcal{L}_{\Gamma}\chi\right)(u)|\lesssim |\taum\cdot\chi'(u)|,\quad |\mathcal{L}_{\Gamma}w_{\gamma}|\lesssim |u|\taum^{2\gamma-1}\lesssim w_{\gamma}.
\end{align*} 
Since $\chi'(u)$ is supported on $u\in[-1,-3/4]$,
\begin{align*}
|\left(\mathcal{L}_{\Gamma}\chi\right)(u)|\lesssim |\chi'(u)|\lesssim 1,\quad |\mathcal{L}_{\Gamma}w^{1/2}_{\gamma}|\lesssim w^{1/2}_{\gamma}.
\end{align*}
As a result, the right hand side of \eqref{classical Klainerman-Sobolev} is bounded by
\begin{align*}
\taup^{-1}\taum^{-1/2}\sum_{|\alpha|\leq 2}\|\left(\mathcal{L}^{\alpha}_{\Gamma}f\right)(t,\cdot)w^{1/2}_{\gamma}\|_{L^{2}(\Sigma_{t}\cap\{u\leq-3/4\})},
\end{align*}
which implies the lemma.
\end{proof}
Let us define $\mathcal{E}_{k}^{\gamma}(t)$ to be the sum of corresponding energies of $\mathcal{L}_{\Gamma}^{\alpha}\phi$ with $|\alpha|\leq k$. If one can show $\mathcal{E}_{k}^{\gamma}(t)$ is bounded by the initial energies, then by Lemma \ref{Klainerman-Sobolev},
\begin{align}\label{Linfinity phi}
|D_{\Lbar}\phi-\frac{1}{r}\phi|\lesssim\taup^{-1}\taum^{-3/2}w_{\gamma}^{-1/2},\quad |\slashed{D}\phi|\lesssim\taup^{-2}\taum^{-1/2}w^{-1/2}_{\gamma},\quad |\phi|\lesssim\taup^{-1}\taum^{-1/2}w^{-1/2}_{\gamma}.
\end{align}
{Similarly the proof of Lemma \ref{L-char decay} implies the following decay rate for $D_{L}\phi+\dfrac{1}{r}\phi$ outside $V_{T}:$
\begin{align}\label{Linfinity DL phi}
|D_{L}\phi+\frac{1}{r}\phi|\lesssim\taup^{-5/2}w_{\gamma}^{-1/2}.
\end{align}}
The point-wise bounds for the components of $F$ in terms of the energy are obtained similarly.
\begin{lemma} 
Under the assumptions of Theorem \ref{main theorem}, and assuming that $\EE_k^\gamma$ and $Q_k^{*}$ are smaller than a sufficiently small $\epsilon$
\Aligns{
|\alpha|\lesssim\taup^{-5/2}{\epsilon},\quad|\alphabar|\lesssim \taup^{-1}\taum^{-3/2}{\epsilon},\quad|\sigma|\lesssim\taup^{-2}\taum^{-1/2}{\epsilon},\quad|\rho|\lesssim r^{-2}{\epsilon}.
}
Furthermore, inside $V_T$ the last estimate can be replaced by $|\rho|\lesssim \taup^{-2}\taum^{-1/2}{\epsilon}.$
\end{lemma}
\begin{proof}
For $\alphabar$ and $\sigma$ the estimates follow from the argument in \cite{CK1}. Indeed in view of Lemma \ref{Poincare lemma} and Remark \ref{Poincare remark} the energies for these components can be bounded by those of $\LL_\Omega\alphabar$ and $\LL_\Omega\sigma$ respectively. This puts us in the context of \cite{CK1} and we are able to use the Sobolev estimates there (in particular (3.54) and (3.56) in \cite{CK1}) to get the desired decay. To get the improved decay rate of $\taup^{-5/2}$ for $\alpha$ we need an additional argument. Taking the difference of the first two equations of (\ref{MKG}) with $\mu=L$ and $\nu=A$ gives
\begin{align*}
\nabla_{\Lbar}\alpha_{A}-r^{-1}\alpha_{A}-\slashed{\nabla}_{A}\rho
-\epsilon_{AB}\slashed{\nabla}_{B}\sigma=-\textrm{Im}(\phi\slashed{D}_{A}\phi).
\end{align*}
From this equation, we see that
\begin{align*}
\int_{\Sigma_{t}\bigcap\{r\geq 1+\frac{t}{2}\}}\Big(\sum_{|k|\leq 2}r^{2}|\alpha(\LL_{\Omega}^{k}F)|^{2}+\sum_{k\leq 1}r^{4}|\nabla_{\Lbar} \alpha(\LL^{k}_{\Omega}F)|^{2}\Big)dx\lesssim Q^{*}_{2}(\phi,F)(t).
\end{align*}
Applying Lemma 2.3 in \cite{CK1} to $U=r\alpha$, we have
\begin{align*}
|r\alpha|\lesssim r^{-3/2}Q_2^{*}(F,\phi).
\end{align*}
This proves estimate for $\alpha$ when $r\geq 1+\frac{t}{2}$. When $r\leq 1+\frac{r}{2}$, this estimate holds by Lemma \ref{Klainerman-Sobolev Lemma}. For $\rho$ we first observe that {the operator $\partial_{r}$ commutes with the average of $\rho$ on $S_{t,r}$, i.e. 
\begin{align*}
\partial_{r}\left(\frac{1}{r^{2}}\int_{S_{t,r}}\rho d\sigma_{S_{t,r}}\right)=\partial_{r}\left(\int_{\mathbb{S}^{2}}\rho d\sigma_{\mathbb{S}^{2}}\right)=\int_{\mathbb{S}^{2}}
\partial_{r}\rho d\sigma_{\mathbb{S}^{2}}=\frac{1}{r^{2}}\int_{S_{t,r}}\partial_{r}\rho d\sigma_{S_{t,r}}.
\end{align*}
Then we apply the second inequality of Lemma 2.3 in \cite{CK1} to 
$U=r(\rho-\overline{\rho})$ as well as \eqref{Poincare lemma} to get
\begin{align*}
|r(\rho-\overline{\rho})|\lesssim \taup^{-1}\taum^{-1/2}\left(\sum_{1\leq k\leq 2}\|r\LL_{\Omega}^{k}\rho\|_{L^{2}(\Sigma_{t})}+\|\taum(\partial_{r}\rho-\overline{\partial_{r}\rho})\|_{L^{2}(\Sigma_{t})}+\sum_{k=1}\|\taum\LL_{\Omega}\partial_{r}\rho\|_{L^{2}(\Sigma_{t})}\right).
\end{align*}
Since $\taum\partial_{r}\sim\Gamma$ and $[\Omega, \partial_{r}]=0$, and using Lemma \ref{Poincare lemma}, the right-hand-side of the above is bounded by
\begin{align*}
\taup^{-1}\taum^{-1/2}\left(\sum_{1\leq k\leq 2}\|r\LL_{\Omega}^{k}\rho\|_{L^{2}(\Sigma_{t})}+\|\LL_{\Gamma}\rho-\overline{\LL_{\Gamma}\rho}\|_{L^{2}(\Sigma_{t})}+\sum_{k=1}\|\LL_{\Omega}\LL_{\Gamma}\rho\|_{L^{2}(\Sigma_{t})}\right)\lesssim\taup^{-1}\taum^{-1/2}\epsilon.
\end{align*}
Now we have already seen that $|\con{\rho}|\lesssim\epsilon r^{-2}$, in the compactly supported case.} In the non-compactly supported case we note that the equation satisfied by $\rho$ is
\Aligns{
-\slashed{\mathrm{div}} \alpha-L\rho-\frac{2}{r}\rho=\frac{1}{2}\varepsilon_{LAB\alpha}\Im(\phi\overline{D^\alpha\phi})=:f.
}
To see this we write the MKG equations as
\Aligns{
\nabla_{[\lambda}F_{\mu\nu]}=0,\quad \nabla_{[\lambda}\star F_{\mu\nu]}=-\varepsilon_{\lambda\mu\nu\alpha}\Im(\phi \overline{D^\alpha \phi}),
}
and compute the $[LAB]$ component of the second equation above using the relations
\Aligns{
\alphabar(\star F)_A=-\varepsilon_{BA}\alphabar(F)^B,\quad \alpha(\star F)_A=\varepsilon_{BA}\alpha(F)^B,\quad \star F_{AB}=\rho.
}
See \cite{CK1} for more details. It follows that $$|f|\lesssim |\phi||(D_{L}+\frac{1}{r})\phi|.$$ Repeating the argument from the case where $\phi$ is compactly supported, we get
\Aligns{
\left|L\left(\int_{S_{t,r}}\rho dS_{t,r}\right)\right|\lesssim \int_{S_{t,r}}|\phi||(D_L+\frac{1}{r})\phi|dS_{t,r}.
}
Integrating this equation in the $\ubar$ direction gives
\Aligns{
\left|\int_{S_{t,r}}\rho dS_{t,r}\right|\lesssim\left|\int_{S_{0,t-r}}\rho dS_{0,t-r}\right|+\int_{C_u}|\phi||(D_L+\frac{1}{r})\phi|dC_u.
}
The first term is bounded by the initial data as in the case of compactly supported scalar fields, and the second term can be bounded by H\"older's inequality as
\Aligns{
\int_{C_u}|\phi||(D_L+\frac{1}{r})\phi|dC_u\lesssim\|\taup(D_L+\frac{1}{r})\phi\|_{L^2(C_u)}\|\frac{\phi}{r}\|_{L^2(C_u)}\lesssim\epsilon^2.
}
We also need to inductively consider the case where $F$ is differentiated $k$ times. This amounts to commuting Lie derivatives with the equation of $F$ and hence differentiating $f$ above. The exact form of the resulting right hand side is computed in (\ref{LGamma J}) and (\ref{LLJ kth commutator}-\ref{LLJ kth commutator error}) with $\mu=L$ below. In view of the cancellations (\ref{LLJ cancelations}) every term in (\ref{LLJ kth commutator}) below can be bounded as above, except $E.$ For $E$ we observe that since $\mu=L,$
\Aligns{
|E|\lesssim \taup|\LL^{k_1}\alpha||\LL^{k_2}\phi||\LL^{k_3}\phi|+\taum|\LL^{k_1}\rho||\LL^{k_2}\phi||\LL^{k_3}\phi|,
}
where at least one of $\LL^{k_2}\phi$ and $\LL^{k_3}\phi$ can be bounded in $L^\infty.$ Moreover, by induction the average of $\LL^{k_1}\rho$ is bounded by $\epsilon r^{-2}.$ It follows that from Lemma \ref{Poincare lemma} that (suppressing the Lie derivatives from the notation)
\Aligns{
\int_{C_u}|E|\lesssim&\int_{C_u}|\frac{\phi}{r}|^2\taum+\left(\int_{C_u}|\alpha|^2\taup^2\right)^{1/2}\left(\int_{C_u}|\frac{\phi}{r}|^2\taum^{-1}\right)^{1/2}\\
&+\left(\int_{C_u}|\LL_\Omega\rho|^2\taum^2\right)^{1/2}\left(\int_{C_u}|\frac{\phi}{r}|^2\taum^{-1}\right)^{1/2}\lesssim\epsilon^2.
}
Finally to get the improved decay of $\rho$ inside $V_T$ we introduce a smooth cut-off 
\begin{align*}
\chi(s):=
\begin{cases}
1 & s\geq -3/4\\
0 & s\leq -1.
\end{cases}
\end{align*}
Note that for $u\geq-3/4$ we have $\chi(t-r)\cdot\rho=\rho$. Applying Lemma 2.3 in \cite{CK1} to $r\chi(t-r)\cdot\rho$, we see that for any point $\{u\geq-3/4\}$
\begin{align*}
r|\rho(t,x)|=|r\chi(t-r)\cdot\rho(t,x)|
\lesssim\taup^{-1}\taum^{-1/2}(Q_2^{*}+\EE_2^\gamma).
\end{align*}
These give the desired decay estimates in $\{u\geq-3/4\},$ but since $\taum$ is like a constant for $u\in [-3/4,-1],$ the estimates in this case hold by the earlier arguments.
\end{proof}
\begin{proof}[Proof of Theorem \ref{main theorem} (part 1: outside $V_T$)]
{We use the same Morawetz estimates for $F,$ replacing $F$ by $\LL_\Omega F$ as before, with the difference that now we need to keep track of the error terms coming from the fact that the equation is no longer conformally Killing invariant. For the scalar field $\phi$ we use the modified Morawetz estimate of Lemma \ref{fractional Morawetz lemma}. In order to estimate the right hand side of the Morawetz estimates, we assume that the left hand sides are bounded by a small constant $\epsilon^{2}$ up to some time $t=T,$ and use a bootstrap argument to improve this bound. Once we close the bootstrap, the decay rates for the scalar field and $F$ follow from the previous lemma and the paragraph preceding it. The flux estimate on $C_{-1}$ follow as in Lemma \ref{flux estimate} and the bound on $\rhobar$ from the previous lemma. We now proceed to bound the error terms from our application of the Morawetz estimates as outlined above. For the scalar field by Lemma \ref{1st commutator}}
\begin{align}\label{one commutator}
D^{\mu}D_{\mu}\LL_{\Gamma}\phi=2\Gamma^{\nu}F_{\mu\nu}D^{\mu}\phi+\nabla^{\mu}\Gamma^{\nu}F_{\mu\nu}\phi+\Gamma^{\mu}J_{\mu}(\phi)\phi.
\end{align} 
If we commute more than one derivative, then according to Lemma \ref{kth commutator}
\begin{align*}
D^{\mu}D_{\mu}\LL^{k}_{X_{1}...X_{k}}\phi=G_{k},
\end{align*}
where
\begin{align}\label{multi commutators}
|G_{k}| & \lesssim \left| 2L^{|I_1|}_{X_{I_1}}X_{i_1}^\alpha \LL^{|I_2|}_{X_{I_2}}F_{\mu\alpha}D^\mu \LL^{|I_3|}_{X_{I_3}}\phi+ \nabla^\mu L^{|I_1|}_{X_{I_1}}X_{i_1}^\alpha \LL^{|I_2|}_{X_{I_2}}F_{\mu\alpha}\LL^{|I_3|}_{X_{I_3}}\phi\right|\\\notag
& + \left| L^{|I_1|}_{X_{I_1}}X_{i_1}^\alpha \nabla^\mu \LL^{|I_2|}_{X_{I_2}}F_{\mu\alpha}\LL^{|I_3|}_{X_{I_3}}\phi \right|+\left| L_{X_{I_1}}^{|I_1|}X_{i_1}^\alpha L_{X_{I_2}}^{|I_2|}X_{i_2}^\beta\LL_{X_{I_3}}^{|I_3|}F_{\mu\alpha}\LL_{X_{I_4}}^{|I_4|}F^\mu_{~\beta}\LL_{X_{I_5}}^{|I_5|}\phi \right|.
\end{align}
Since the first and second terms above are in exactly the same form as the right hand side of \eqref{one commutator}, we only give the details for the right hand side of \eqref{one commutator} discussing two cases: the electromagnetic fields receive more derivatives than the scalar fields and vice versa. However, for the second term we should also consider the contributions which arise from commuting derivatives with the equation for $F,$ which are present only if we commute more than one derivative. These terms are computed in (\ref{A2 error 1}) and (\ref{A2 error 2}) below. We start with the first term on the right hand side of \eqref{one commutator}. If $\Gamma=S$,
\begin{align}\label{outside SP error S}
|S^\nu F_{\mu\nu}D^\mu\phi|+|\nabla^\mu S^\nu F_{\mu\nu}\phi|&\leq\ubar (|\rho||D_L\phi|+|\alpha||\sD\phi|)+u(|\rho||D_{\Lbar}\phi|+|\alphabar||\sD\phi||)\\\notag
&~+(\nabla^LS^{\Lbar}-\nabla^{\Lbar}S^L)\rho\phi+\frac{1}{2}(\nabla^AS^B-\nabla^BS^A)F_{AB}\phi\\\notag
&\lesssim\taup(|\rho||\left(D_L+\frac{1}{r}\right)\phi|+|\alpha||\sD\phi|)+\taum(|\rho||\left(D_{\Lbar}-\frac{1}{r}\right)\phi|\\\notag
& +|\alphabar||\sD\phi|) +|\rho||\phi|.
\end{align}
If $\Gamma=\Omega_{ij}$, we need to exploit the structure of $2\Omega_{ij}^AF_{A\mu}D^\mu\phi+\nabla^\mu\Omega_{ij}^\nu F_{\mu\nu}\phi.$ Specifically we have
\Aligns{
2\Omega_{ij}^AF_{\mu A}D^\mu\phi+\nabla^\mu\Omega_{ij}^\nu F_{\mu\nu}\phi=&2\Omega_{ij}^AF_{BA}
D^B\phi-\Omega_{ij}^AF_{LA}D_{\Lbar}\phi-\Omega_{ij}^AF_{\Lbar A}D_L\phi\\
&+\frac{\Omega_{ij}^A}{r}F_{LA}\phi-\frac{\Omega_{ij}^A}{r}F_{\Lbar A}\phi+(2e_B(\Omega_{ij}^A)+\Omega_{ij}^C\SGamma_{CB}^A)F^B_{~A}\phi,
}
so
\Align{\label{outside SP error rotation}
|2\Omega_{ij}^AF_{A\mu}D^\mu\phi+\nabla^\mu\Omega_{ij}^\nu F_{\mu\nu}\phi|\lesssim&\taup\left(|\alpha||(D_{\Lbar}-\frac{1}{r})\phi|+|\alphabar||(D_L+\frac{1}{r})\phi||+|\sigma|(|\sD\phi|+\frac{|\phi|}{r})\right).
}
Finally, if $\Gamma=\Omega_{0i}$,
{\Aligns{
2\Omega_{0i}^\nu &F_{\mu\nu}D^\mu\phi+\nabla^\mu\Omega_{0i}^\nu F_{\mu\nu}\phi\\
&=2t\omega^A_{i}F_{BA}\sD^B\phi+t\omega_i^AF_{AL}D_{\Lbar}\phi+\omega_i^A F_{AL}\phi+\ubar\omega_iF_{AL}\sD^A\phi\\
&\quad+t\omega_i^AF_{A\Lbar}D_L\phi+\omega_i^AF_{A\Lbar}\phi-u\omega_iF_{A\Lbar}\sD^A\phi\\
&\quad-\frac{1}{2}\ubar\omega_iF_{\Lbar L}D_L\phi+\frac{1}{2}u\omega_{i}F_{L\Lbar}D_{\Lbar}\phi+\omega_iF_{L\Lbar}\phi,
}}
and therefore
\Aligns{
\left|2\Omega_{0i}^\nu F_{\mu\nu}D^\mu\phi\nabla^\mu\Omega_{0i}^\nu F_{\mu\nu}\phi\right|&\lesssim|\alpha||tD_{\Lbar}\phi+\phi|+|\alphabar||tD_L\phi+\phi|\\
&\quad+\left(\taup(|\sigma|+|\rho|+|\alpha|)+\taum|\alphabar|\right)(|\sD\phi|+\frac{|\phi|}{r})\\
&\quad+|\rho|(\taup|D_L\phi|+\taum|D_{\Lbar}\phi|).
}
So
\begin{align}\label{outside SP error boots}
\left|2\Omega_{0i}^\nu F_{\mu\nu}D^\mu\phi\nabla^\mu\Omega_{0i}^\nu F_{\mu\nu}\phi\right|&\lesssim|\alpha|\taup|D_{\Lbar}\phi-\frac{1}{r}\phi|+|\alphabar|\taum|D_L\phi+\frac{1}{r}\phi|\\\notag
&\quad+\left(\taup(|\sigma|+|\rho|+|\alpha|)+\taum|\alphabar|\right)(|\sD\phi|+\frac{|\phi|}{r})\\\notag
&\quad+|\rho|(\taup|D_L\phi+\frac{1}{r}\phi|+\taum|D_{\Lbar}\phi-\frac{1}{r}\phi|).
\end{align}
We also need to consider the contribution from commuting with $\Gamma=\partial_{\mu}$. Notice that  $\nabla^{\mu}\Gamma^{\nu}=0$, so we only need to consider the first term. For $\Gamma=\partial_{t}=\frac{1}{2}(L+\Lbar)$, we have
\begin{align*}
|2\Gamma^{\nu}F_{\mu\nu}D^{\mu}\phi|\lesssim|\rho|(|D_{L}\phi|+|D_{\Lbar}\phi|)+|\sD\phi|(|\alpha|+|\alphabar|),
\end{align*}
which has already been considered when we estimate $\EE_{1}^{1}(S)$. For $\Gamma=\partial_{i}$, by the second equation of \eqref{vf decompositions}
\begin{align*}
|2\Gamma^{\nu}F_{\mu\nu}D^{\mu}\phi|\lesssim|\rho|(|D_{L}\phi|+|D_{\Lbar}\phi|)+|\sD\phi|(|\alpha|+|\alphabar|)+|\alpha||D_{\Lbar}\phi|+|\alphabar||D_{L}\phi|.
\end{align*}
The terms $|\alpha||D_{\Lbar}\phi|+|\alphabar||D_{L}\phi|$ can be bounded by
\begin{align*}
|\alpha||(D_{\Lbar}-\frac{1}{r})\phi|+|\alphabar||(D_{L}+\frac{1}{r})\phi|+\frac{(|\alpha|+|\alphabar|)|\phi|}{r}
\end{align*}
which has better decay than the terms already considered. We treat the terms on the right hand sides of the expressions above separately.
\Aligns{
\|\taup^4\taum |\rho|^2|(D_L+\frac{1}{r})\phi|^2\wgamma\|_{L_1(\Vo)}&\leq\|r^2\rho\|_{L^\infty}^2\|\taum|(D_L+\frac{1}{r})\phi|^2\wgamma\|_{L^1(\Vo)}\\
&\lesssim\epsilon^2\|
\taum^2|(D_L+\frac{1}{r})\phi|^{2}\wgammap\|_{L^1(\Vo)}\lesssim\epsilon^4.
}
Here some more care is needed if $\rho$ is differentiated enough times that it cannot be bounded in $L^\infty.$ Suppose $\rho$ is replaced by $\LL^N\rho$ in the formula above. If $N$ is the maximum number of derivatives we are commuting, then $\LL^N=\LL_\Omega\LL^{N-1}$ so $\LL^N\rho=\rho(\LL_\Omega\LL^{N-1}F)$ can be placed in $L^2(C_u)$ with the the weight $\taum^2$ and we get 
\Aligns{
\|\taup^4\taum |\LL^N\rho|^2|(D_L+\frac{1}{r})\phi|^2\wgamma\|_{L^1(\Vo)}&\leq\|\taum\taup^{-1}|\LL^N \rho|^2\|_{L^1(\Vo)}\|\taup^5|(D_L+\frac{1}{r})\phi|^2\wgamma\|_{L^\infty}\\
&\lesssim\epsilon^2\int_{-\infty}^{-1}\taum^{-2}\|\taum\LL^N\rho\|_{L^2(C_u)}^2du\lesssim\epsilon^4.
}
If $N$ is less than the maximum number of derivatives but large enough that $\LL^N\rho$ cannot be placed in $L^\infty,$ then we write $\LL^N\rho=\LL^N\rho-\overline{\LL^N\rho}+\overline{\LL^N\rho}.$ The averaged term can be bounded in $L^\infty$ as explained above. For the difference, since we have at least one more $\LL_\Omega$ derivative at our disposal, we use the Poincar\'e inequality from Lemma \ref{Poincare lemma} to bound this term in $L^2(C_u)$ as above. Next
\begin{align*}
\|\taup^2\taum^3|\rho|^2|(D_{\Lbar}-\frac{1}{r})\phi|^2\wgamma\|_{L^1(\Vo)} &\lesssim\|\taup^{4}|\rho|^{2}\|_{L^{\infty}}\cdot\|\taup^{-2}\taum^{4}|\left(D_{\Lbar}-\frac{1}{r}\right)\phi|^{2}w'_{\gamma}\|_{L^{1}(\Vo)}\\
& \lesssim\epsilon^{2}\|\tau_{0}^{3/2}\taum^{2}|\left(D_{\Lbar}-\frac{1}{r}\right)\phi|^{2}w'_{\gamma}\|_{L^{1}(\Vo)}\lesssim \epsilon^{4},
\end{align*}
or
\begin{align*}
\|\taup^2\taum^3|{\LL_\Omega\rho}|^2|(D_{\Lbar}-\frac{1}{r})\phi|^2\wgamma\|_{L^1(\Vo)} &\lesssim\|\taup^{2}\taum^{3}|\left(D_{\Lbar}-\frac{1}{r}\right)\phi|^{2}w_{\gamma}\|_{L^{\infty}}\||{\LL_\Omega}\rho|^{2}\|_{L^1(\Vo)}\lesssim\epsilon^{4},
\end{align*}
and
\begin{align*}
&\|\taup^{2}\taum|\rho|^{2}|\phi|^{2}w_{\gamma}\|_{L^{1}(\Vo)}\lesssim\|\taup^{4}|\rho|^{2}\|_{L^{\infty}}\|\tau_{0}^{3/2}|\phi|^{2}w'_{\gamma}\|_{L^{1}(\Vo)}\lesssim\epsilon^{4},\\
&\|\taup^{2}\taum|{\LL_\Omega}\rho|^{2}|\phi|^{2}w_{\gamma}\|_{L^{1}(\Vo)}\lesssim\|\taup^{2}\taum|\phi|^{2}w_{\gamma}\|_{L^{\infty}}\||{\LL_\Omega}\rho|^{2}\|_{L^{1}(\Vo)}\lesssim\epsilon^{4}.
\end{align*}
Similarly,
\Aligns{
&\|\taup^4\taum|\alpha|^2|\sD\phi|^2\wgamma\|_{L^1(\Vo)}\leq\|\taup^2\taum^{1/2}|\alpha|\|_{L^\infty}^2\||\sD\phi|^2\wgamma\|_{L^1(\Vo)}\lesssim\epsilon^4,\\
&\|\taup^4\taum|\alpha|^2|\sD\phi|^2\wgamma\|_{L^1(\Vo)}\leq\||\alpha|^2\|_{L^1(\Vo)}\|\taup^4\taum|\sD\phi|^2\wgamma\|_{L^\infty}\lesssim\epsilon^4.
}
$\|\taup^2\taum^3|\alphabar|^2|\sD\phi|^2\wgamma\|_{L^1(\Vo)}$ can be estimated in a similar fashion.
\Aligns{
\|\taup^4\taum|\alphabar|^2|(D_L+\frac{1}{r})\phi|^2\wgamma\|_{L^1(\Vo)}\leq\|\taup^2\taum^3|\alphabar|^2\|_{L^\infty}\|\taum^{-2}|\taup(D_L+\frac{1}{r})\phi|^2\wgamma\|_{L^1(\Vo)}\lesssim\epsilon^4.
}
If there are more derivatives on $\alphabar$, we need a more delicate analysis. Denoting by $r_m(u)$ and $r_M(u)$ the smallest and largest radii on $C_u,$ respectively, we have
\Align{\label{first L4}
&\iint\taup^{4}\taum\left|\alphabar\right|^{2}\left|\left(D_{L}+\frac{1}{r}\right)\phi\right|^{2}w_{\gamma}dxdt \\
&=\int_{-\infty}^{-1}\int_{r_{m}(u)}^{r_{M}(u)}\int_{S_{t,r}}\taup^{4}\taum\left|\alphabar\right|^{2}\left|\left(D_{L}+\frac{1}{r}\right)\phi\right|^{2}dS_{t,r}drw_{\gamma}du\\
&\lesssim \int_{-\infty}^{-1}\int_{r_{m}(u)}^{r_{M}(u)}\left(\int_{S_{t,r}}r^{6}\left|\left(D_{L}+\frac{1}{r}\right)\phi\right|^{4}\right)^{1/2}\left(\int_{S_{t,r}}r^{2}\taum^{6}\left|\alphabar\right|^{4}\right)^{1/2}drw_{\gamma}\taum^{-2}du.
}
The inner integral
\begin{align*}
\int_{r_{m}(u)}^{r_{M}(u)}\left(\int_{S_{t,r}}r^{6}\left|\left(D_{L}+\frac{1}{r}\right)\phi\right|^{4}\right)^{1/2}\left(\int_{S_{t,r}}r^{2}\taum^{6}\left|\alphabar\right|^{4}\right)^{1/2}dr
\end{align*}
is bounded by
\begin{align*}
\int_{r_m(u)}^{r_M(u)}\left(\int_{S_{u,r}}|(D_L+\frac{1}{r})\phi|^4r^6\right)^{1/2}dr\cdot\sup_{r_m(u)\leq r\leq r_M(u)}\left(\int_{S_{u,r}}|\alphabar|^4r^2\taum^6\right)^{1/2}.
\end{align*}
By Lemma \ref{L4 Sobolev lemma},
{\begin{align*}
\Big(\int_{\widetilde{S}_{u,r}}r^{2}\taum^{6}|\alphabar|^{4}\Big)^{1/4}&\lesssim ``\textrm{initial data}"\\
&\quad +\Big(\int_{\Cbar_{\ubar}}\taum^{2}|\alphabar|^{2}+r^{2}\taum^{2}|\slashed{\nabla}|\alphabar||^{2}+\taum^{2}|\Lbar(\taum |\alphabar|)|^{2}\Big)^{1/2}\\
&\lesssim``\textrm{initial data}"+\Big(\int_{\Cbar_{\ubar}}\taum^{2}|\alphabar|^{2}+\taum^{2}|\LL\alphabar|^{2}\Big)^{1/2}\lesssim\epsilon.
\end{align*}}
For the other term we apply Lemma \ref{isoperimetric lemma} to $f=r^{3}|\left(D_{L}+\frac{1}{r}\right)\phi|^{2}$ to get
\begin{align*}
&\int_{r_{m}(u)}^{r_{M}(u)}\left(\int_{\widetilde{S}_{u,r}}r^{6}|\left(D_{L}+\frac{1}{r}\right)\phi|^{4}\right)^{1/2}dr\\
&\lesssim\int_{r_{m}(u)}^{r_{M}(u)}\left(\int_{\widetilde{S}_{u,r}}\left(r^{2}|\left(D_{L}+\frac{1}{r}\right)\phi|^{2}+r^{3}|\left(D_{L}+\frac{1}{r}\right)\phi||\sD \left(D_{L}+\frac{1}{r}\right)\phi|\right)\right)dr.
\end{align*}
Bounding the second term above as
\begin{align*}
|\sD\left(D_{L}+\frac{1}{r}\right)\phi|\sim \frac{1}{r}|D_{\Omega}\left(D_{L}+\frac{1}{r}\right)\phi|\sim|\frac{1}{r}|\left(D_{L}+\frac{1}{r}\right)\LL_{\Omega}\phi|+|\alpha||\phi|,
\end{align*}
we have
\begin{align*}
&\int_{r_{m}(u)}^{r_{M}(u)}\left(\int_{\widetilde{S}_{u,r}}r^{6}|\left(D_{L}+\frac{1}{r}\right)\phi|^{4}\right)^{1/2}dr\\
&\lesssim\int_{C_{u}}\left(r^{2}|\left(D_{L}+\frac{1}{r}\right)\phi|^{2}+r^{2}|\left(D_{L}+\frac{1}{r}\right)\LL_{\Omega}\phi|^{2}+r^{4}|\alpha|^{2}|\phi|^{2}\right)\\
&\lesssim\epsilon^{2}+\int_{C_{u}}\left(r^{4}|\alpha|^{2}|\phi|^{2}\right).
\end{align*}
Since now there are less derivatives on $\phi$, the last term above is bounded as:
\begin{align*}
\int_{C_{u}}\left(\taup^{2}\taum^{-1}|\alpha|^{2}\right)\lesssim\int_{C_{u}}\taup^{2}|\alpha|^{2}\lesssim\epsilon^{2}.
\end{align*}
The estimates for $\|\taup^4\taum|\alpha|^2|(D_{\Lbar}-\frac{1}{r})\phi|^2\wgamma\|_{L^1(\Vo)}$ are similar:
\begin{align*}
\|\taup^4\taum|\alpha|^2|(D_{\Lbar}-\frac{1}{r})\phi|^2\wgamma\|_{L^1(\Vo)}\lesssim \|\taup^{2}\taum^3\left|\left(D_{\Lbar}-\frac{1}{r}\right)\phi\right|^{2}w_{\gamma}\|_{L^{\infty}}
\int_{-\infty}^{-1}\taum^{-2}\|\taup^{2}|\alpha|^{2}\|_{L^{1}(C_{u})}du\lesssim\epsilon^{4}.
\end{align*}
If there are more derivatives hitting $\left(D_{\Lbar}-\frac{1}{r}\right)\phi$, we use $L^{4}$ estimates:
\begin{align*}
&\iint\taup^{4}\taum\left|\alpha\right|^{2}\left|\left(D_{\Lbar}-\frac{1}{r}\right)\phi\right|^{2}w_{\gamma}dxdt \\
&=\int_{-\infty}^{-1}\int_{r_{m}(u)}^{r_{M}(u)}\taup^{4}\taum\left|\alpha\right|^{2}\left|\left(D_{\Lbar}-\frac{1}{r}\right)\phi\right|^{2}dS_{t,r}drw_{\gamma}du\\
&\lesssim \int_{-\infty}^{-1}\int_{r_{m}(u)}^{r_{M}(u)}\left(\int_{S_{t,r}}r^{2}\taum^{6}\left|\left(D_{\Lbar}-\frac{1}{r}\right)\phi\right|^{4}\right)^{1/2}\left(\int_{S_{t,r}}r^{6}\left|\alpha\right|^{4}\right)^{1/2}dr\taum^{-2}\wgamma du.
\end{align*}
The rest of the argument is similar to the previous case. Next,
\Aligns{
&\|\taup^4\taum|\sigma|^2|\sD\phi|^2\wgamma\|_{L^1(\Vo)}\leq\|\taup^4\taum|\sigma|^2\|_{L^\infty}\||\sD\phi|^2\wgamma\|_{L^1(\Vo)}\lesssim\epsilon^4,\\
&\|\taup^4\taum|\sigma|^2|\sD\phi|^2\wgamma\|_{L^1(\Vo)}\leq\||\sigma|^2\|_{L^1(\Vo)}\|\taup^4\taum|\sD\phi|^2\wgamma\|_{L^\infty}\lesssim\epsilon^4.
}
The case where $\sD\phi$ is replaced by $\frac{\phi}{r}$ is treated in the same way. For the second term in \eqref{multi commutators} and the last term in \eqref{one commutator}, note that at least one of the two $|\phi|^2$ factors can be placed in $L^\infty.$
\Aligns{
&\|\taup^4\taum|\phi|^4|\sD\phi|^2\wgamma\|_{L^1(\Vo)}\leq\|\taup^4\taum^2|\phi|^4\|_{L^\infty}\|\taum^{-1}|\sD\phi|^2\wgamma\|_{L^1(\Vo)}\lesssim\epsilon^6.\\
&\|\taup^4\taum|\phi|^4|\sD\phi|^2\wgamma\|_{L^1(\Vo)}\leq\|\taup^6\taum^2|\phi|^2|\sD\phi|^2\|_{L^\infty}\left\|\taum^{-1}\left|\frac{\phi}{r}\right|^2\wgamma\right\|_{L^1(\Vo)}\lesssim\epsilon^6.
}
The estimates for $\|\taup^4\taum|\phi|^4|(D_L+\frac{1}{r})\phi|^2\wgamma\|_{L^1(\Vo)}$ and $\|\taup^2\taum^3|\phi|^4|(D_{\Lbar}-\frac{1}{r})\phi|^2\wgamma\|_{L^1(\Vo)}$ are similar. Next, we estimate the contribution from the last term in \eqref{multi commutators}. We keep in mind that since these terms arise only after commuting two or more derivatives $\phi$ can aways be bounded in $L^\infty.$ {Using} \eqref{vectorfield commutators}, we write these contributions as
\Align{\label{higher derivative extra term}
&\taup^2|\phi|(|\alpha|^2+|\alpha||\rho|+|\alpha||\sigma|+|\alpha||\alphabar|+|\sigma|^2)\\
&+\taup\taum|\phi|(|\rho|^2+|\alpha||\alphabar|+|\rho||\alphabar|+|\sigma||\alphabar|)+\taum^2|\phi||\alphabar|^2.
}
We have
{\begin{align*}
\|\taup^{6}\taum|\alpha|^{4}|\phi|^{2}w_{\gamma}\|_{L^{1}(\Vo)}&\lesssim\|\taup^{5}|\alpha|^{2}\|_{L^{\infty}}\|\taup\taum|\alpha|^{2}|\phi|^{2}\wgamma\|_{L^{1}(\Vo)}\\
&\lesssim \|\taup^{5}|\alpha|^{2}\|_{L^{\infty}}\|\taup^{2}\taum|\phi|^{2}\wgamma\|_{L^{\infty}}\|\taup^{-1}|\alpha|^{2}\|_{L^{1}(\Vo)}\lesssim\epsilon^{6}.
\end{align*}}
\begin{align*}
\|\taup^{6}\taum|\alpha|^{2}|\rho|^{2}|\phi|^{2}\wgamma\|_{L^{1}(\Vo)}&\lesssim\|\taup^{5}|\alpha|^{2}\|_{L^{\infty}}\|\taup^{2}\taum|\phi|^{2}\wgamma\|_{L^{\infty}}\|\taup^{-1}|{\LL_\Omega}\rho|^{2}\|_{L^{1}(\Vo)}\\
&\quad{+}\|\taup^{5}|\alpha|^{2}\|_{L^{\infty}}\|\taup^{4}|{\rhobar}|^{2}\|_{L^{\infty}}\|\taup^{-3}\taum|\phi|^{2}\wgamma\|_{L^{1}(\Vo)}\lesssim\epsilon^{6},\\
&\textrm{or}\\
&\lesssim\|\taup^{4}|\rho|^{2}\wgamma\|_{L^{\infty}}\|\taup^{2}\taum|\phi|^{2}\|_{L^{\infty}}{\||\alpha|^{2}\|_{L^{1}(\Vo)}}\lesssim\epsilon^{6}.
\end{align*}
{\begin{align*}
\|\taup^{6}\taum|\alpha|^{2}|\alphabar|^{2}|\phi|^{2}\wgamma\|_{L^{1}(\Vo)}&\lesssim\|\taup^{2}\taum|\phi|^{2}\wgamma\|_{L^{\infty}}\|\taup^4|\alphabar|^{2}|\alpha|^{2}\|_{L^{1}(\Vo)}\\
&\lesssim\epsilon^2\int_{-\infty}^{-1}\int_{r_{m}(u)}^{r_{M}(u)}\left(\int_{S_{u,r}}r^{6}|\alpha|^{4}\right)^{1/2}\left(\int_{S_{u,r}}r^{2}\taum^{6}|\alphabar|^{4}\right)^{1/2}dr\taum^{-3}du\\
&\lesssim\epsilon^6.
\end{align*}
\begin{align*}
\|\taup^{6}\taum|\sigma|^{4}|\phi|^{2}\wgamma\|_{L^{1}(\Vo)}&\lesssim\|\taup^{2}\taum|\phi|^{2}\wgamma\|_{L^{\infty}}\|\taup^{4}\taum|\sigma|^{2}\|_{L^{\infty}}\|\taum^{-1}|\sigma|^{2}\|_{L^{1}(\Vo)}\lesssim\epsilon^{6}.
\end{align*}}
\begin{align*}
\|\taup^{4}\taum^{3}|\rho|^{4}|\phi|^{2}\wgamma\|_{L^{1}(\Vo)}&\lesssim\|\taup^{2}\taum|\phi|^{2}\wgamma\|_{L^{\infty}}\|\taup^{4}|\rho|^{2}\|_{L^{\infty}}\|\taup^{-2}\taum^{2}|{\LL_\Omega}\rho|^{2}\|_{L^{1}(\Vo)}\\
&\quad{+\|\taup^{8}|\rho|^{2}|\rhobar|^2\|_{L^{\infty}}\|\taup^{-4}\taum^{3}|\phi|^{2}\wgamma\|_{L^{1}(\Vo)}\lesssim\epsilon^{4}\|\tau_{0}^{3/2}|\phi|^{2}w'_{\gamma}\|_{L^{1}(\Vo)}\lesssim\epsilon^{6}.}
\end{align*}
\begin{align*}
\|\taup^{4}\taum^{3}|\rho|^{2}|\alphabar|^{2}|\phi|^{2}\wgamma\|_{L^{1}(\Vo)}&\lesssim\|\taup^{2}\taum|\phi|^{2}\wgamma\|_{L^{\infty}}\|\taup^{4}|\rho|^{2}\|_{L^{\infty}}\|\taup^{-2}\taum^{2}|\alphabar|^{2}\|_{L^{1}(\Vo)}{\lesssim\epsilon^6}\\
&\textrm{or}\\
&\lesssim\|\taup^{2}\taum|\phi|^{2}\wgamma\|_{L^{\infty}}\|\taup^{2}\taum^{3}|\alphabar|^{2}\|_{L^{\infty}}\|\taum^{-1}|{\LL_\Omega}\rho|^{2}\|_{L^{1}(\Vo)}\\
&\quad{+\|\taup^{2}\taum|\phi|^{2}\wgamma\|_{L^{\infty}}\|\taup^{4}|\rho||\rhobar|\|_{L^{\infty}}\|\taup^{-2}\taum^{2}|\alphabar|^{2}\|_{L^{1}(\Vo)}\lesssim\epsilon^6.}
\end{align*}
\begin{align*}
&{\|\taup^{4}\taum^{3}|\sigma|^{2}|\alphabar|^{2}|\phi|^{2}\wgamma\|_{L^{1}(\Vo)}\lesssim\|\taup^{2}\taum|\phi|^{2}\wgamma\|_{L^{\infty}}\|\taup^{4}\taum|\sigma|^{2}\|_{L^{\infty}}\|\taup^{-2}\taum|\alphabar|^{2}\|_{L^{1}(\Vo)}\lesssim\epsilon^6.}\\
&\|\taup^{2}\taum^{5}|\alphabar|^{4}|\phi|^{2}\wgamma\|_{L^{1}(\Vo)}\lesssim\|\taup^{2}\taum|\phi|^{2}\wgamma\|_{L^{\infty}}\|\taup^{2}\taum^{3}|\alphabar|^{2}\|_{L^{\infty}}\|\taup^{-2}\taum|\alphabar|^{2}\|_{L^{1}(\Vo)}\lesssim{\epsilon^{6}.}
\end{align*}
The contributions of (\ref{A2 error 1}) and (\ref{A2 error 2}) are estimated using the following point-wise bounds:
\allowdisplaybreaks{\begin{align*}
&\taup^6\taum|\alpha|^2|\phi|^6\wgamma\lesssim\epsilon^6\taum^{-2}\wgamma^{-2}|\alpha|^2,\quad\quad\quad\quad\taup^6\taum|\alpha|^2|\phi|^6\wgamma\lesssim\epsilon^6\taup^{-1}\taum^{-1}\wgamma^{-1}|\frac{\phi}{r}|^2,\\
&\taup^4\taum^3|\alphabar|^2|\phi|^6\wgamma\lesssim\epsilon^6\taup^{-2}\wgamma^{-2}|\alphabar|^2,\quad\quad\quad\quad\taup^4\taum^3|\alphabar|^2|\phi|^6\wgamma\lesssim\epsilon^6\taum^{-2}\wgamma^{-1}|\frac{\phi}{r}|^2,\\
&\taup^6\taum|\sigma|^2|\phi|^6\wgamma\lesssim\epsilon^6\taum^{-2}\wgamma^{-2}|\sigma|^2,\quad\quad\quad\quad\taup^6\taum|\sigma|^2|\phi|^6\wgamma\lesssim\epsilon^6\taum^{-2}\wgamma^{-1}|\frac{\phi}{r}|^2,\\
&\taup^4\taum^3|\LL_\Omega\rho|^2|\phi|^6\wgamma\lesssim\taup^{-2}\wgamma^{-2}|\LL_\Omega\rho|^2,\hspace{0.9cm}\taup^4\taum^3|\rhobar|^2|\phi|^6\wgamma\lesssim\taup^{-2}\taum\wgamma^{-1}|\frac{\phi}{r}|^2.
\end{align*}}
{Finally note that the term with $|F|_{L^\infty[0,T]}$ in the statement of Lemma \ref{fractional Morawetz lemma} can be absorbed in the left hand side by our bootstrap assumptions. This completes the error analysis for the terms coming from the scalar field. It remains to estimate the error terms from the Morawetz estimate for $F.$ Except for the terms involving $\rho$ these terms are treated the same way inside and outside of $V_T$ because the other components of $F$ have the same decay. We have therefore chosen to postpone this analysis to the next section. These are equations (\ref{F errors 1}-\ref{F errors 4}) and (\ref{F errors 5}) below. Here we only provide alternative estimates for the terms involving $\rho$ in these equations.
\begin{align*}
&\|\taup^2\taum|\phi|^2|\rhobar|^2\|_{L^1(C_u)}\lesssim\|\taup^{-2}\taum|\phi|^{2}\|_{L^{1}(C_{u})}\lesssim \|\taum^{2}|\frac{\phi}{r}|^{2}w_{\gamma}\|_{L^{1}(C_{u})}\taum^{-1-2\gamma}\lesssim\epsilon^{4}\taum^{-1-2\gamma},\\
&\|\taum^{2}|\rhobar||\left(D_{\Lbar}-\frac{1}{r}\right)\phi||\phi|\|_{L^{1}({\Cbar}_{\ubar})}\lesssim\|\taum^{2}\taup^{-2}|\left(D_{\Lbar}-\frac{1}{r}\right)\phi||\phi|\|_{L^{1}({\Cbar}_{\ubar})}\\
&\quad\quad\quad\quad\quad\quad\quad\lesssim\||\left(D_{\Lbar}-\frac{1}{r}\right)\phi|\taum|\frac{\phi}{r}|\taup w_{\gamma}\|_{L^{1}({\Cbar}_{\ubar})}\taup^{-1-2\gamma}\lesssim\epsilon^{3}\taum^{-1-2\gamma},\\
&\|\taup^2|\rhobar||\left(D_{L}+\frac{1}{r}\right)\phi||\phi|\|_{L^{1}(C_{u})}\lesssim\||\left(D_{L}+\frac{1}{r}\right)\phi||\phi|\|_{L^{1}(C_{u})}\\
&\quad\quad\quad\quad\quad\quad\quad\lesssim\|\taup|\left(D_{L}+\frac{1}{r}\right)\phi||\frac{\taum\phi}{\taup}|w_{\gamma}\|_{L^{1}(C_{u})}\taum^{-1-2\gamma}\lesssim\epsilon^{3}\taum^{-1-2\gamma},\\
&\|\taup\taum^2|\rhobar||\alphabar||\phi|^2\|_{L^1({\Cbar}_{\ubar})}\lesssim\epsilon^2\|\taup^{-3/2}|\alphabar|\taum|\frac{\phi}{r}|\taup\|_{L^1({\Cbar}_{\ubar})}\lesssim\epsilon^4\taup^{-3/2},\\
&\|\taup^3|\rhobar||\phi|^{2}|\alpha|\|_{L^{1}(C_{u})}\lesssim\epsilon^2\|\taum^{-1/2}|\phi||\alpha|\|_{L^{1}(C_{u})}\lesssim\epsilon^2\|\taum\frac{|\phi|}{r}\taup|\alpha|w_{\gamma}^{1/2}\|_{L^{1}(C_{u})}\taum^{-3/2-\gamma}\lesssim\epsilon^{4}\taum^{-3/2-\gamma}.
\end{align*}}
\end{proof}
\section{Inside $V_T$}\label{Inside V_T}
{In this section we complete the proof of Theorem \ref{main theorem}. In view of the results from the last section we only need to prove the decay rates in the interior of $C_{-1}.$}
\begin{proof}[Proof of Theorem \ref{main theorem} inside $V_{T}$]
{We set up a continuity argument as follows. Assume that $Q_k^{*}(T)\leq C\epsilon.$ We then show that if $\epsilon$ is sufficiently small, then $Q_k^{*}(T)\leq C\epsilon/2.$ The proof is divided into three steps corresponding to commuting zero, one, or more vector fields.}
\subsection*{Step 1: $Q_0^{*}(T)\leq {C\epsilon/2}$}
The bound on $Q_0^{*}(T)$ simply follows from the flux estimates on $C_{-1}$ from the previous section and Corollary \ref{div theorem}, and by taking the initial data sufficiently small.
\subsection*{Step 2: $Q_1^{*}(T)\leq {C\epsilon/2}$}
{In light of  (\ref{[L,equations]}), the flux estimates on $C_{-1}$, and Corollary \ref{div theorem} we have
\begin{align}{}
Q_1^{*}(T)^2&\leq C\epsilon_0^2\nonumber\\
&+C\sum_{\Gamma}\int_{V_T}|\frac{1}{\Omega}D_{\overline{K}_{0}}(\Omega\LL_{\Gamma}\phi)||2\Gamma^\nu F_{\mu\nu} D^\mu\phi+\nabla^\mu\Gamma^\nu F_{\mu\nu}\phi|\label{one}\\
&+C\sum_\Gamma\int_{V_T}|\frac{1}{\Omega}D_{\overline{K}_{0}}(\Omega\LL_{\Gamma}\phi)||\Gamma^\mu J_\mu(\phi)||\phi|\label{two}\\
&+C\sum_\Gamma\int_{V_T}|\Kbar^\mu F_{\mu\nu} J^\nu(\LL_\Gamma\phi)|\label{three}\\
&+C\sum_{\Gamma}\int_{V_T}|\Kbar^\mu\LL_\Gamma F_{\mu\nu}\LL_\Gamma J^\nu(\phi)|\label{four}\\
&=:C\epsilon_0^2+\sum_\Gamma\sum_{i=1}^4
\EE_i^1(\Gamma).\nonumber&&
\end{align}}
 For each $i$ and $\Gamma$ we will prove the bound
\[\EE_i^1(\Gamma)\leq C \epsilon^3 .\]
\subsubsection*{$\EE_1^1(S)$} 
{We have
\Align{\label{LKbar LS psi}
|\frac{1}{r}D_{\Kbar}(r\LL_S\phi)|\lesssim\taup^2|(D_L+\frac{1}{r})\LL_S\phi|+\taum^2|(D_{\Lbar}-\frac{1}{r})\LL_S\phi|.
}}
Now we need to bound the integrals of the products of the terms on the right hand side of the previous expression with \eqref{outside SP error S}.
\begin{align*}\tag{\ref{one}.1}
\int_{V_T}\taup^3|(D_L+\frac{1}{r})\LL_S\phi|&|\rho||D_L\phi|\\
&\leq\int_{-1}^\infty\|\taup|(D_L+\frac{1}{r})\LL_S\phi|\|_{L^2(C_u)}\|\taup|(D_L+\frac{1}{r})\phi|\|_{L^2(C_u)}\|\taup|\rho|\|_{L^\infty(C_u)}du\\
&~~+\int_{-1}^\infty\|\taup|(D_L+\frac{1}{r})\LL_S\phi|\|_{L^2(C_u)}\|\taum\frac{|\phi|}{r}\|_{L^2(C_u)}\|\taum^{-1}\taup^2|\rho|\|_{L^\infty(C_u)}du\\
&\lesssim \epsilon^3\int_{-1}^\infty\taum^{-3/2}du\lesssim\epsilon^3.
\end{align*}
\begin{align*}\tag{\ref{one}.2}
\int_{V_T}\taup^3&|(D_L+\frac{1}{r})\LL_S\phi||\alpha||\sD\phi|\\
&\leq\int_{-1}^\infty\|\taup|(D_L+\frac{1}{r})\LL_S\phi|\|_{L^2(C_u)}\|\taum|\sD\phi|\|_{L^2(C_u)}\|\taum^{-1}\taup^2|\alpha|\|_{L^\infty(C_u)}du\lesssim\epsilon^3.
\end{align*}
\begin{align*}\tag{\ref{one}.3}
\int_{V_T}\taup^2\taum&|(D_L+\frac{1}{r})\LL_S\phi||\alphabar||\sD\phi|\\
&\leq\int_{-1}^\infty\|\taup|(D_L+\frac{1}{r})\LL_S\phi|\|_{L^2(C_u)}\|\taum|\sD\phi|\|_{L^2(C_u)}\|\taup|\alphabar|\|_{L^\infty(C_u)}du\lesssim\epsilon^3.
\end{align*}
\begin{align*}\tag{\ref{one}.4}
\int_{V_T}\taup^2\taum&|(D_L+\frac{1}{r})\LL_S\phi||\rho||D_{\Lbar}\phi|\\
&\leq\int_{-1}^\infty\|\taup|(D_L+\frac{1}{r})\LL_S\phi|\|_{L^2(C_u)}\|\taum|\rho|\|_{L^2(C_u)}\|\taup|D_{\Lbar}\phi|\|_{L^\infty(C_u)}du\lesssim\epsilon^3,
\end{align*}
\begin{align*}\tag{\ref{one}.5}
\int_{V_T}\taum^{2}\taup&|(D_{\Lbar}-\frac{1}{r})\LL_S\phi||\rho||D_L\phi|\\
&\leq\int_{0}^\infty\|\taum|(D_{\Lbar}-\frac{1}{r})\LL_S\phi|\|_{L^2({\Cbar}_{\ubar})}\|\taup|\rho|\|_{L^2({\Cbar}_{\ubar})}\|\taum|D_L\phi|\|_{L^\infty({\Cbar}_{\ubar})}d\ubar\lesssim\epsilon^3.
\end{align*}
\begin{align*}\tag{\ref{one}.6}
\int_{V_T}\taum^2\taup&|(D_{\Lbar}-\frac{1}{r})\LL_S\phi||\alpha||\sD\phi|\\
&\leq\int_{0}^\infty\|\taum|(D_{\Lbar}-\frac{1}{r})\LL_S\phi|\|_{L^2({\Cbar}_{\ubar})}\|\taup|\sD\phi|\|_{L^2({\Cbar}_{\ubar})}\|\taum|\alpha|\|_{L^\infty({\Cbar}_{\ubar})}d\ubar\lesssim\epsilon^3.
\end{align*}
\begin{align*}\tag{\ref{one}.7}
\int_{V_T}\taum^3&|(D_{\Lbar}-\frac{1}{r})\LL_S\phi||\alphabar||\sD\phi|\\
&\leq\int_{0}^\infty\|\taum|(D_{\Lbar}-\frac{1}{r})\LL_S\phi|\|_{L^2({\Cbar}_{\ubar})}\|\taup|\sD\phi|\|_{L^2({\Cbar}_{\ubar})}\|\taup^{-1}\taum^2|\alphabar|\|_{L^\infty({\Cbar}_{\ubar})}d\ubar\lesssim\epsilon^3.
\end{align*}
\begin{align*}\tag{\ref{one}.8}
\int_{V_T}\taum^3&|(D_{\Lbar}-\frac{1}{r})\LL_S\phi||\rho||D_{\Lbar}\phi|\\
&\leq\int_{0}^{\infty}\|\taum(D_{\Lbar}-\frac{1}{r})\LL_{S}\phi\|_{L^{2}(\Cbar_{\ubar})}\|\taup\rho\|_{L^{2}(\Cbar_{\ubar})}\|\taup^{-1}\taum^{2}D_{\Lbar}\phi\|_{L^{\infty}}d\ubar\lesssim\epsilon^{3}.
\end{align*}
\subsubsection*{$\EE_1^1(\Omega_{ij})$} As before, we can bound
\Align{\label{LKbar LOmegaij psi}
|\frac{1}{r}D_{\Kbar}(r\LL_{\Omega_{ij}}\phi)|\lesssim\taup^2|(D_L+\frac{1}{r})\LL_{\Omega_{ij}}\phi|+\taum^2|(D_{\Lbar}-\frac{1}{r})\LL_{\Omega_{ij}}\phi|.
}
Now we have to estimate the integral of the product of this right hand side with \eqref{outside SP error rotation}.
{\begin{align*}\tag{\ref{one}.9}
\int_{V_T}\taup^3&|(D_L+\frac{1}{r})\LL_{\Omega_{ij}}\phi||\alpha||(D_{\Lbar}-\frac{1}{r})\phi|\\
&\leq\int_{-1}^\infty\|\taup|(D_L+\frac{1}{r})\LL_{\Omega_{ij}}\phi|\|_{L^2(C_u)}\|\taup|\alpha|\|_{L^2(C_u)}\|\taup(D_{\Lbar}-\frac{1}{r})\phi|\|_{L^\infty}du\lesssim\epsilon^3.
\end{align*}}
\begin{align*}\tag{\ref{one}.10}
\int_{V_T}\taup^3&|(D_L+\frac{1}{r})\LL_{\Omega_{ij}}\phi||\alphabar||(D_{L}+\frac{1}{r})\phi|\\
&\leq\int_{-1}^\infty\|\taup|(D_L+\frac{1}{r})\LL_{\Omega_{ij}}\phi|\|_{L^2(C_u)}\|\taup|(D_L+\frac{1}{r})\phi|\|_{L^2(C_u)}\|\taup|\alphabar|\|_{L^\infty(C_u)}du\lesssim\epsilon^3.
\end{align*}
\begin{align*}\tag{\ref{one}.11}
\int_{V_T}\taup^3&|(D_L+\frac{1}{r})\LL_{\Omega_{ij}}\phi||\sigma|(|\sD\phi|+\frac{|\phi|}{r})\\
&\leq\int_{-1}^\infty\|\taup|(D_L+\frac{1}{r})\LL_{\Omega_{ij}}\phi|\|_{L^2(C_u)}\|\taum(|\sD\phi|+\frac{|\phi|}{r})\|_{L^2(C_u)}\|\taum^{-1}\taup^2|\sigma|\|_{L^\infty(C_u)}du\lesssim\epsilon^3.
\end{align*}
\begin{align*}\tag{\ref{one}.12}
\int_{V_T}\taum^2\taup&|(D_{\Lbar}-\frac{1}{r})\LL_{\Omega_{ij}}\phi||\alpha||(D_{\Lbar}-\frac{1}{r})\phi|\\
&\leq\int_{0}^\infty\|\taum|(D_{\Lbar}-\frac{1}{r})\LL_{\Omega_{ij}}\phi|\|_{L^2({\Cbar}_{\ubar})}\|\taum|(D_{\Lbar}-\frac{1}{r})\phi|\|_{L^2({\Cbar}_{\ubar})}\|\taup|\alpha|\|_{L^\infty({\Cbar}_{\ubar})}d\ubar\lesssim\epsilon^3.
\end{align*}
\begin{align*}\tag{\ref{one}.13}
\int_{V_T}\taum^2\taup&|(D_{\Lbar}-\frac{1}{r})\LL_{\Omega_{ij}}\phi||\alphabar||(D_{L}+\frac{1}{r})\phi|\\
&\leq\int_{0}^\infty\|\taum|(D_{\Lbar}-\frac{1}{r})\LL_{\Omega_{ij}}\phi|\|_{L^2({\Cbar}^i_{\ubar})}\|\taum|\alphabar|\|_{L^2({\Cbar}^i_{\ubar})}\|\taup|(D_L+\frac{1}{r})\phi|\|_{L^\infty({\Cbar}^i_{\ubar})}d\ubar\lesssim\epsilon^3.
\end{align*}
\begin{align*}\tag{\ref{one}.14}
\int_{V_T}\taum^2\taup&|(D_{\Lbar}-\frac{1}{r})\LL_{\Omega_{ij}}\phi||\sigma|(|\sD\phi|+\frac{|\phi|}{r})\\
&\leq\int_{0}^\infty\|\taum|(D_{\Lbar}-\frac{1}{r})\LL_{\Omega_{ij}}\phi|\|_{L^2({\Cbar}_{\ubar})}\|\taup(|\sD\phi|+\frac{|\phi|}{r})\|_{L^2({\Cbar}_{\ubar})}\|\taum|\sigma|\|_{L^\infty({\Cbar}_{\ubar})}d\ubar\lesssim\epsilon^3.
\end{align*}
As in the previous section (see \eqref{outside SP error boots}), the other vectorfields do not contribute new terms.
\subsubsection*{$\EE_2^1(S)$} Observe that
\Aligns{
S^\mu J_\mu(\phi)&=S^L\Im(\phibar D_L\phi)+S^{\Lbar}\Im(\phibar D_{\Lbar}\phi)\\
&=\frac{1}{2}\ubar \Im(\phibar(D_L+\frac{1}{r})\phi)+\frac{1}{2}u\Im(\phibar(D_{\Lbar}-\frac{1}{r})\phi).
}
We must consider the product of this expression with (\ref{LKbar LS psi}).
\begin{align*}\tag{\ref{two}.1}
\int_{V_T}\taup^3&|(D_L+\frac{1}{r})\LL_{S}\phi||\phi|^2|(D_{L}+\frac{1}{r})\phi|\\
&\leq\int_{-1}^\infty\|\taup|(D_L+\frac{1}{r})\LL_{S}\phi|\|_{L^2(C_u)}\|\taup|(D_L+\frac{1}{r})\phi|\|_{L^2(C_u)}\|\taup|\phi|^2\|_{L^\infty(C_u)}du\lesssim\epsilon^3.
\end{align*}
\begin{align*}\tag{\ref{two}.2}
\int_{V_T}\taup^2\taum&|(D_L+\frac{1}{r})\LL_{S}\phi||\phi|^2|\left(D_{\Lbar}-\frac{1}{r}\right)\phi|\\
&\leq\int_{-1}^\infty\|\taup|(D_L+\frac{1}{r})\LL_{S}\phi|\|_{L^2(C_u)}\|\frac{\taum}{r}|\phi|\|_{L^2(C_u)}\|\taup^2|\phi||\left(D_{\Lbar}-\frac{1}{r}\right)\phi|\|_{L^\infty(C_u)}du\lesssim\epsilon^3.
\end{align*}
\begin{align*}\tag{\ref{two}.3}
\int_{V_T}\taum^2\taup&|(D_{\Lbar}-\frac{1}{r})\LL_{S}\phi||\phi|^2|(D_{L}+\frac{1}{r})\phi|\\
&\leq\int_{0}^\infty\|\taum|(D_{\Lbar}-\frac{1}{r})\LL_{S}\phi|\|_{L^2({\Cbar}_{\ubar})}\|\frac{\taup}{r}|\phi|\|_{L^2({\Cbar}_{\ubar})}\|\taum\taup|\phi||(D_L+\frac{1}{r})\phi|\|_{L^\infty({\Cbar}_{\ubar})}d\ubar\lesssim\epsilon^3.
\end{align*}
\begin{align*}\tag{\ref{two}.4}
\int_{V_T}\taum^3&|(D_{\Lbar}-\frac{1}{r})\LL_{S}\phi||\phi|^2|(D_{\Lbar}-\frac{1}{r})\phi|\\
&\leq\int_{0}^\infty\|\taum|(D_{\Lbar}-\frac{1}{r})\LL_{S}\phi|\|_{L^2({\Cbar}_{\ubar})}\|\taum|(D_{\Lbar}-\frac{1}{r})\phi|\|_{L^2({\Cbar}_{\ubar})}\|\taum|\phi|^2|\|_{L^\infty({\Cbar}_{\ubar})}d\ubar\lesssim\epsilon^3.
\end{align*}
\subsubsection*{$\EE_2^1(\Omega_{ij})$} Since $\Omega_{ij}^\mu J_\mu(\phi)=\Omega_{ij}^A\Im(\phibar D_A\phi),$
\Aligns{
|\Omega_{ij}^\mu J_\mu(\phi)||\phi|\lesssim \taup|\phi|^2|\sD\phi|.
}
We need to consider the product of this expression with (\ref{LKbar LOmegaij psi}).
\begin{align*}\tag{\ref{two}.5}
\int_{V_T}\taup^3&|(D_L+\frac{1}{r})\LL_{\Omega_{ij}}\phi||\phi|^2|\sD\phi|\\
&\leq\int_{-1}^\infty\|\taup|(D_L+\frac{1}{r})\LL_{\Omega_{ij}}\phi|\|_{L^2(C_u)}\|\taum|\sD\phi||\|_{L^2(C_u)}\|\taup^2\taum^{-1}|\phi|^2\|_{L^\infty(C_u)}du\lesssim\epsilon^3.\\
\end{align*}
\begin{align*}\tag{\ref{two}.6}
\int_{V_T}\taum^2\taup&|(D_{\Lbar}-\frac{1}{r})\LL_{\Omega_{ij}}\phi||\phi|^2|\sD\phi|\\
&\leq\int_{0}^\infty\|\taum|(D_{\Lbar}-\frac{1}{r})\LL_{\Omega_{ij}}\phi|\|_{L^2({\Cbar}_{\ubar})}\|\taup|\sD\phi|\|_{L^2({\Cbar}_{\ubar})}\|\taum|\phi|^2\|_{L^\infty({\Cbar}_{\ubar})}d\ubar\lesssim\epsilon^3.\\
\end{align*}
\subsubsection*{$\EE_2^1(\Omega_{0i})$} Here
{\Aligns{
\Omega_{0i}^\mu J_\mu(\phi)&=\frac{\omega_i}{2}\ubar\Im(\phibar D_L\phi)-\frac{\omega_i}{2}u\Im(\phibar D_{\Lbar}\phi)+t\omega_i^A \Im(\overline{\phi}\sD_A\phi)\\
&=\frac{\omega_i}{2}\ubar\Im(\phibar(D_L+\frac{1}{r})\phi)-\frac{\omega_i}{2}u\Im(\phibar(D_{\Lbar}-\frac{1}{r})\phi)+t\omega_i^A \Im(\overline{\phi}\sD_A\phi),
}}
\Aligns{
\Omega_{0i}^\mu J_\mu(\phi)&=\frac{\omega_i}{2}\ubar\Im(\phibar D_L\phi)-\frac{\omega_i}{2}u\Im(\phibar D_{\Lbar}\phi)+t\omega_i^A\sD_A\phi\\
&=\frac{\omega_i}{w}\ubar\Im(\phibar(D_L+\frac{1}{r})\phi)-\frac{\omega_i}{2}u\Im(\phibar(D_{\Lbar}-\frac{1}{r})\phi)+t\omega_i^A\sD_A\phi.
}
But the contributions of all these terms were already considered while estimating $\EE_2^1(S)$ and $\EE_2^1(\Omega_{ij}).$ {When $\Gamma=\partial_{\mu}$, by \eqref{vf decompositions} 
\begin{align*}
|\Gamma^{\mu}J_{\mu}(\phi)|\lesssim |(D_{L}+\frac{1}{r})\phi||\phi|+|(D_{\Lbar}-\frac{1}{r})\phi||\phi|+|\phi||\sD\phi|,
\end{align*}
which has been already considered.}
{\subsection*{$\EE_{3}^{1}(\Gamma)$}
We have
\begin{align}\label{KbarF3}
\overline{K}^{\nu}_{0}F_{\nu\mu}=\frac{1}{2}\taup^{2}F_{L\mu}+\frac{1}{2}\taum^{2}F_{\Lbar\mu}.
\end{align}
For any commutator $\Gamma$ we see that
\begin{align*}
|\overline{K}_{0}^{\nu}F_{\nu\mu}J^{\mu}(\LL_{\Gamma}\phi)|\lesssim \Big(\taup^{2}|(D_{L}+\frac{1}{r})\LL_{\Gamma}\phi|+\taum^{2}|(D_{\Lbar}-\frac{1}{r})\LL_{\Gamma}\phi|\Big)|\rho||\LL_{\Gamma}\phi|\\
+\big(\taup^{2}|\alpha|+\taum^{2}|\alphabar|\big)|\sD\LL_{\Gamma}\phi||\LL_{\Gamma}\phi|.
\end{align*}
Bounding the terms involving $F$ in $L^{\infty}$, we have the following estimates
\begin{align*}
\|\taup^{2}|(D_{L}+\frac{1}{r})\LL_{\Gamma}\phi||\LL_{\Gamma}\phi||\rho|\|_{L^{1}(C_{u})}&\lesssim \|\taup(D_{L}+\frac{1}{r})\LL_{\Gamma}\phi\|_{L^{2}(C_{u})}\|\frac{\taum}{r}\LL_{\Gamma}\phi\|_{L^{2}(C_{u})}\\
&\cdot\|\taup^{2}\taum^{-1}\rho\|_{L^{\infty}}\lesssim{\taum^{-3/2}},
\end{align*}
\begin{align*}
&\|\taum^{2}|(D_{\Lbar}-\frac{1}{r})\LL_{\Gamma}\phi||\LL_{\Gamma}\phi||\rho|\|_{L^{1}(C_{u})}\lesssim\|\taum(D_{\Lbar}-\frac{1}{r})\LL_{\Gamma}\phi\|_{L^{2}(\Cbar_{\ubar})}\|\frac{\taup}{r}\LL_{\Gamma}\phi\|_{L^{2}(\Cbar_{\ubar})}\\
&\quad\quad\quad\quad\quad\cdot\|\taum\rho\|_{L^{\infty}}\lesssim\taup^{-3/2},\\
&\|\taup^{2}|\alpha||\LL_{\Gamma}\phi||\sD\LL_{\Gamma}\phi|\|_{L^{1}(\Cbar_{\ubar})}\lesssim\|\taup|\sD\LL_{\Gamma}\phi|\|_{L^{2}(\Cbar_{\ubar})}\|\frac{\taup}{r}\LL_{\Gamma}\phi\|_{L^{2}(\Cbar_{\ubar})}\|\taup|\alpha|\|_{L^{\infty}}\lesssim \taup^{-3/2},\\
&\|\taum^{2}|\alphabar||\LL_{\Gamma}\phi||\sD\LL_{\Gamma}\phi|\|_{L^{1}(C_{u})}\lesssim \|\taum|\sD\LL_{\Gamma}\phi|\|_{L^{2}(C_{u})}\|\frac{\taum}{r}\LL_{\Gamma}\phi\|_{L^{2}(C_{u})}\|\taup|\alphabar|\|_{L^{\infty}}\lesssim\taum^{-3/2}.
\end{align*}
}
\subsubsection*{$\EE_4^1(\Gamma)$} Suppose {$\Gamma$ is in $\mathbb{L}$ with conformal factor $\Omega_\Gamma.$ Then
\Align{\label{LGamma J}
\LL_\Gamma J^\mu(\phi)&=\Gamma^\alpha\nabla_\alpha\Im(\con{\phi}D^\mu\phi)+\nabla^\mu\Gamma^\alpha\Im(\con{\phi}D_\alpha\phi)+\Omega_\Gamma\Im(\phibar D^\mu\phi)\\
&=\Im(\overline{D_\Gamma\phi}D^\mu\phi+\phibar D_\Gamma D^\mu\phi)+\nabla^\mu\Gamma^\alpha\Im(\con{\phi}D_\alpha\phi)+\Omega_\Gamma\Im(\phibar D^\mu\phi)\\
&=\Im(\con{\LL_\Gamma\phi}D^\mu\phi)+\Im(\phibar\LL_\Gamma D^\mu\phi)\\
&=\Im(\con{\LL_\Gamma\phi}D^\mu\phi)+\Im(\phibar D^\mu\LL_\Gamma\phi)+\Gamma^\alpha F^\mu_{~\alpha}|\phi|^2,
}
where we have used (\ref{D-LL commutator}) to pass to the last equality. It follows that
\Align{\label{LGamma J bound}
|\LL_{\Gamma}J^\mu(\phi)|\lesssim |\Im(\overline{\LL_{\Gamma}\phi}D^\mu\phi)
+\Im(\phibar D^\mu\LL_\Gamma\phi)|+|\phi|^2
|\Gamma^\nu F_{\nu}^{~\mu}|.
}
We also have
\Align{\label{Kbar F}
\Kbar^\nu (\LL_{S}F)_{\nu\mu} = \frac{1}{2}\taup^2 (\LL_{S}F)_{L\mu}+\frac{1}{2}\taum^2(\LL_{S}F)_{\Lbar\mu}.
}
Now if $\Gamma=S$
\begin{align*}
|\Kbar^\mu (\LL_{S}F)_{\mu\nu}\LL_SJ^\nu(\phi)|\lesssim&\left(\taup^{2}|\alpha(\LL_{S}F)|+\taum^{2}|\alphabar(\LL_{S}F)|\right)\left(|\sD\phi||\LL_{S}\phi|+|\sD\LL_{S}\phi||\phi|+|\phi|^{2}(\taup|\alpha|+\taum|\alphabar|)\right)\\
&+\taum^{2}\rho(\LL_{S}F)\left(|(D_{\Lbar}-\frac{1}{r})\phi||\LL_{S}\phi|+|(D_{\Lbar}-\frac{1}{r})\LL_{S}\phi||\phi|+\taup|\rho||\phi|^{2}\right)\\
&+\taup^{2}\rho(\LL_{S}F)\left(|(D_{L}+\frac{1}{r})\phi||\LL_{S}\phi|+|(D_{L}+\frac{1}{r})\LL_{S}\phi||\phi|+\taum|\rho||\phi|^{2}\right).
\end{align*}
The terms involving $(D_{L}+\frac{1}{r})\phi$ and $(D_{\Lbar}-\frac{1}{r})\phi$ in the last two lines come from
\begin{align*}
\textrm{Im}(\overline{\LL_{S}\phi}D_{L}\phi)+\textrm{Im}(\overline{\phi}D_{L}\LL_{S}\phi)=\textrm{Im}(\overline{\LL_{S}\phi}(D_{L}+\frac{1}{r})\phi)+\textrm{Im}(\overline{\phi}(D_{L}+\frac{1}{r})\LL_{S}\phi),\\
\textrm{Im}(\overline{\LL_{S}\phi}D_{\Lbar}\phi)+\textrm{Im}(\overline{\phi}D_{\Lbar}\LL_{S}\phi)=\textrm{Im}(\overline{\LL_{S}\phi}(D_{\Lbar}-\frac{1}{r})\phi)+\textrm{Im}(\overline{\phi}(D_{\Lbar}-\frac{1}{r})\LL_{S}\phi).
\end{align*}
The terms involving $|\phi|^{2}$ can be bounded as follows.
\Align{\label{F errors 1}
&\|\taup^3|\phi|^2|\alpha||\alpha(\LL_{S}F)|\|_{L^1(C_u)}\leq\|\taup|\alpha|\|_{L^2(C_u)}\|\taup|\alpha(\LL_{S}F)|\|_{L^2(C_u)}\|\taup|\phi|^2\|_{L^\infty(C_u)}\lesssim\taum^{-2},\\
&\|\taum^3|\phi|^2|\alphabar||\alphabar(\LL_{S}F)|\|_{L^1({\Cbar}_{\ubar})}\leq\|\taum|\alphabar|\|_{L^2({\Cbar}_{\ubar})}\|\taum|\alphabar(\LL_{S}F)|\|_{L^2({\Cbar}_{\ubar})}\|\taum|\phi|^2\|_{L^\infty({\Cbar}_{\ubar})}\lesssim\taup^{-2},\\
&\|\taup^2\taum|\phi|^2|\alpha(\LL_{S}F)||\alphabar|\|_{L^1(C_u)}\leq\|\taup|\alpha(\LL_{S}F)|\|_{L^2(C_u)}\|\frac{\taum}{r}|\phi|\|_{L^2(C_u)}\|\taup|\phi|\|_{L^\infty}\|\taup|\alphabar|\|_{L^\infty}\lesssim\taum^{-2},\\
&\|\taup\taum^2|\phi|^2|\alphabar(\LL_{S}F)||\alpha|\|_{L^1(\Cbar_{\ubar})}\leq\|\taum|\alphabar(\LL_{S}F)|\|_{L^2(\Cbar_{\ubar})}\|\frac{\taup}{r}|\phi|\|_{L^2(\Cbar_{\ubar})}\|\taum|\phi|\|_{L^\infty}\|\taup|\alpha|\|_{L^\infty}\lesssim\taup^{-2},\\
&\|\taup^2\taum|\phi|^2|\rho(\LL_{S}F)||\rho|\|_{L^1({\Cbar}_{\ubar})}\leq\|\taup|\rho(\LL_{S}F)|\|_{L^2({\Cbar}_{\ubar})}\|\taup|\rho|\|_{L^2({\Cbar}_{\ubar})}\|\taum|\phi|^2\|_{L^\infty}\lesssim\taup^{-2},\\
&\|\taup\taum^2|\phi|^2|\rho(\LL_{S}F)||\rho|\|_{L^1({\Cbar}_{\ubar})}\leq\|\taum|\rho(\LL_{S}F)|\|_{L^2({C}_{u})}\|\taum|\rho|\|_{L^2({C}_{u})}\|\taup|\phi|^2\|_{L^\infty}\lesssim\taum^{-2}.
}
Next we investigate the terms involving $|\sD\phi||\LL_{S}\phi|$ and $|\phi||\sD\LL_{S}\phi|$, which can be bounded as
\Align{\label{F errors 2}
&\|\taup^{2}|\alpha(\LL_{S}F)||\LL_{S}\phi||\sD\phi|\|_{L^{1}(C_{u})}\lesssim\|\taup|\alpha(\LL_{S}F)|\|_{L^{2}(C_{u})}\|\frac{\taum}{r}|\LL_{S}\phi|\|_{L^{2}(C_{u})}\|\taup^{2}\taum^{-1}\sD\phi\|_{L^{\infty}}\lesssim\taum^{-3/2},\\
&\||\taup^{2}|\alpha(\LL_{S}F)||\phi||\sD\LL_{S}\phi|\|_{L^{1}(C_{u})}\lesssim\|\taup|\alpha(\LL_{S}F)|\|_{L^{2}(C_{u})}\|\taum|\sD\phi|\|_{L^{2}(C_{u})}\|\taup\taum^{-1}|\phi|\|_{L^{\infty}}\lesssim\taum^{-3/2},\\
&\|\taum^{2}|\alphabar(\LL_{S}F)||\LL_{S}\phi||\sD\phi|\|_{L^{1}(\Cbar_{\ubar})}\lesssim\|\taum|\alphabar|\|_{L^{2}(\Cbar_{\ubar})}\|\frac{\taup}{r}|\LL_{S}\phi|\|_{L^{2}(C_{u})}\|\taum|\sD\phi|\|_{L^{\infty}}\lesssim\taup^{-3/2},\\
&\|\taum^{2}|\alphabar(\LL_{S}F)||\phi||\sD\LL_{S}\phi|\|_{L^{1}(\Cbar_{\ubar})}\lesssim\|\taum|\alphabar|\|_{L^{2}(\Cbar_{\ubar})}\|\taup|\sD\LL_{S}\phi|\|_{L^{2}(\Cbar_{\ubar})}\|\taup^{-1}\taum|\phi|\|_{L^{\infty}}\lesssim\taup^{-3/2}.
}
For the remaining terms we have
\Align{\label{F errors 3}
&\|\taum^{2}\rho(\LL_{S}F)|(D_{\Lbar}-\frac{1}{r})\phi||\LL_{S}\phi|\|_{L^{1}(C_{u})}\lesssim\|\taum\rho(\LL_{S}F)\|_{L^{2}(C_{u})}\|\frac{\taum}{r}|\LL_{S}\phi|\|_{L^{2}(C_{u})}\\
&\quad\quad\quad\quad\quad\quad\quad\quad\quad\quad\quad\quad\quad\quad\quad\quad\quad\quad\cdot\|\taup(D_{\Lbar}-\frac{1}{r})\phi\|_{L^{\infty}}\lesssim\taum^{-3/2},\\
&\|\taum^{2}\rho(\LL_{S}F)|(D_{\Lbar}-\frac{1}{r})\LL_{S}\phi||\phi|\|_{L^{1}(\Cbar_{\ubar})}\lesssim\|\taup\rho(\LL_{S}F)\|_{L^{2}(\Cbar_{\ubar})}\|\taum(D_{\Lbar}-\frac{1}{r})\LL_{S}\phi\|_{L^{2}(\Cbar_{\ubar})}\\
&\quad\quad\quad\quad\quad\quad\quad\quad\quad\quad\quad\quad\quad\quad\quad\quad\quad\quad\cdot\|\taup^{-1}\taum|\phi|\|_{L^{\infty}}\lesssim\taup^{-3/2},\\
&\|\taup^{2}|\rho(\LL_{S}F)||\LL_{S}\phi||(D_{L}+\frac{1}{r})\phi|\|_{L^{1}(\Cbar_{\ubar})}\lesssim\|\taup\rho(\LL_{S}F)\|_{L^{2}(\Cbar_{\ubar})}\|\frac{\taup}{r}|\LL_{S}\phi|\|_{L^{2}(\Cbar_{\ubar})}\\
&\quad\quad\quad\quad\quad\quad\quad\quad\quad\quad\quad\quad\quad\quad\quad\quad\quad\quad\cdot\|\taup(D_{L}+\frac{1}{r})\phi\|_{L^{\infty}}\lesssim\taup^{-3/2},\\
&\|\taup^{2}|\rho(\LL_{S}F)||\phi||(D_{L}+\frac{1}{r})\LL_{S}\phi|\|_{L^{1}(C_{u})}\lesssim\|\taum\rho(\LL_{S}F)\|_{L^{2}(C_{u})}\|\taup(D_{L}+\frac{1}{r})\LL_{S}\phi\|_{L^{2}(C_{u})}\\
&\quad\quad\quad\quad\quad\quad\quad\quad\quad\quad\quad\quad\quad\quad\quad\quad\quad\quad\cdot\|\taup\taum^{-1}\phi\|_{L^{\infty}}\lesssim\taum^{-3/2}.
}
Next if  $\Gamma=\Omega_{ij}$,
\begin{align*}
|\overline{K}_{0}^{\mu}(\LL_{\Omega_{ij}}F)_{\mu\nu}\LL_{\Omega_{ij}}J^{\mu}(\phi)|\lesssim&\left(\taup^{2}|\alpha(\LL_{\Omega_{ij}}F)|+\taum^{2}|\alphabar(\LL_{\Omega_{ij}}F)|\right)\left(|\sD\phi||\LL_{\Omega_{ij}}\phi|+|\sD\LL_{\Omega_{ij}}\phi||\phi|+\taup|\sigma||\phi|^{2}\right)\\
&+\taum^{2}\rho(\LL_{\Omega_{ij}}F)\left(|(D_{\Lbar}-\frac{1}{r})\phi||\LL_{\Omega_{ij}}\phi|+|(D_{\Lbar}-\frac{1}{r})\LL_{\Omega_{ij}}\phi||\phi|+\taup|\alphabar||\phi|^{2}\right)\\
&+\taup^{2}\rho(\LL_{\Omega_{ij}}F)\left(|(D_{L}+\frac{1}{r})\phi||\LL_{\Omega_{ij}}\phi|+|(D_{L}+\frac{1}{r})\LL_{\Omega_{ij}}\phi||\phi|+\taup|\alpha||\phi|^{2}\right).
\end{align*}
Here we only need to consider the terms involving $|\phi|^{2}$, which are different from the case when $\Gamma=S$.
\Align{\label{F errors 4}
&\|\taup^{3}|\alpha(\LL_{\Omega_{ij}}F)||\sigma||\phi|^{2}\|_{L^{1}(C_{u})}\lesssim\|\taup|\alpha(\LL_{\Omega_{ij}}F)|\|_{L^{2}(C_{u})}\|\taum|\sigma|\|_{L^{2}(C_{u})}\|\taup^{2}\taum^{-1}|\phi|^{2}\|_{L^{\infty}}\lesssim\taum^{-2},\\
&\|\taum^{2}\taup|\alphabar(\LL_{\Omega_{ij}}F)||\sigma||\phi|^{2}\|_{L^{1}(\Cbar_{\ubar})}\lesssim\|\taum|\alphabar(\LL_{\Omega_{ij}}F)|\|_{L^{2}(\Cbar_{\ubar})}\|\taup|\sigma|\|_{L^{2}(\Cbar_{\ubar})}\|\taum|\phi|^{2}\|_{L^{\infty}}\lesssim\taup^{-2},\\
&\|\taum^{2}\taup\rho(\LL_{\Omega_{ij}}F)|\alphabar||\phi|^{2}\|_{L^{1}(\Cbar_{\ubar})}\lesssim\|\taup\rho(\LL_{\Omega_{ij}}F)\|_{L^{2}(\Cbar_{\ubar})}\|\taum|\alphabar|\|_{L^{2}(\Cbar_{\ubar})}\|\taum|\phi|^{2}\|_{L^{\infty}}\lesssim\taup^{-2},\\
&\|\taup^{3}\rho(\LL_{\Omega_{ij}}F)|\alpha||\phi|^{2}\|_{L^{1}(C_{u})}\lesssim\|\taum\rho(\LL_{\Omega_{ij}}F)\|_{L^{2}(C_{u})}\|\taup|\alpha|\|_{L^{2}(C_{u})}\|\taup^2\taum^{-1}|\phi|^{2}\|_{L^{\infty}}\lesssim\taum^{-2}.
}
For $\Gamma=\Omega_{0i}$, we again only need to consider the contribution from the commutator terms. As in the previous cases, since $\Omega_{0i}=\dfrac{\omega_{i}}{2}(\ubar L-u\Lbar)+t\omega_{i}^{A}e_{A},$ the $L$ and $\Lbar$-components of $\Omega_{0i}$ have the same behavior as those of $S$, while the $e_{A}$-components have the same behavior as those of $\Omega_{ij}.$ The estimates for the contribution from $\Omega_{0i}$ therefore follow from those of $\Omega_{ij}$ and $S$. Similarly by \eqref{vf decompositions}, all the contributions from $\Gamma=\partial_\mu$ have already been considered.}
\subsection*{Step 3: {$Q_k^{*}(T)\leq\frac{C\epsilon}{2}$}} {Appealing to the flux estimates on $C_{-1}$ and Corollary \ref{div theorem} we need to bound}
\Aligns{
\int_{V_T} \left( \left| \frac{1}{\Omega}D_{\Kbar}(\Omega \LL^k_{X_1\dots X_k} \phi)  D^\alpha D_\alpha \LL^k_{X_1\dots X_k} \phi \right|+ \left| \Kbar^\nu F_{\mu\nu} J^\mu(\LL^k_{X_1\dots X_k} \phi) \right| + \left| \Kbar^\nu \LL^k_{X_1\dots X_k} F_{\mu\nu} \LL^k_{X_1\dots X_k} J^\mu(\phi) \right| \right).
}
In view of Lemma \ref{kth commutator} this can be bounded by terms of the following forms:
\Aligns{
& A_1: \int_{V_T} \left|\frac{1}{\Omega}  D_{\Kbar}(\Omega \LL^k_{X_1\dots X_k} \phi) \right| \left| 2L^{|I_1|}_{X_{I_1}}X_{i_1}^\alpha \LL^{|I_2|}_{X_{I_2}}F_{\mu\alpha}D^\mu \LL^{|I_3|}_{X_{I_3}}\phi+ \nabla^\mu L^{|I_1|}_{X_{I_1}}X_{i_1}^\alpha \LL^{|I_2|}_{X_{I_2}}F_{\mu\alpha}\LL^{|I_3|}_{X_{I_3}}\phi\right|, \\&\hspace{10cm}|I_1|+|I_2|+|I_3|=k-1,~ i_1\notin I_1,\\
& A_2: \int_{V_T} \left| \frac{1}{\Omega} D_{\Kbar}(\Omega \LL^k_{X_1\dots X_k} \phi) \right| \left| L^{|I_1|}_{X_{I_1}}X_{i_1}^\alpha \nabla^\mu \LL^{|I_2|}_{X_{I_2}}F_{\mu\alpha}\LL^{|I_3|}_{X_{I_3}}\phi \right|,\quad|I_1|+|I_2|+|I_3|=k-1,\\
& B: \int_{V_T} \left|  \LL_{\Kbar} \LL^k_{X_1\dots X_k} \phi \right| \left| L_{X_{I_1}}^{|I_1|}X_{i_1}^\alpha L_{X_{I_2}}^{|I_2|}X_{i_2}^\beta\LL_{X_{I_3}}^{|I_3|}F_{\mu\alpha}\LL_{X_{I_4}}^{|I_4|}F^\mu_{~\beta}\LL_{X_{I_5}}^{|I_5|}\phi \right|,\\&\hspace{10cm}|I_1|+\dots+|I_5|=k-2,~i_j\notin I_j,\\
& C: \int_{V_T} \left| \Kbar^\nu F_{\mu\nu}  J^\mu(\LL^k_{X_1\dots X_k}\phi) \right|,\\
&D: \int_{V_T} \left| \Kbar^\nu \LL^k_{X_1\dots X_k} F_{\mu\nu} \LL^k_{X_1\dots X_k} J^\mu(\phi) \right|.
}
\subsubsection*{{$A_1$}} {We consider} the cases in (\ref{one}.1-\ref{one}.14) where the terms that were previously bounded in $L^\infty$ are differentiated so many times that they can no longer be placed in $L^\infty.$ For simplicity of notation we drop the Lie derivatives from the notations.
\subsubsection*{(\ref{one}.1)}
\Aligns{
\int_{V_T}&\taup^3|(D_L+\frac{1}{r})\phithree||\rhotwo||D_L\phi|\\
&\leq\int_{-1}^\infty\|\taup(D_L+\frac{1}{r})\phithree\|_{L^2(C_u)}\|\taum\rhotwo\|_{L^2(C_u)}\|\taum^{-1}\taup^2 D_L\phi\|_{L^\infty(C_u)}du\lesssim \epsilon^3.
}
\subsubsection*{(\ref{one}.2)}
\Aligns{
\int_{V_T}&\taup^3|(D_L+\frac{1}{r})\phithree||\alphatwo||\sD\phi|\\
&\leq\int_{-1}^\infty \|\taup(D_L+\frac{1}{r})\phithree\|_{L^2(C_u)}\|\taup\alphatwo\|_{L^2(C_u)}\|\taup |\sD\phi|\|_{L^\infty(C_u)}du\lesssim\epsilon^3.
}
\subsubsection*{(\ref{one}.3)} Fix $\delta\in(0,1).$ Then
\Aligns{
\int_{V_T}&\taup^2\taum|(D_L+\frac{1}{r})\phithree||\alphabartwo||\sD\phi|\lesssim\epsilon \int_{V_T}\taum^{\frac{1}{2}}|(D_L+\frac{1}{r})\phithree||\alphabartwo|\\
&\leq\epsilon \left( \int_{V_T} \taum^{-1-\delta} \taup^2 | (D_L+\frac{1}{r}) \phithree |^2 \right)^{\frac{1}{2}} \left( \int_{V_T}\taup^{-2+\delta}\taum^2|\alphabartwo|^2\right)^{\frac{1}{2}}\\
&=\epsilon \left( \int_{-1}^\infty \taum^{-1-\delta} \left(\int_{C_u}\taup^2 | (D_L+\frac{1}{r}) \phithree |^2\right)du \right)^{\frac{1}{2}} \left( \int_{0}^\infty\taup^{-2+\delta}\left(\int_{{\Cbar}_{\ubar}}\taum^2|\alphabartwo|^2\right)d\ubar\right)^{\frac{1}{2}}\\
&\lesssim\epsilon^3.
}
\subsubsection*{(\ref{one}.4)} Write $D_{\Lbar}\phitwo=(D_{\Lbar}-\frac{1}{r})\phitwo+\frac{\phitwo}{r}.$ For the first term we have
\Aligns{
&\int_{V_T}\taup^2\taum|(D_L+\frac{1}{r})\phithree||(D_{\Lbar}-\frac{1}{r})\phitwo||\rho|\lesssim\epsilon \int_{V_T}\taum^{\frac{1}{2}}|(D_L+\frac{1}{r})\phithree||(D_{\Lbar}-\frac{1}{r})\phitwo|\\
&~\leq\epsilon \left( \int_{-1}^\infty \taum^{-1-\delta} \left(\int_{C_u}\taup^2 | (D_L+\frac{1}{r}) \phithree |^2\right)du \right)^{\frac{1}{2}} \left( \int_{0}^\infty\taup^{-2+\delta}\left(\int_{{\Cbar}_{\ubar}}\taum^2|(D_{\Lbar}-\frac{1}{r})\phitwo|^2\right)d\ubar\right)^{\frac{1}{2}}\\
&~\lesssim\epsilon^3.
}
The other term {can be} bounded as
\Aligns{
{\int_{V_T}\taup^2\taum|(D_L+\frac{1}{r})\phithree||\frac{\phi}{r}||\rho|\leq\int_{-1}^\infty\|\taup(D_L+\frac{1}{r})\phithree\|_{L^2(C_u)}\|\frac{\taum}{r}|\phitwo|\|_{L^2(C_u)}\|\taup|\rho|\|_{L^\infty}du\lesssim\epsilon^3.}
}
{\subsubsection*{(\ref{one}.5)} Writing $D_{L}\phitwo=(D_{L}+\dfrac{1}{r})\phitwo-\dfrac{\phitwo}{r}$, we have:
\begin{align*}
\int_{V_{T}}\taum^{2}\taup|(D_{\Lbar}-\frac{1}{r})\LL^{k}\phi||\rho||(D_{L}+\frac{1}{r})\phi|\lesssim\int_{V_{T}}\taum^{1/2}|(D_{\Lbar}-\frac{1}{r})\LL^{k}\phi||(D_{L}+\frac{1}{r})\phi|
\end{align*}
which can be bounded as in the previous case. For the other {term we} have
\begin{align*}
\int_{V_{T}}\taum^{2}\taup|(D_{\Lbar}-\frac{1}{r})\LL^{k}\phi||\rho||\frac{1}{r}\phi|&\lesssim\int_{0}^{\infty}\taum\|\rho\|_{L^{\infty}}\|\taum(D_{\Lbar}-\frac{1}{r})\LL^{k}\phi\|_{L^{2}(\Cbar_{\ubar})}\|\frac{\taup}{r}\phi\|_{L^{2}(\Cbar_{\ubar})}d\ubar{\lesssim\epsilon^3.}
\end{align*}
}
\subsubsection*{(\ref{one}.6)} 
\Aligns{
&\int_{V_T}\taum^2\taup|(D_{\Lbar}-\frac{1}{r})\phithree||\alphatwo||\sD\phi|\lesssim\epsilon \int_{V_T}\taum^{\frac{1}{2}}|(D_{\Lbar}-\frac{1}{r})\phithree||\alphatwo|\\
&\quad\leq\epsilon \left( \int_{-1}^\infty \taum^{-1-\delta} \left(\int_{C_u}\taup^2 | \alphatwo|^2\right)du \right)^{\frac{1}{2}} \left( \int_{0}^\infty\taup^{-2+\delta}\left(\int_{{\Cbar}_{\ubar}}\taum^2|(D_{\Lbar}-\frac{1}{r})\phithree|^2\right)d\ubar\right)^{\frac{1}{2}}{\lesssim\epsilon^3.}
}
\subsubsection*{(\ref{one}.7)} 
{\Aligns{
\int_{V_T}\taum^3|(D_{\Lbar}-\frac{1}{r})\phithree||\alphabartwo||\sD\phi|\leq\int_{0}^\infty \|\taum(D_{\Lbar}-\frac{1}{r})\phithree\|_{L^2({\Cbar}_{\ubar})}\|\taum\alphabartwo\|_{L^2({\Cbar}_{\ubar})}\|\taum |\sD\phi|\|_{L^\infty}d\ubar\lesssim\epsilon^3.
}}
\subsubsection*{(\ref{one}.8)} 
\Aligns{
&\int_{V_T}\taum^3|(D_{\Lbar}-\frac{1}{r})\phithree||\rhotwo||D_{\Lbar}\phi|\\
&\quad\leq\int_{0}^\infty \|\taum(D_{\Lbar}-\frac{1}{r})\phithree\|_{L^2({\Cbar}_{\ubar})}\left(\|\frac{\taup}{r}|\phi|\|_{L^2({\Cbar}_{\ubar})}+\|\taum|(D_{\Lbar}-\frac{1}{r})\phi|^2\|_{L^2({\Cbar}_{\ubar})}\right)\|\taum|\rho|\|_{L^\infty({\Cbar}_{\ubar})}d\ubar\\
&\quad\lesssim\epsilon^3.
}
\subsubsection*{(\ref{one}.9)} Here we will use $L^4-L^4$ Sobolev estimates, since we cannot bound the factor involving $\big(D_{\Lbar}-\dfrac{1}{r}\big)\phi$ in $L^{\infty}$ directly. We consider the two regions $V_{T}\bigcap\{r\leq1+\frac{t}{2}\}$ and $V_{T}\bigcap\{r\geq 1+\frac{t}{2}\}$ separately. In the first region, $\taup$ behaves like $\taum$. Since $\|\alpha\|_{L^{\infty}}\lesssim\taup^{-5/2}$, we have
\begin{align*}
&\int_{V_{T}\bigcap\{r\leq 1+\frac{t}{2}\}}\taup^{3}|(D_{L}+\frac{1}{r})\LL^{k}\phi||\alpha||(D_{\Lbar}-\frac{1}{r})\phi|\\
&\quad\lesssim\int_{0}^{\infty}(1+t)^{-3/2}\|\taup(D_{L}+\frac{1}{r})\LL^{k}\phi\|_{L^{2}(\Sigma_{t}\bigcap\{r\leq 1+\frac{t}{2}\})}\|\taum(D_{\Lbar}-\frac{1}{r})\phi\|_{L^{2}(\Sigma_{t}\bigcap\{r\leq 1+\frac{t}{2}\})}dt\\
&\quad\lesssim\int_{0}^{\infty}(1+t)^{-3/2}\|\taup(D_{L}+\frac{1}{r})\LL^{k}\phi\|_{L^{2}(\Sigma_{t})}\|\taum(D_{\Lbar}-\frac{1}{r})\phi\|_{L^{2}(\Sigma_{t})}dt.
\end{align*}
The second region is more delicate, because we need to use the $L^{4}$-Sobolev estimates. We denote by $r_m(u)$ and $r_M(u)$ the minimum and maximum radii on $C_u(T)$ respectively. Then
\begin{align*}
\int_{V_{T}\bigcap\{r\geq 1+\frac{t}{2}\}}\taup^{3}&|(D_{L}+\frac{1}{r})\LL^{k}\phi||\alpha||(D_{\Lbar}-\frac{1}{r})\phi|\lesssim\int_{V_{T}}r^{3}|(D_{L}+\frac{1}{r})\LL^{k}\phi||\alpha||(D_{\Lbar}-\frac{1}{r})\phi|\\
&\lesssim\int_{-1}^{\infty}\Big(\int_{C_{u}}r^{2}|(D_{L}+\frac{1}{r})\LL^{k}\phi|^{2}\Big)^{1/2}\\
&\quad\quad\quad\quad\cdot
\Big[\int_{r_{m}(u)}^{r_{M}(u)}\Big(\int_{\widetilde{S}_{u,r}}r^{2}\taum^6|(D_{\Lbar}-\frac{1}{r})\phi|^{4}\Big)^{1/2}\Big(\int_{\widetilde{S}_{u,r}}r^{6}|\alpha|^{4}\Big)^{1/2}dr\Big]^{1/2}\taum^{-3/2}du.
\end{align*}
The rest of the argument is the same as (\ref{first L4}) above.
\subsubsection*{(\ref{one}.10)} This is similar to the previous term, but in our application of $L^4$ estimates we place $\alphabar$ where we previously had $(D_{\Lbar}-\frac{1}{r})\phi$ and $(D_L+\frac{1}{r})\phi$ where we previously had $\alpha.$
\subsubsection*{(\ref{one}.11)}
{\Aligns{
\int_{V_T}&\taup^3|(D_L+\frac{1}{r})\phithree||\sigmatwo|(|\sD\phi|+\frac{|\phi|}{r})\\
&\leq\int_{-1}^\infty\|\taup|(D_L+\frac{1}{r})\phithree|\|_{L^2(C_u)}\|\taum\sigmatwo\|_{L^2(C_u)}\|\taum^{-1}\taup^2(|\sD\phi|+\frac{|\phi|}{r})\|_{L^\infty(C_u)}\lesssim\epsilon^3.
}}
\subsubsection*{(\ref{one}.12)}
\Aligns{
\int_{V_T}&\taum^2\taup|(D_{\Lbar}-\frac{1}{r})\phithree||\alphatwo||(D_{\Lbar}-\frac{1}{r})\phi|\lesssim\epsilon \int_{V_T}\taum^{\frac{1}{2}}|(D_{\Lbar}-\frac{1}{r})\phithree||\alphatwo|\\
&\leq\epsilon \left( \int_{-1}^\infty \taum^{-1-\delta} \left(\int_{C_u}\taup^2 | \alphatwo|^2\right)du \right)^{\frac{1}{2}} \left( \int_{0}^\infty\taup^{-2+\delta}\left(\int_{{\Cbar}_{\ubar}}\taum^2|(D_{\Lbar}-\frac{1}{r})\phithree|^2\right)d\ubar\right)^{\frac{1}{2}}\lesssim\epsilon^3.
}
\subsubsection*{(\ref{one}.13)} Noting that $|\alphabar|\lesssim\epsilon\taum^{-1}\taup^{-3/2}$ we find ourselves in the same situation as (\ref{one}.4) above.
\subsubsection*{(\ref{one}.14)}
\Aligns{
\int_{V_T}\taum^2\taup|(D_{\Lbar}&-\frac{1}{r})\phithree||\sigmatwo|(|\sD\phi|+\frac{|\phi|}{r})\\
&\leq\int_{0}^\infty\|\taum|(D_{\Lbar}-\frac{1}{r})\phithree|\|_{L^2({\Cbar}_{\ubar})}\|\taum(|\sD\phi|+\frac{|\phi|}{r})\|_{L^\infty({\Cbar}_{\ubar})}\|\taup|\sigmatwo|\|_{L^2({\Cbar}_{\ubar})}d\ubar\lesssim\epsilon^3.
}
\subsubsection*{$A_2$} Here we commute more Lie derivatives with $J$ (recall that $\nabla$ and $\LL_X$ commute.) In (\ref{LGamma J}) above we computed the effect of commuting one derivative. Repeating that computation we get
{\Align{\label{LLJ commutator}
\LL_Y\LL_X J_\mu=&\Im(\con{\LL_Y\LL_X\phi}D_\mu\phi)+\Im(\con{\LL_X\phi}D_\mu\LL_Y\phi)+\Im(\con{\LL_Y\phi}D_\mu\LL_X\phi)\\
&+\Im(\phibar D_\mu\LL_Y\LL_X\phi)-2Y^\beta F_{\mu\beta}\Re(\phibar\LL_X\phi)+2X^\alpha F_{\mu\alpha}\Re(\phibar\LL_Y\phi)\\
&+L_YX^\alpha F_{\mu\alpha}|\phi|^2+X^\alpha\LL_YF_{\mu\alpha}|\phi|^2+\Omega_YX^\alpha F_{\mu\alpha}|\phi|^2.
}}
Note also that when $\mu=L$ or $\Lbar$ we have
\Align{\label{LLJ cancelations}
&\Im(\con{\LL_X\phi}D_L\LL_Y\phi)+\Im(\con{\LL_Y\phi}D_L\LL_X\phi)=\Im(\con{\LL_X\phi}(D_L+\frac{1}{r})\LL_Y\phi)+\Im(\con{\LL_Y\phi}(D_L+\frac{1}{r})\LL_X\phi),\\
&\Im(\con{\LL_X\phi}D_{\Lbar}\LL_Y\phi)+\Im(\con{\LL_Y\phi}D_{\Lbar}\LL_X\phi)=\Im(\con{\LL_X\phi}(D_{\Lbar}-\frac{1}{r})\LL_Y\phi)+\Im(\con{\LL_Y\phi}(D_{\Lbar}-\frac{1}{r})\LL_X\phi),\\
&\Im(\con{\LL_Y\LL_X\phi}D_L\phi)+\Im(\con{\phi}D_L\LL_Y\LL_X\phi)=\Im(\con{\LL_Y\LL_X\phi}(D_L+\frac{1}{r})\phi)+\Im(\con{\phi}(D_L+\frac{1}{r})\LL_Y\LL_X\phi),\\
&\Im(\con{\LL_Y\LL_X\phi}D_{\Lbar}\phi)+\Im(\con{\phi}D_{\Lbar}\LL_Y\LL_X\phi)=\Im(\con{\LL_Y\LL_X\phi}(D_{\Lbar}-\frac{1}{r})\phi)+\Im(\con{\phi}(D_{\Lbar}-\frac{1}{r})\LL_Y\LL_X\phi).
}
{The case where we commute more derivatives is similar: 
\Align{\label{LLJ kth commutator}
\LL_Z\LL_Y\dots\LL_X J_\mu=&\quad\Im(\con{\LL_Z\LL_Y\dots\LL_X\phi}D_\mu\phi)+\Im(\phibar D_\mu\LL_Z\LL_Y\dots\LL_X\phi)\\
&+\Im(\con{\LL_Y\dots\LL_X\phi}D_\mu\LL_Z\phi)+\Im(\con{\LL_Z\phi}D_\mu\LL_Y\dots\LL_X\phi)\\
&+\dots\\
&+\Im(\con{\LL_X\phi}D_\mu\LL_Z\LL_Y\dots\phi)+\Im(\con{\LL_Z\LL_Y\dots\phi}D_\mu\LL_X\phi)\\
&+E,
}
where the error term $E$ can be bounded as
\Align{\label{LLJ kth commutator error}
|E|\lesssim\sum_{|I_1|+|I_2|+|I_3|+|I_4|\leq k-1}|L^{I_1}X^\alpha\LL^{I_2}F_{\mu\alpha}||\LL^{I_3}\phi||\LL^{I_4}\phi|.
}
Notice that we also have cancelations analogous to (\ref{LLJ cancelations}).} Again we have to repeat the computations in estimating $\EE_2^1$ but where the terms which were previously bounded in $L^\infty$ are differentiated enough times that one cannot use $L^\infty$ estimates for them. $E$ was not considered in estimating $\EE_2^1$ so we will handle these separately at the end. Note, however, that similar terms came up in $\EE_3^1.$
\subsubsection*{(\ref{two}.1)} Here we need to consider the case where $\phi^2$ cannot be bounded in $L^\infty.$ Note, however, that we can still bound one of the $\phi$ factors in $\phi^2$ as well as $(D_L+\frac{1}{r})\phi$ in $L^\infty$ 
\begin{align*}
\int_{V_T}&\taup^3|(D_L+\frac{1}{r})\phithree||\phi||\phi||(D_{L}+\frac{1}{r})\phi|\\
&\leq\int_{-1}^\infty\|\taup|(D_L+\frac{1}{r})\LL^{k}\phi|\|_{L^2(C_u)}\|\frac{\taum}{r}\phi\|_{L^2(C_u)}\|\taup^2|(D_L+\frac{1}{r})\phi|\|_{L^{\infty}(C_u)}\|\taup\phi\|_{L^\infty(C_u)}\taum^{-1}du\lesssim\epsilon^3.
\end{align*}
\subsubsection*{(\ref{two}.2)}\begin{align*}
&\int_{V_T}\taup^2\taum|(D_L+\frac{1}{r})\phithree||\phi|^2|(D_{\Lbar}-\frac{1}{r})\phi|\lesssim\epsilon\int_{V_T}\taup|(D_L+\frac{1}{r})\phithree|\taum|(D_{\Lbar}-\frac{1}{r})\phi|\taup^{-1}\taum^{-1}\\
&\quad\leq\epsilon\left(\int_{-1}^\infty\taum^{-2}\left(\int_{C_u}\taup^2|(D_L+\frac{1}{r})\phithree|^2\right)du\right)^{1/2}\left(\int_{0}^\infty\taup^{-2}\left(\int_{{\Cbar}_{\ubar}}\taum^2|(D_{\Lbar}-\frac{1}{r})\phi|^2\right)d\ubar\right)^{1/2}\lesssim\epsilon^3.
\end{align*}
\subsubsection*{(\ref{two}.3)} Bounding $|\phi|^2$ in $L^\infty$ the computation is exactly the same as in the previous case.
{\subsubsection*{(\ref{two}.4)} Suppose only one of the $\phi$ factors can be placed in $L^\infty.$
\begin{align*}
\int_{V_T}\taum^3&|(D_{\Lbar}-\frac{1}{r})\phithree||\phi||\phi||(D_{\Lbar}-\frac{1}{r})\phi|\\
&\leq\int_{0}^\infty\|\taum|(D_{\Lbar}-\frac{1}{r})\phithree|\|_{L^2({\Cbar}_{\ubar})}\|\frac{\taup}{r}|\phi|\|_{L^2({\Cbar}_{\ubar})}\|\taum^2|\phi||(D_{\Lbar}-\frac{1}{r})\phi|\|_{L^\infty({\Cbar}_{\ubar})}d\ubar\lesssim\epsilon^3.
\end{align*}}
\subsubsection*{(\ref{two}.5)}
\begin{align*}
\int_{V_T}\taup^3&|(D_L+\frac{1}{r})\phithree||\phi||\phi||\sD\phi|\\
&\leq\int_{-1}^\infty\|\taup|(D_L+\frac{1}{r})\phithree|\|_{L^2(C_u)}\|\frac{\taum}{r}|\phi||\|_{L^2(C_u)}\|\taup^3\taum^{-1}|\phi||\sD\phi|\|_{L^\infty(C_u)}du\lesssim\epsilon^3.
\end{align*}
\subsubsection*{(\ref{two}.6)}
\begin{align*}
\int_{V_T}\taum^2\taup&|(D_{\Lbar}-\frac{1}{r})\phithree||\phi||\phi||\sD\phi|\\
&\leq\int_{0}^\infty\|\taum|(D_{\Lbar}-\frac{1}{r})\phithree|\|_{L^2({\Cbar}_{\ubar})}\|\frac{\taup}{r}|\phi|\|_{L^2({\Cbar}_{\ubar})}\|\taup\taum|\phi||\sD\phi|\|_{L^\infty({\Cbar}_{\ubar})}d\ubar\lesssim\epsilon^3.
\end{align*}
Next we consider the contribution of the last five terms in {$E.$} We have to consider integrals over $V_T$ of terms of the form
\Align{\label{A2 error 1}
Y^\mu X^\alpha F_{\mu\alpha}{|\phi|^3}|\left(|\taup^2(D_L+\frac{1}{r})\phithree|+\taum^2|(D_{\Lbar}-\frac{1}{r})\phithree|\right),
}
where {one of the factors in $|\phi|^3$} cannot be placed in $L^\infty.$ Inspecting the commutator vector-fields we see that we need to consider the cases where $|Y^\mu X^\alpha F_{\mu\alpha}|$ is one of the following:
\Align{\label{A2 error 2}
&\taup^2|\alpha|,\quad\taup^2|\sigma|,\quad\taup\taum|\alphabar|,\quad\taup\taum|\rho|.
}
To estimate these terms, it suffices {to} observe the following point-wise {bounds:}
\allowdisplaybreaks{\begin{align*}
&\taup^4|\alpha||(D_L+\frac{1}{r})\phithree|{|\phi|^3}\lesssim\epsilon\taup|(D_L+\frac{1}{r})\phithree||\frac{\taum}{r}{\phi}|\taup^{-1/2}\taum^{-2},\\
&{\taup^4|\alpha||(D_L+\frac{1}{r})\phithree|{|\phi|^3}\lesssim\epsilon\taup|(D_L+\frac{1}{r})\phithree|\taup|\alpha|\taup^{-1}\taum^{-3/2},}\\
&\taup^2\taum^{2}|\alpha||(D_{\Lbar}-\frac{1}{r})\phithree|{|\phi|^3}\lesssim\epsilon\taum|(D_{\Lbar}-\frac{1}{r})\phithree||\frac{\taup}{r}{\phi}|\taup^{-5/2},\\
&{\taup^2\taum^{2}|\alpha||(D_{\Lbar}-\frac{1}{r})\phithree|{|\phi|^3}\lesssim\epsilon|(D_{\Lbar}-\frac{1}{r})\phithree||\alpha|\taup^{-1/2}},\\
&\taup^3\taum|\alphabar||(D_{L}+\frac{1}{r})\phithree|{|\phi|^3}\lesssim\epsilon\taup|(D_{L}+\frac{1}{r})\phithree||\frac{\taum}{r}{\phi}|\taum^{-5/2},\\
&{\taup^3\taum|\alphabar||(D_{L}+\frac{1}{r})\phithree|{|\phi|^3}\lesssim\epsilon|(D_{L}+\frac{1}{r})\phithree||\alphabar|\taum^{-1/2},}\\
&\taup\taum^3|\alphabar||(D_{\Lbar}-\frac{1}{r})\phithree|{|\phi|^3}\lesssim\epsilon\taum|(D_{\Lbar}-\frac{1}{r})\phithree||\frac{\taup}{r}{\phi}|\taup^{-2}\taum^{-1/2},\\
&{\taup\taum^3|\alphabar||(D_{\Lbar}-\frac{1}{r})\phithree|{|\phi|^3}\lesssim\epsilon\taum|(D_{\Lbar}-\frac{1}{r})\phithree|\taum|\alphabar|\taup^{-2}\taum^{-1/2},}\\
&\taup^4|\sigma||(D_L+\frac{1}{r})\phithree|{|\phi|^3}\lesssim\epsilon\taup|(D_L+\frac{1}{r})\phithree||\frac{\taum}{r}{\phi}|{\taum^{-5/2}},\\
&{\taup^4|\sigma||(D_L+\frac{1}{r})\phithree|{|\phi|^3}\lesssim\epsilon\taup|(D_L+\frac{1}{r})\phithree|\taum|\sigma|{\taum^{-5/2}},}\\
&\taup^2\taum^2|\sigma||(D_{\Lbar}-\frac{1}{r})\phithree|{|\phi|^3}\lesssim\epsilon\taum|(D_{\Lbar}-\frac{1}{r})\phithree||\frac{\taup}{r}{\phi}|\taup^{-2}\taum^{-1/2},\\
&{\taup^2\taum^2|\sigma||(D_{\Lbar}-\frac{1}{r})\phithree|{|\phi|^3}\lesssim\epsilon\taum|(D_{\Lbar}-\frac{1}{r})\phithree|\taup|\sigma|\taup^{-2}\taum^{-1/2},}\\
&\taup^3\taum|\rho||(D_{L}+\frac{1}{r})\phithree|{|\phi|^3}\lesssim\epsilon\taup|(D_{L}+\frac{1}{r})\phithree||\frac{\taum}{r}{\phi}|\taup^{-1}\taum^{-3/2},\\
&{\taup^3\taum|\rho||(D_{L}+\frac{1}{r})\phithree|{|\phi|^3}\lesssim\epsilon\taup|(D_{L}+\frac{1}{r})\phithree|\taum|\rho|\taup^{-1}\taum^{-3/2},}\\
&\taup\taum^3|\rho||(D_{\Lbar}-\frac{1}{r})\phithree|{|\phi|^3}\lesssim\epsilon\taum|(D_{\Lbar}-\frac{1}{r})\phithree||\frac{\taup}{r}{\phi}|\taup^{-3}\taum^{1/2},\\
&{\taup\taum^3|\rho||(D_{\Lbar}-\frac{1}{r})\phithree|{|\phi|^3}\lesssim\epsilon\taum|(D_{\Lbar}-\frac{1}{r})\phithree|\taup|\rho|\taup^{-5/2}.}
\end{align*}}
This finishes the estimates for $A_2.$ We postpone the estimates for $B$ for now (because $B$ corresponds to the terms that did not come up when we commuted only one derivative{). The} estimates for $C$ are identically the same as those when we commute only one derivative, because we always bound the electromagnetic {field} in $L^{\infty}.$
\subsubsection*{$D$} {Here we} have to commute $k$ derivatives with $J^\mu,$ the effect of which was computed in (\ref{LLJ kth commutator}) above. {Since} in the case where we commute only one derivative, the terms to be bounded in $L^{\infty}$ and the terms to be bounded in $L^{2}$ are {symmetric, the} contribution of the first four lines {in (\ref{LLJ kth commutator})} are bounded {in the} same way as before, and we only need to consider the contribution of the error terms $E.$ Inspecting the definition of $D$ we see that we have to consider the following {terms:}
\begin{align*}
{|\phi|^2\Big(\taup^3(|\rho||\alpha|+|\alpha|^2+|\alpha||\sigma|)+\taum\taup^2(|\rho|^2+|\alpha||\alphabar|)
+\taum^2\taup(|\rho||\alphabar|+|\alphabar||\sigma|+|\alphabar||\alpha|)+\taum^3|\alphabar|^2\Big).}
\end{align*}
To estimate the {contributions of these terms it suffices} to observe the following point-wise {bounds.}
\Aligns{
\int_{V_{T}}\taup^2\taum|\phi|^{2}{|\alpha||\alphabar|}&\lesssim\epsilon\int_{V_{T}}\taup^{-1}\taum^{-1}{|\taup\alpha||\taum\alphabar|}\\
&\lesssim{\epsilon}\Big(\int_{-1}^{\infty}\taup^{-2}\|\taum|\alphabar|\|_{L^{2}(\Cbar_{\ubar})}^{2}d\ubar\Big)^{1/2}\Big(\int_{0}^{\infty}\taum^{-2}\|\taup|\alpha|\|_{L^{2}(C_{u})}^{2}du\Big)^{1/2},
}
and
\allowdisplaybreaks{\begin{align}
&\taum\taup^2|\phi|^2|\rho|^2\lesssim\epsilon|\frac{\taum}{r}\phi||\taum|\rho||\taum^{-2},\nonumber\\
&\taum\taup^{2}|\phi|^{2}|\rho||\rho|\lesssim\epsilon{|\taum\rho|^2}\taum^{-2},\nonumber\\
&\taup^3|\phi|^{2}|\rho||\alpha|\lesssim\epsilon|\taum\rho||\taup|\alpha||\taum^{-2},\nonumber\\
&\taup^{3}|\phi|^2|\rho||\alpha|\lesssim\epsilon|\taup\rho||\frac{\taup}{r}\phi|\taup^{-3/2}\taum^{-1/2},\nonumber\\
&\taup^3|\phi|^2|\alpha|^2\lesssim\epsilon|\frac{\taum}{r}\phi||\taup|\alpha||\taum^{-3/2}\taup^{-1/2},\nonumber\\
&\taup^{3}|\phi|^{2}|\alpha|^2\lesssim\epsilon|\taup\alpha|^2\taup^{-1}{\taum^{-1},}\nonumber\\
&\taup^3|\phi|^{2}\alpha||\sigma|\lesssim\epsilon|\taup\alpha||\taum\sigma|\taum^{-2},\nonumber\\
&\taup^{3}|\phi|^2|\sigma||\alpha|\lesssim\epsilon|\taup\alpha||\frac{\taum}{r}\phi|{\taum^{-2},}\nonumber\\
&\taup^{2}\taum|\phi|^2|\alphabar||\alpha|\lesssim\epsilon\taup|\alpha||\frac{\taum}{r}\phi|\taum^{-2},\label{F errors 5}\\
&\taum^2\taup|\phi|^{2}|\rho||\alphabar|\lesssim\epsilon\taup|\rho||\taum\alphabar|\taup^{-2},\nonumber\\
&\taum^{2}\taup|\phi||\phi||\alphabar||\rho|\lesssim\epsilon|\taum\rho||\frac{\taum}{r}\phi|\taum^{-2},\nonumber\\
&\taum^3|\phi|^{2}|\alphabar|^2\lesssim\epsilon|\taum\alphabar|^2\taup^{-2},\nonumber\\
&\taum^3|\phi|^2|\alphabar|^2\lesssim\epsilon|\taum\alphabar||\frac{\taup}{r}\phi|\taup^{-2},\nonumber\\
&\taum^2\taup|\phi|^{2}|\alphabar||\sigma|\lesssim\epsilon|\taum\alphabar||\taup\sigma|\taup^{-2},\nonumber\\
&\taum^{2}\taup|\phi|^2|\alphabar||\sigma|\lesssim\epsilon|\taum\alphabar||\frac{\taup}{r}\phi|\taup^{-2},\nonumber\\
&\taum^2\taup|\phi|^{2}|\alphabar||\alpha|\lesssim\epsilon|\alphabar||\alpha|,\nonumber\\
&\taum^2\taup|\phi|^2|\alphabar||\alpha|\lesssim\epsilon|\taum\alphabar||\frac{\taup}{r}\phi|{\taup^{-2}}.\nonumber
\end{align}}
\subsubsection*{$B$} {Dropping the derivatives from the notation, the structure of $X^\alpha Y^\beta F_{\mu\alpha}F^\mu_{~\beta}\phi$ was computed in (\ref{higher derivative extra term}).} Note that since the total number of derivatives here is $k-2,$ every term can be bounded in $L^\infty.$ To estimate {these contributions} in $B$ it suffices to observe the following {point-wise bounds}.
\allowdisplaybreaks{\begin{align*}
&\taup^4|(D_L+\frac{1}{r})\phithree||\phi||\alpha|(|\alpha|+|\alphabar|+|\rho|+|\sigma|)\lesssim\epsilon|\taup(D_L+\frac{1}{r})\phithree||\taup\alpha|\taum^{-2},\\
&\taup^2\taum^2|(D_{\Lbar}-\frac{1}{r})\phithree||\phi||\alpha|(|\alpha|+|\alphabar|+|\rho|+|\sigma|)\lesssim\epsilon|\taum(D_{\Lbar}-\frac{1}{r})\phithree||\frac{\taup}{r}|\phi||\taup^{-3/2}\taum^{-1/2},\\
&\taup^4|(D_L+\frac{1}{r})\phithree||\phi||\sigma|^2\lesssim\epsilon|\taup(D_L+\frac{1}{r})\phithree||\taum|\sigma||\taum^{-2},\\
&\taup^2\taum^2|(D_{\Lbar}-\frac{1}{r})\phithree||\phi||\sigma|^2\lesssim\epsilon|\taum(D_{\Lbar}-\frac{1}{r})\phi||\taup|\sigma||\taup^{-2},\\
&\taup^3\taum|(D_L+\frac{1}{r})\phithree||\phi||\rho|^2\lesssim\epsilon|\taup(D_L+\frac{1}{r})\phithree||\taum|\rho||\taup^{-1}\taum^{-1},\\
&\taup\taum^3|(D_{\Lbar}-\frac{1}{r})\phithree||\phi||\rho|^2\lesssim\epsilon|\taum(D_{\Lbar}-\frac{1}{r})\phi||\taup|\rho||\taup^{-2},\\
&\taup^3\taum|(D_L+\frac{1}{r})\phithree||\phi||\alphabar|(|\alpha|+|\rho|+|\sigma|)\lesssim\epsilon|\taup(D_L+\frac{1}{r})\phithree||\frac{\taum}{r}|\phi||\taum^{-2},\\
&\taup\taum^3|(D_{\Lbar}-\frac{1}{r})\phithree||\phi||\alphabar|(|\alpha|+|\rho|+|\sigma|)\lesssim\epsilon|\taum(D_{\Lbar}-\frac{1}{r})\phithree||\frac{\taup}{r}|\phi||\taup^{-2},\\
&\taup^2\taum^2|(D_L+\frac{1}{r})\phithree||\phi||\alphabar|^2\lesssim\epsilon|\taup(D_L+\frac{1}{r})\phithree||\frac{\taum}{r}|\phi||\taum^{-2},\\
&\taum^4|(D_{\Lbar}-\frac{1}{r})\phithree||\phi||\alphabar|^2\lesssim\epsilon|\taum(D_{\Lbar}-\frac{1}{r})\phi||\taum|\alphabar||\taup^{-2}.
\end{align*}}
\end{proof}
\section{Conclusions}
We have now completed the estimates inside and outside $V_T,$ concluding the proof of Theorem \ref{main theorem}. In the case of a non-compactly supported scalar field, our main simplification was avoiding fractional Morawetz estimates for the electromagnetic field as well as certain space-time estimates appearing in \cite{LS1}. When the scalar field is compactly supported, we were able to dispense with the fractional Morawetz estimates for the scalar field as well.
\section*{Acknowledgments} L. Bieri and S. Miao are supported by the NSF grant DMS-1253149 to The University of Michigan and S. Shahshahani is supported by the NSF grant NSF-1045119. Support of the NSF is gratefully acknowledged by the authors. The authors would like to thank Sung-Jin Oh and Jacob Sterbenz for helpful discussions.
\bibliographystyle{plain}
\bibliography{MKGbib}
\vspace{1cm}
\centerline{\scshape Lydia Bieri \& Shuang Miao \& Sohrab Shahshahani}
\medskip
{\footnotesize
 \centerline{Department of Mathematics, The University of Michigan}
\centerline{530 Church Street
Ann Arbor, MI  48109-1043, U.S.A.}
\centerline{\email{lbieri@umich.edu \& shmiao@umich.edu \& shahshah@umich.edu}}
} 
\end{document}